\documentclass[11pt]{article}
\newif\ifjournalversion
\journalversionfalse
\usepackage{mathtools}
\usepackage[T1]{fontenc}
\usepackage{amsfonts}
\usepackage{amsmath}
\usepackage{amssymb}
\usepackage{amsthm}
\usepackage{bbm}
\usepackage{bm}
\usepackage{mathrsfs}
\usepackage{xcolor}
\usepackage{comment}
\usepackage{needspace}
\usepackage[colorlinks=true, linkcolor=blue, citecolor=blue, urlcolor=blue]{hyperref}
\mathtoolsset{showonlyrefs}
\usepackage[
  backend=biber,
  style=numeric,
  sorting=nyt,
  giveninits=true,
  maxbibnames=99,
  doi=true,
  isbn=false,
  url=false,
  eprint=true
]{biblatex}

\DeclareNameAlias{author}{family-given}
\DeclareNameAlias{editor}{family-given}
\DeclareFieldFormat[article]{title}{#1}
\DeclareFieldFormat[incollection]{title}{#1}
\DeclareFieldFormat[inproceedings]{title}{#1}
\DeclareFieldFormat[book]{title}{\mkbibemph{#1}}
\DeclareFieldFormat[online,unpublished]{title}{#1} 
\renewbibmacro{in:}{}

\AtEveryBibitem{%
  \clearfield{eprintclass}%
  \clearfield{primaryclass}%
}
\providebibmacro*{eprint:arxiv}{}
\renewbibmacro*{eprint:arxiv}{%
  \printtext[eprint:arxiv]{%
    \printfield{eprinttype}%
    \setunit{\addcolon}%
    \printfield{eprint}}} 

\DeclareMathOperator*{\argmax}{argmax}

\DeclareMathOperator{\spt}{spt}
\DeclareMathOperator{\supp}{supp}
\DeclareMathOperator{\dist}{dist}
\DeclareMathOperator{\Lip}{Lip}

\DeclareMathOperator{\diam}{diam}

\newcommand{\RR}{\mathbb{R}}
\newcommand{\R}{\RR}

\newcommand{\eps}{\varepsilon}

\newcommand{\cR}{\mathcal{R}}

\newcommand{\Id}{\mathbbm{I}}

\newcommand{\1}{\mathbf{1}}
 
\newcommand{\QOT}{{\rm QOT}}

\usepackage[top=27mm,bottom=27mm,left=29mm,right=29mm]{geometry}
\newcommand{\mykill}[1]{}

\numberwithin{equation}{section}
\usepackage[capitalize, noabbrev]{cleveref}
\makeatletter
\AtBeginDocument{%
  \let\label@in@display\cref@old@label@in@display
  \let\ltx@label\cref@old@label
}
\makeatother %
\crefname{equation}{}{} %

\theoremstyle{plain}
\newtheorem{theorem}{Theorem}[section]
\newtheorem{proposition}[theorem]{Proposition}
\newtheorem{lemma}[theorem]{Lemma}
\newtheorem{corollary}[theorem]{Corollary}
\theoremstyle{definition}

\newtheorem{remark}[theorem]{Remark}

\newtheorem{assumption}[theorem]{Assumption}
\crefname{assumption}{Assumption}{Assumptions}
\Crefname{assumption}{Assumption}{Assumptions}
\theoremstyle{remark}

\AddToHook{env/proposition/begin}{\crefalias{theorem}{proposition}}
\AddToHook{env/lemma/begin}{\crefalias{theorem}{lemma}}
\AddToHook{env/corollary/begin}{\crefalias{theorem}{corollary}}
\AddToHook{env/definition/begin}{\crefalias{theorem}{definition}}
\AddToHook{env/remark/begin}{\crefalias{theorem}{remark}}
\AddToHook{env/example/begin}{\crefalias{theorem}{example}}
\AddToHook{env/assumption/begin}{\crefalias{theorem}{assumption}}
\crefname{theorem}{Theorem}{Theorems}
\crefname{proposition}{Proposition}{Propositions}
\crefname{lemma}{Lemma}{Lemmas}
\crefname{corollary}{Corollary}{Corollaries}
\crefname{definition}{Definition}{Definitions}
\crefname{remark}{Remark}{Remarks}
\crefname{example}{Example}{Examples}
\crefname{assumption}{Assumption}{Assumptions}

\renewenvironment{thebibliography}[1]{%
\begin{oldthebibliography}{#1}%
\setlength{\baselineskip}{.9em}
\linespread{1}
\small
\setlength{\parskip}{.30ex}%
\setlength{\itemsep}{.25em}%
}%
{%
\end{oldthebibliography}%
}
\usepackage{tocloft}

\newcommand{\dd}{\,\mathrm d}
\newcommand{\intr}{\operatorname{int}}
\DeclareMathOperator{\id}{id}
\newcommand{\OT}{\operatorname{OT}}

\newcommand{\HH}{\mathcal H^{d-1}}
\newcommand{\cE}{\mathcal E}

\newcommand{\BMO}{\operatorname{BMO}}
\newcommand{\tr}{\operatorname{tr}}

\hypersetup{
 pdftitle={Geometry and Convergence of Quadratically Regularized Optimal Transport I},
 pdfsubject={Quadratically regularized optimal transport, support geometry, convergence of potentials},
 pdfkeywords={optimal transport, quadratic regularization, support concentration, boundary regularity, reversible Markov kernel},
 pdfauthor={Alberto Gonzalez-Sanz, Marcel Nutz}
}
\title{\vspace{-1em}Geometry and Convergence of Quadratically Regularized Optimal Transport I}
\author{Alberto Gonz{\'a}lez-Sanz\thanks{Department of Statistics, Columbia University, ag4855@columbia.edu.} \and Marcel Nutz\thanks{Departments of Mathematics and Statistics, Columbia University, mnutz@columbia.edu.}}
\date{\today}
\begin{document}

\setlength{\abovedisplayskip}{9.4pt plus 2.5pt minus 4.5pt}
\setlength{\belowdisplayskip}{9.4pt plus 2.5pt minus 4.5pt}
\setlength{\abovedisplayshortskip}{0pt plus 2.5pt}
\setlength{\belowdisplayshortskip}{6pt plus 2.5pt minus 2.5pt}
\setlength{\abovecaptionskip}{6pt}
\setlength{\belowcaptionskip}{3pt}

\maketitle
\vspace{-1em}
\begin{abstract}
We establish sharp convergence rates for quadratically regularized optimal transport with quadratic cost in the regime of small regularization $\varepsilon$. In particular, we quantify the sparsity of the support of the regularized optimizer. For smooth marginal densities in $\R^d$, this support lies within a distance of order $\varepsilon^{1/(d+2)}$ from the Brenier graph, and every section of the support is sandwiched between balls with radii of that order. The geometry of the support is closely linked to the dual potentials. We show that the potentials satisfy uniform two-sided Hessian bounds and converge uniformly at rate $\varepsilon^{2/(d+2)}$, while their gradients converge at rate $\varepsilon^{1/(d+2)}$. The rates are sharp, and we further identify the regularity threshold where convergence breaks down.
\end{abstract}
\medskip
\setcounter{tocdepth}{2}
\tableofcontents

\section{Introduction}
\label{sec:main}

Optimal transport provides a geometric framework for comparing probability distributions, with applications in statistics, machine learning and image processing. Regularization plays a central role in these applications by facilitating computation and improving statistical sample complexity. We study quadratically regularized optimal transport (QOT), a sparse alternative to entropic regularization. Let $\Pi(\mu,\nu)$ denote the set of couplings of two probability measures $\mu,\nu$ on $\R^d$. Given a regularization parameter $\eps>0$, the QOT problem with quadratic transport cost is
\begin{equation}\label{eq:problem}
 \QOT_\eps=\min_{\pi\in\Pi(\mu,\nu)}\left\{\frac12\int|x-y|^2\dd\pi+\frac\eps2\int\left(\frac{\dd\pi}{\dd(\mu\otimes\nu)}\right)^2\dd(\mu\otimes\nu)\right\},
\end{equation}
where $|\cdot|$ is the Euclidean norm and $\mu\otimes\nu$ is the product measure. We write $\pi_\eps$ for its unique minimizer and $\OT:=\QOT_0$ for the classical (unregularized) value corresponding to $\eps=0$.

A key motivation for quadratic regularization is to preserve localization of the transport while making the optimization problem strictly convex; see \cite{Muzellec.2017.AAAI,BSR,EssidSolomon.18}. Indeed, entropic (Kullback--Leibler) regularization produces couplings with full support, even when the classical (unregularized) optimal transport is concentrated on a graph. This ``overspreading'' can be undesirable when the coupling itself, rather than only the optimal value, is of interest, as illustrated in image processing by \cite{BSR}. The QOT density has the representation (e.g., \cite{Nutz})
\begin{equation}\label{eq:slack}
 h_\eps(x,y)=\frac{\dd\pi_\eps}{\dd(\mu\otimes\nu)}(x,y)=\frac{(q_\eps(x,y))_+}{\eps},\qquad q_\eps(x,y)=x\cdot y-u_\eps(x)-v_\eps(y),
\end{equation}
where $u_\eps,v_\eps$ are the convex transformed dual potentials, and these functions are unique up to an additive constant in the setting considered below. In contrast to the exponential density of entropic regularization, the positive part in \eqref{eq:slack} allows the density to vanish, so that the support of $\pi_\eps$ can be sparse.

Despite the limited smoothness of this formulation and the lack of strong concavity of its dual, recent work has established parametric sample complexity \cite{GdBN}, finite-sample bounds that match or improve on their entropic counterparts \cite{GNSFiniteSample}, and linear convergence of several dual algorithms \cite{GNRVPL}.  %
Thus QOT combines sparsity with rigorous computational and statistical guarantees, making it an increasingly attractive alternative to entropic regularization.

A fundamental question is how sparse the coupling actually is as $\eps\downarrow0$: how large are its support sections, what is their geometry, and how closely do they follow the classical optimal transport map? This question has motivated numerous recent works including \cite{GSN,WX,NCN,GK,GGK}. Sharp results have been obtained in one dimension, in special situations such as when $\mu=\nu$, and most recently there are partial results for the interior geometry of the sections (see also \cref{se:previousResults} for a discussion of related literature). A global description for general marginal pairs, however, remained open.

We give an essentially complete (at the level of sharp scaling laws) answer to this quantitative sparsity question, for smooth marginals in arbitrary dimension $d\ge1$. Every support section is sandwiched between balls with radii comparable to $\eps^{1/(d+2)}$ and lies within distance $C\eps^{1/(d+2)}$ of the classical Brenier image. These estimates hold uniformly over all sections; in particular, the Hausdorff distance between the regularized support and the Brenier graph has exact order $\eps^{1/(d+2)}$. Matching lower bounds show that the rates cannot be improved. These results determine the sharp size and location of every support section and the rate at which the support converges to the classical transport graph.

The support analysis is closely coupled with a global regularity and convergence theory for the dual potentials. We establish two-sided Hessian bounds uniform in $\eps$ and prove that the potentials converge uniformly at the rate $\eps^{2/(d+2)}$, while their gradients converge uniformly at rate $\eps^{1/(d+2)}$. Both rates are sharp, and we identify $C^{1,1}$ as the regularity threshold at which convergence fails. Together, our results characterize the global convergence behavior of the potentials.

The analytic theory developed here is the foundation for the companion paper \cite{GNProfiles}. Using our Hessian bounds and convergence estimates for the potentials, it shows that the Brenier graph is indeed contained in the regularized support for small $\eps$; more precisely, \cref{rk:BrenierGraphInclusion} below holds. Thus our support localization amounts to concentration around, rather than merely near, the classical Brenier graph. The companion further identifies the limiting shapes and coefficients that arise when the support sections and potentials are centered at their classical counterparts and rescaled by the aforementioned rates. Overall, our results convey that quadratic regularization induces sparsity while maintaining high fidelity to classical optimal transport. Beyond these implications, our findings can also be used to upgrade several existing results in QOT, for instance the Polyak--{\L}ojasiewicz inequality and algorithmic analysis in \cite{GNRVPL}. On the other hand, forthcoming work will adapt the techniques to substantially improve the convergence and regularity theory of entropic optimal transport.

\subsection{Assumptions and main results}

The following is our main assumption on the marginals $\mu,\nu$ and their supports $X,Y$.

\begin{assumption}\label{ass:geometrydata}
Let $0<\alpha<1$. The sets $X,Y\subset\R^d$ are the closures of bounded uniformly convex domains with $C^{3,\alpha}$ boundaries. The probability measures $\mu=\rho_0\1_X\dd x$ and $\nu=\rho_1\1_Y\dd y$ have densities $\rho_0\in C^{1,\alpha}(X)$ and $\rho_1\in C^{1,\alpha}(Y)$ with uniform bounds $0<m\le\rho_i\le M<\infty$ on the supports $X$ and $Y$, respectively.
\end{assumption}

For $d=1$, the sets $X,Y$ are simply nondegenerate compact intervals. Recalling the potentials and the slack from \eqref{eq:slack}, we consider 
\begin{align}\label{eq:sections}
 S_{\eps,x}&=\{y\in Y:q_\eps(x,y)\ge0\},&
 z_{\eps,x}&=\nabla u_\eps(x),\nonumber\\
 K_{\eps,y}&=\{x\in X:q_\eps(x,y)\ge0\},&
 \zeta_{\eps,y}&=\nabla v_\eps(y).
\end{align}
The closed convex sets $S_{\eps,x}$ and $K_{\eps,y}$ are the sections of the support $\spt\pi_\eps$; see \eqref{eqv:eq:support}. The gradients $z_{\eps,x}$ and $\zeta_{\eps,y}$ can be interpreted as barycenters under marginals restricted to the corresponding sections; cf.~\eqref{eqv:eq:centers}. Here and below, gradients at boundary points of $X$ or $Y$ are defined as the continuous extensions of the interior gradients. 

We can now state our first main result on the optimal supports and potentials for small~$\eps$. We write $A\asymp B$ if $cB\le A\le CB$ for some constants $c,C>0$ independent of sufficiently small~$\eps$ and of arguments $x,y$. Throughout, we use the length scale
\begin{equation}\label{eq:length}
 \ell=\eps^{1/(d+2)}.
\end{equation}

\begin{theorem}[Uniform ellipticity and section geometry]\label{thm:geometry}
Suppose \cref{ass:geometrydata} holds. Let $u_\eps,v_\eps$ be the convex transformed potentials of the optimizer in \eqref{eq:problem}. There are constants $c,C,\eps_0>0$ such that, for $0<\eps<\eps_0$,
\begin{equation}\label{eq:hessians}
 c\Id\preceq D^2u_\eps\preceq C\Id\quad\text{a.e. on }X,\qquad
 c\Id\preceq D^2v_\eps\preceq C\Id\quad\text{a.e. on }Y.
\end{equation}
The sections~\eqref{eq:sections} of $\spt\pi_\eps$ satisfy, with $\ell$ as in~\eqref{eq:length},
\begin{align}\label{eq:balls}
 B(\nabla u_\eps(x),c\ell)&\subset S_{\eps,x}
     \subset B(\nabla u_\eps(x),C\ell) &&(x\in X),\nonumber\\
 B(\nabla v_\eps(y),c\ell)&\subset K_{\eps,y}
     \subset B(\nabla v_\eps(y),C\ell) &&(y\in Y).
\end{align}
Recalling the slack $q_\eps(x,y)=x\cdot y-u_\eps(x)-v_\eps(y)$, we then have
\begin{equation}\label{eq:scales}
\begin{aligned}
 |S_{\eps,x}|&\asymp\eps^{d/(d+2)},&
 |K_{\eps,y}|&\asymp\eps^{d/(d+2)},\\
 \max_{y\in Y}q_\eps(x,y)&\asymp\eps^{2/(d+2)},&
 \max_{x\in X}q_\eps(x,y)&\asymp\eps^{2/(d+2)}.
\end{aligned}
\end{equation}
The constants in \eqref{eq:balls}--\eqref{eq:scales} are uniform in the respective base point $x,y$ and in $0<\eps<\eps_0$.
\end{theorem}
We emphasize that the inclusions \eqref{eq:balls} hold also for $x\in\partial X$ and $y\in\partial Y$. The inner balls in~\eqref{eq:balls} are full Euclidean balls of $\R^d$ (not intersections of balls with $Y$ or $X$). Since they are necessarily contained in $Y$ and $X$, respectively, we have in particular
\begin{equation}\label{eq:centerboundary}
 \dist(z_{\eps,x},\partial Y)\ge c\ell,\qquad
 \dist(\zeta_{\eps,y},\partial X)\ge c\ell.
\end{equation}
The centers of these balls are the barycenters under the restricted marginal measures, as expressed in \eqref{eqv:eq:centers}. However, we show in \eqref{eq:conditional-mean-section-balls} that analogous inclusions also hold with the conditional means as centers.

\bigskip

To state the second main result, let $T=\nabla u_0:X\to Y$ be the Brenier map and $S=T^{-1}=\nabla v_0$, where the potentials $u_0,v_0$ for the classical (unregularized) optimal transport problem are chosen with
\begin{equation}\label{eq:classicaldual}
 u_0(x)+v_0(y)\ge x\cdot y,\qquad
 u_0(x)+v_0(T(x))=x\cdot T(x),\qquad T_\#\mu=\nu.
\end{equation}
More precisely, \cref{lem:classicalfinite} gives uniformly elliptic classical potentials $u_0,v_0\in C^{3,\beta}$ for some $\beta>0$; in particular, their Hessians are Lipschitz.

\begin{theorem}[Convergence of support and potentials]\label{thm:sharp}
Suppose \cref{ass:geometrydata} holds. Let $u_\eps,v_\eps$ be the convex transformed regularized potentials, and let $u_0,v_0$ be the classical convex dual potentials, so that $T=\nabla u_0:X\to Y$ and $T^{-1}=\nabla v_0:Y\to X$. Choose their additive constants by
\begin{equation}\label{eq:gauge}
 \int_Xu_\eps\dd\mu=\int_Yv_\eps\dd\nu,\qquad
 \int_Xu_0\dd\mu=\int_Yv_0\dd\nu.
\end{equation}
There are constants $c,C,\eps_1>0$ such that, for $0<\eps<\eps_1$ and every $1\le p\le\infty$,
\begin{align}\label{eq:allp}
 c\eps^{2/(d+2)}&\le\|u_\eps-u_0\|_{L^p(\mu)}\le C\eps^{2/(d+2)},\nonumber\\
 c\eps^{2/(d+2)}&\le\|v_\eps-v_0\|_{L^p(\nu)}\le C\eps^{2/(d+2)}.
\end{align}
Moreover, the uniform gradient error satisfies
\begin{equation}\label{eq:sharpgradient}
 \|\nabla u_\eps-\nabla u_0\|_{L^\infty(X)}+
 \|\nabla v_\eps-\nabla v_0\|_{L^\infty(Y)}
 \asymp\eps^{1/(d+2)}.
\end{equation}
For the sections~\eqref{eq:sections} of $\spt\pi_\eps$, we have
\begin{align}\label{eq:sharpfiber}
 c\eps^{1/(d+2)}&\le
 \sup_{y\in S_{\eps,x}}|y-T(x)|\le C\eps^{1/(d+2)}
 &&(x\in X),\nonumber\\
 c\eps^{1/(d+2)}&\le
 \sup_{x\in K_{\eps,y}}|x-T^{-1}(y)|\le C\eps^{1/(d+2)}
 &&(y\in Y).
\end{align}
Finally, the Hausdorff distance to the Brenier graph obeys
\begin{equation}\label{eq:hausdorff}
d_H(\spt\pi_\eps,\operatorname{graph}T)\asymp\eps^{1/(d+2)}.
\end{equation}
The upper bounds in \eqref{eq:allp} also hold if, instead of \eqref{eq:gauge}, the pairs are normalized by $u_\eps(x_*)=u_0(x_*)$ at a fixed point $x_*\in X$.
\end{theorem}
Here $
 d_H(A,B)=\max\left\{\sup_{a\in A}\dist(a,B),\ \sup_{b\in B}\dist(b,A)\right\}
$ is the Hausdorff distance for the Euclidean norm on $\R^{2d}$. The lower bounds for the individual potential norms in \eqref{eq:allp} use the normalization \eqref{eq:gauge}, whereas the support and gradient statements do not depend on the normalization. We emphasize that \eqref{eq:sharpfiber} controls every point of the support rather than only a large fraction of the coupling mass.

\begin{remark}[Brenier graph inclusion]\label{rk:BrenierGraphInclusion}
The present results are used in the companion \cite{GNProfiles} to show that there are constants $c,r,R,\eps_*>0$ such that, for $0<\eps<\eps_*$ and every $x\in X$,
\begin{equation}\label{eq:companion-graph-inclusion}
q_\eps(x,T(x))\ge c\ell^2,\qquad Y\cap B(T(x),r\ell)\subset S_{\eps,x}\subset Y\cap  B(T(x),R\ell).
\end{equation}
In particular, $\operatorname{graph}T\subset\spt\pi_\eps$, i.e., the Brenier graph is contained in the regularized support for $\eps<\eps_*$. (This can fail for $\eps>\eps_*$.) Note that the neighborhoods asserted in \eqref{eq:companion-graph-inclusion} are relative to $Y$, unlike the full balls in \eqref{eq:balls}, and that this restriction is unavoidable at the boundary. The proof of \eqref{eq:companion-graph-inclusion} builds on \cref{thm:geometry,thm:sharp} by passing to a rescaled limit and classifying the solutions to a limiting equation.
\end{remark}

The gradient rate in \eqref{eq:sharpgradient} is forced by the boundary geometry. For every $x\in\partial X$, the
image $T(x)$ belongs to $\partial Y$, whereas
\eqref{eq:centerboundary} places $\nabla u_\eps(x)$ at
distance at least $c\ell$ from $\partial Y$. Consequently,
\[
 |\nabla u_\eps(x)-\nabla u_0(x)|
 \ge \dist(\nabla u_\eps(x),\partial Y)\ge c\ell.
\]
Together with the upper bound in \eqref{eq:sharpgradient},
this shows that the gradient error is of order $\ell$ at
every source boundary point. The analogous statement
holds for $v_\eps$ on $\partial Y$. Thus the gradient rate
cannot be improved even at a single boundary point,
despite the $O(\ell^2)$ uniform convergence of the
potentials in \eqref{eq:allp}. The following corollary
quantifies convergence in intermediate H\"older norms
and shows that convergence breaks down in $C^{1,1}$.

\begin{corollary}[Moments and H\"older norms]\label{cor:consequences}
Suppose \cref{ass:geometrydata} and the normalization \eqref{eq:gauge} hold.
For every fixed $1\le q<\infty$, the displacement from $T=\nabla u_0$ satisfies
\begin{equation}\label{eq:displacement}
 \int|y-\nabla u_0(x)|^q\dd\pi_\eps\asymp_q\eps^{q/(d+2)}.
\end{equation}
For $0\le\beta\le1$ and $0\le\gamma<1$,
\begin{align}\label{eq:holder}
 \|u_\eps-u_0\|_{C^{0,\beta}(X)}+
 \|v_\eps-v_0\|_{C^{0,\beta}(Y)}
 &\asymp_\beta\eps^{(2-\beta)/(d+2)},\nonumber\\
 \|u_\eps-u_0\|_{C^{1,\gamma}(X)}+
 \|v_\eps-v_0\|_{C^{1,\gamma}(Y)}
 &\asymp_\gamma\eps^{(1-\gamma)/(d+2)},
\end{align}
whereas both $\|u_\eps-u_0\|_{C^{1,1}(X)}$ and $\|v_\eps-v_0\|_{C^{1,1}(Y)}$ are bounded away from zero as $\eps\downarrow0$. Here $C^{0,0}=C^0$, $C^{0,1}$ is the Lipschitz space and $C^{1,0}=C^1$. The BMO seminorms defined in \eqref{m:display:051} satisfy
\begin{equation}\label{eq:bmosharp-main}
 [u_\eps-u_0]_{\BMO(X)}\asymp\eps^{2/(d+2)},\qquad
 [v_\eps-v_0]_{\BMO(Y)}\asymp\eps^{2/(d+2)}.
\end{equation}
The potential difference along the classical graph satisfies
\begin{equation}\label{eq:graphsum}
 \|(u_\eps-u_0)+(v_\eps-v_0)\circ T\|_{L^\infty(X)}
 \asymp\eps^{2/(d+2)}.
\end{equation}
\end{corollary}

For probability measures $P,Q$ on the same Euclidean space, we denote by $W_p(P,Q)$ their $p$-Wasserstein distance, where $1\le p\le\infty$. When $P,Q$ are transport plans, Wasserstein distance is relative to the Euclidean norm on the product space.

\begin{corollary}[Conditional concentration and Wasserstein convergence]\label{cor:conditional-concentration}
Suppose \cref{ass:geometrydata} holds. Define the canonical conditional laws by
\begin{equation}\label{eq:canonical-conditional-laws}
 \pi_\eps^x(\!\dd y)=h_\eps(x,y)\nu(\!\dd y),\qquad
 \widehat\pi_\eps^y(\!\dd x)=h_\eps(x,y)\mu(\!\dd x),
\end{equation}
which are probability measures for every $x\in X$ and $y\in Y$. Let $T=\nabla u_0$ be the Brenier map. There are $c,C,\eps_2>0$ such that, for $0<\eps<\eps_2$ and every $1\le p\le\infty$,
\begin{align}\label{eq:rowwise-moment-rates}
 c\eps^{1/(d+2)}&\le
 \bigl\|y\mapsto |y-T(x)|\bigr\|_{L^p(\pi_\eps^x)}
 \le C\eps^{1/(d+2)} &&(x\in X),\nonumber\\
 c\eps^{1/(d+2)}&\le
 \bigl\|x\mapsto |x-T^{-1}(y)|\bigr\|_{L^p(\widehat\pi_\eps^y)}
 \le C\eps^{1/(d+2)} &&(y\in Y).
\end{align}
Moreover,
\begin{equation}\label{eq:plan-wasserstein-rate}
 c\eps^{1/(d+2)}\le
 W_p\bigl(\pi_\eps,(\id,T)_\#\mu\bigr)
 \le C\eps^{1/(d+2)},\qquad 1\le p\le\infty.
\end{equation}
\end{corollary}

\subsection{Related literature}\label{se:previousResults}
Starting with \cite{Muzellec.2017.AAAI,BSR,EssidSolomon.18}, quadratic regularization has been studied as an alternative to
entropic regularization that allows sparse couplings. Lorenz, Manns and
Meyer \cite{LMM} provide a continuous theory with Lebesgue measure as reference, proving existence and strong duality and developing several algorithms. For the
product reference measure used here and in most of the literature, Nutz \cite{Nutz} proves
existence of potentials for general square-integrable costs.
Under connectedness
assumptions, \cite{Nutz} establishes uniqueness of the potentials and qualitative localization of the
support near the unregularized optimal support as $\eps\downarrow0$.

The scale of the regularization error in the optimal value is already well understood. Eckstein and Nutz \cite{EN} obtain $\QOT_\eps-\OT\asymp\eps^{2/(d+2)}$ using a quantization technique for the upper bound and adapting arguments of Carlier, Pegon and Tamanini \cite{CarlierPegonTamanini.22} for the lower bound. Garriz-Molina, Gonz\'alez-Sanz and Mordant \cite{GGM} determine
the exact leading constant under regularity assumptions. Their
construction uses profiles related to the Barenblatt solution of the porous medium equation.
Cardoso-Perell\'o, Gonz\'alez-Sanz and Nutz \cite{CGN} extend these asymptotics to weaker assumptions on the
marginals and general $L^p$ penalties. %

For the geometry of the optimal support, Gonz\'alez-Sanz and Nutz
\cite{GSN} treat one-dimensional marginals with strictly positive
continuous densities on compact intervals. They prove uniform
two-sided bounds on the second derivatives of the transformed
potentials and show that every support section is an interval of
length comparable to $\eps^{1/3}$. They also obtain localization
around the Monge map and gradient convergence at rate $\eps^{1/3}$
in $L^2$. The one-dimensional Sobolev embedding then gives
$C^{0,1/2}$ convergence of suitably normalized potentials at the
same rate. The present results substantially improve on \cite{GSN} when specialized to $d=1$.

In higher dimensions, Wiesel and Xu \cite{WX} obtain quantitative
bounds on support concentration and barycentric bias by combining
pointwise density estimates with Minty's trick. In the
self-transport case $\mu=\nu$, they establish the global
concentration scale $\eps^{1/(d+2)}$ under mass-spread assumptions.
Nguyen-Chi, Nguyen and Nguyen \cite{NCN} prove, under regularity assumptions, that the directed Hausdorff distance from
the regularized support to the Monge graph cannot decay faster
than $\eps^{1/(d+2)}$. In the special case of marginals with an affine Brenier
map, they obtain the matching upper bound by reducing to
self-transport and applying the results of \cite{WX}.

Gvalani and Koch \cite{GK} establish interior Hessian bounds and local
$C^{1,\gamma}$ convergence of the potentials for $\gamma<1$.
Their method combines local support estimates with a Campanato
excess iteration stopped at the regularization scale.
Building on these bounds, Gonz\'alez-Sanz, Gvalani and Koch
\cite{GGK} prove that interior support sections contain and are
contained in balls centered at the gradients of the transformed
potentials. Both radii are comparable to
$\eps^{1/(d+2)}$. They also prove interior strong convexity and local $L^2$ localization around the Monge map at that rate. These two
works apply to H\"older densities with a bi-$C^{1,\alpha}$ Brenier
map and cover $L^p$ penalties with $1<p\le2$, not only quadratic
regularization; \cite{GK} also treats entropic regularization.

Our results establish the sharp section geometry and uniform Hessian bounds up to the boundary. The original inspiration for our approach, namely to couple these two facts, is a proof of Caffarelli for classical optimal transport. Caffarelli considers so-called sections of the potential, defined as sublevel sets cut out by translates of supporting hyperplanes \cite[Section~4.1]{Figalli.17}. Showing that these sections are ball-like (i.e., satisfy inclusions similar to \eqref{eq:balls}) is a key step to obtain the Hessian bounds. Our analytical implementation builds on the local variational approach of \cite{GK}, which in turn draws from Goldman and Otto \cite{GoldmanOtto2020} and others. The potential argument is also related to Gonz\'alez-Sanz, Mordant and Sheng \cite{GMS}, who
linearize the push-forward relation around a reference Brenier map
to obtain potential stability estimates. Here we apply this idea
to a corrected potential whose gradient is the conditional mean,
and obtain the final uniform estimate through a reversible
finite-range equation.

We emphasize that the present contribution goes beyond extending interior
regularity estimates to the boundary. We obtain sharp global control of both the
location and thickness of the regularized support, together
with the optimal uniform convergence rate for each potential. We remark that the potential rate $\eps^{2/(d+2)}$ does not
follow by directly integrating our uniform gradient
estimate---this yields only $O(\eps^{1/(d+2)})$ control
of the potential oscillation, in the spirit of the
argument in
\cite[Corollary~2.7]{GSN}. We obtain the sharper rate by
exploiting cancellations in the marginal
equations.
The boundary analysis also reveals a precise obstruction to
stronger convergence: the gradient error has order
$\eps^{1/(d+2)}$ at every boundary point, and convergence
fails in $C^{1,1}$ despite uniform two-sided Hessian bounds.
Together with the matching H\"older estimates, these results
identify both the optimal global convergence rates and the
regularity at which convergence breaks down.

Finally, we compare with entropically regularized optimal transport. As entropic optimizers always have full support, we only compare results for the potentials. Qualitative  convergence of potentials was established by Gigli and Tamanini \cite{GigliTamanini.21}, 
and convergence of gradients by Chiarini, Conforti, Greco and Tamanini \cite{CCGTGradients}. For strongly log-concave marginals on $\R^d$, Chewi and Pooladian \cite{CPHessians} establish global two-sided Hessian bounds for the convex transformed entropic potentials, uniform for small $\eps$, using covariance inequalities. Under curvature and Poincar\'e assumptions on the marginals, L\'opez-Rivera \cite{LRUniform} obtains locally uniform $O(\eps^{1/(d+4)})$ bounds for the normalized potential and its gradient by combining Hessian bounds and Gagliardo--Nirenberg interpolation. The interior Hessian estimates of Gvalani and Koch \cite{GK} also yield qualitative $C^{1,\gamma}_{\mathrm{loc}}$ convergence, $\gamma<1$, for entropic potentials. More recently, Gonz\'alez-Sanz, Mordant and Sheng \cite{GMS} prove the sharp rate $\eps\log(1/\eps)$ in $L^1(\mu\otimes\nu)$ for the sum of the two potential errors, using coupling-based linearization and elliptic duality. %
Here, for QOT, we obtain a complete theory with matching uniform rates and identification of the limiting regularity. To our knowledge, this is not currently available for entropic optimal transport, but it is plausible that some of the present techniques can be used to help achieve it.

\subsection{Proof strategy}
The proof separates three questions: how large the support sections are, where they are located, and how accurately the potentials approximate their classical counterparts. These questions require different arguments. We first establish the section geometry, then locate the sections through an elliptic estimate, and finally exploit a cancellation in the marginal equations to obtain the potential rate.

\paragraph{Section geometry.}
We build on the variational comparison and stopped Campanato iteration developed for interior regularity by Gvalani and Koch \cite{GK}. The starting point is the global energy bound $\QOT_\eps-\OT\le C\ell^2$ known from \cite{EN}. Uniform convexity of the classical potentials turns it into the averaged estimate $\int|y-T(x)|^2\dd\pi_\eps\le C\ell^2$. This average does not control every section: a small amount of mass could still lie much farther than $\ell$ from the Brenier graph. The iteration extracts local information by comparing the regularized plan with unregularized optimal transport on successively smaller neighborhoods.

For these comparisons, we restrict $\pi_\eps$ to a product of source and target balls of radius~$R$ and retain the two marginals of that restriction. A quantization construction, using the cyclical monotonicity argument of \cite[Lemma~3.13]{EN}, partitions an unregularized optimal plan into small pieces and replaces each piece by the product of its marginals divided by its mass. This produces a competitor with finite penalty and exactly the required marginals. The penalty is measured against the original product reference $\mu\otimes\nu$, so that replacing the plan inside the product set leaves both global marginal constraints intact and creates no cross-term in the quadratic penalty. This gives a local comparison error at most $CR^d\ell^2$; see \cref{geom:prop:rectangle}. Importantly, the construction does not require smoothness of the restricted marginals.

A key difficulty in treating the boundary is to ensure that the unregularized comparison problem nevertheless has smooth densities near the point being studied. At a boundary pair $(p,T(p))$, where $T$ is the Brenier map, we apply affine changes whose source and target linear parts are inverse transposes. These coordinates make the derivative of $T$ the identity and place the two supports on the same side of a common tangent plane. In these coordinates, we prove a boundary counterpart of the local support estimate of \cite{GK}: small averaged displacement forces every support pair with one endpoint in a smaller ball to have its other endpoint nearby. Hence restriction to a slightly larger product of balls removes no mass from either marginal in the smaller neighborhood. There the comparison marginals agree exactly with the original smooth densities. We apply the boundary $\eps$-regularity theorem of Chen and Figalli \cite{CF}, then linearize the Monge--Amp\`ere equation and its boundary correspondence. The resulting oblique estimates give an affine approximation to the comparison map, with error controlled by its root mean-square displacement from the identity after rescaling the neighborhood to unit size, together with the rescaled variation of the boundaries and densities.

At the $n$th scale $R_n=\theta^nR_0$, we recenter, apply the current reciprocal affine change, and divide lengths by $R_n$. The current affine fit then becomes the identity in the normalized coordinates $(\xi,\zeta)$. Let $\pi^{(n)}$ denote the transformed plan multiplied by $R_n^{-d}$. Its excess is
\[
E_n=\int_{\{|\xi|<1\}\,\cup\,\{|\zeta|<1\}}|\xi-\zeta|^2\dd\pi^{(n)}(\xi,\zeta).
\]
Thus $E_n$ measures the dimensionless squared deviation from the current affine fit over pairs with at least one endpoint in the corresponding ``physical'' neighborhood of size $R_n$. Repeated comparison gives
\[
E_{n+1}\le aE_n+CR_n^2+C(\ell/R_n)^2,\qquad 0<a<1,
\]
where $0<\theta<1$ and the initial radius $R_0>0$ are fixed independently of $\eps$. The last term explains why the iteration stops at a sufficiently large fixed multiple of $\ell$. The affine changes have summable distortions, and the local support estimate at the stopping scale yields section diameters $O(\ell)$. To cover points near, but not on, the boundary, we use the boundary estimates to restart an interior iteration at a radius comparable with the distance to the boundary. This gives uniform bounds throughout both supports.

The passage from section geometry to curvature builds on the marginal differentiation formulas used in \cite{GSN,GGK}. The two diameter bounds first control the slack gradient on each section. The marginal equations then give height $\asymp\ell^2$ and volume $\asymp\ell^d$, and basic convex geometry supplies full inner balls centered at $\nabla u_\eps$ and $\nabla v_\eps$. We use the moving-section Hessian formula of \cite[Proposition~4.3]{GGK} to recover uniform curvature. Its boundary integral includes only the part of the section boundary lying inside the marginal support. Our free-boundary moment estimate shows that this part still has a nondegenerate second moment in every direction, even when the section meets the marginal boundary. This completes \cref{thm:geometry} without a quantitative estimate for the location of the centers.

\paragraph{Locating the sections through the conditional mean.}
The quantity that appears when testing the marginal constraints is the conditional mean $m_\eps(x)=\int y\,\pi_\eps^x(\dd y)$, rather than the restricted-marginal barycenter $\nabla u_\eps(x)$. We relate them through the row penalty:
\[
a_\eps(x)=\frac1{2\eps}\int_Y(q_\eps(x,y)_+)^2\dd\nu(y),\qquad \Theta_\eps=u_\eps+a_\eps,\qquad \nabla\Theta_\eps=m_\eps.
\]
The section geometry gives $\|a_\eps\|_\infty\le C\ell^2$, $\|\nabla a_\eps\|_\infty\le C\ell$, and uniform upper and lower Hessian bounds for $\Theta_\eps$. Thus the conditional mean is itself the gradient of a uniformly convex potential, and passing between the two notions of center costs only $O(\ell)$.

We adapt the coupling-based linearization and elliptic duality method used by Gonz\'alez-Sanz, Mordant and Sheng for entropic optimal  transport in \cite[Section~4.3]{GMS}. For QOT, the row-penalty correction supplies the potential whose gradient is $m_\eps$. Since $m_\eps$ need not push $\mu$ forward to~$\nu$, we use the marginal identity for the coupling itself rather than treating $m_\eps$ as a transport map. For any smooth test function $\phi$, the marginal constraints give
\[
\int\bigl(\phi(y)-\phi(T(x))\bigr)\dd\pi_\eps(x,y)=0.
\]
Taylor expansion at $T(x)$ expresses the linear term through $m_\eps-T$ and leaves a remainder quadratic in $y-T(x)$. Put $w=\Theta_\eps-u_0$ and $W=w-\int_Xw\dd\mu$, so $\nabla W=m_\eps-T$. Choosing test functions through the Neumann problem for the fixed operator $-\operatorname{div}(\rho_0(DT)^{-1}\nabla)$ turns the preceding identity into an estimate on $W$. %

To control the Taylor remainder, we must bound the densities of the interpolated points $(1-t)T(x)+ty$, $0\le t\le1$. Each such point is within $Ct\ell$ of $(1-t)T(x)+tm_\eps(x)$, a uniformly strongly monotone map. For $t>0$, only source points in a set of volume $O((t\ell)^d)$ can contribute to a given interpolation point. Together with $h_\eps\le C\ell^{-d}$ and the Jacobian factor $t^{-d}$ from the change of target variable, this gives uniformly bounded interpolation densities. Consequently, the quadratic remainder tensor is bounded by $C(\ell^2+\|m_\eps-T\|_\infty^2)$. The elliptic adjoint maps bounded tensor fields into BMO and yields
\[
[W]_{\BMO(X)}\le C\bigl(\ell^2+\|\nabla W\|_\infty^2\bigr).
\]
The BMO seminorm measures the largest mean oscillation over relative balls in $X$; its precise definition is \eqref{m:display:051}. A uniform Hessian bound alone does not yield a bound for $[W]_{\BMO(X)}$ in terms of $\ell$: gradient interpolation would produce a term $C[W]_{\BMO(X)}$ on the right-hand side, with no guarantee that $C<1$. Moving this term to the left therefore need not give an upper bound.

The additional input comes from using the stopped iteration a second time. For each prescribed accuracy $\tau>0$, there is a fixed $L_\tau$ such that, at scale $R=L_\tau\ell$ and for sufficiently small $\eps$, the root mean-square error of a constant fit to $m_\eps-T$ on each relative ball is at most $\tau R$. The fitted constant may vary between balls and need not be small. Patching the corresponding affine approximations of $W$ gives $W=V+E$ with $\|D^2V\|_\infty\le\eta_0$ for any fixed $\eta_0>0$ and errors $\|E\|_\infty=O(\ell^2)$, $\|\nabla E\|_\infty=O(\ell)$, with constants allowed to depend on $\eta_0$. Interpolating the gradient of $V$, rather than that of $W$, gives the needed small coefficient:
\[
[W]_{\BMO(X)}\le C_0\eta_0[W]_{\BMO(X)}+C_0[W]_{\BMO(X)}^2+C_{\eta_0}\ell^2.
\]
Here $\eta_0$ can be chosen independently of $\eps$, and the preliminary energy and uniform Hessian bounds ensure $[W]_{\BMO(X)}\to0$. We choose $\eta_0$ and then $\eps$ small enough that the first two terms on the right sum to at most $[W]_{\BMO(X)}/2$. Moving them to the left gives $[W]_{\BMO(X)}=O(\ell^2)$. Since $a_\eps=O(\ell^2)$, the same BMO bound holds for $u_\eps-u_0$; the uniform Hessian bounds then imply $\|\nabla u_\eps-T\|_\infty=O(\ell)$. Together with the diameter bounds, this locates every section within $O(\ell)$ of its Brenier image. In summary, the combination of the corrected conditional-mean potential, elliptic duality, and coarse affine approximation upgrades averaged concentration to sharp global localization.

\paragraph{The cancellation behind the potential rate.}
BMO control does not give an $L^\infty$ bound at the same scale. Passing directly from derivative estimates to potential estimates, as in the one-dimensional Poincar\'e--Sobolev argument of \cite[Corollary~2.7]{GSN}, would retain the slower scale $O(\ell)$. For the $O(\ell^2)$ uniform rate, we return to the two marginal equations and exploit their symmetry after pulling the target back by $T^{-1}$.

Write $r=u_\eps-u_0$ and $s=v_\eps-v_0$. The sharp support bound and the slack-gradient estimate imply $|q_\eps(x,T(x))|\le C\ell^2$. Classical dual equality therefore controls the sum $r+s\circ T$, but not yet $r$ and $s$ separately. Define
\[
e(x):=r(x)+s(T(x))=-q_\eps(x,T(x)),\qquad \eta(x):=\tfrac12\bigl(r(x)-s(T(x))\bigr).
\]
We have $e=O(\ell^2)$, while the symmetric normalization gives $\int_X\eta\dd\mu=0$. It remains to bound~$\eta$.

The pulled-back density $k(x,z)=h_\eps(x,T(z))$ has row and column integrals equal to one relative to $\mu$. Comparing $q_\eps(x,T(z))$ with $q_\eps(z,T(x))$ cancels the already controlled sum $e$. The remaining cost term is the antisymmetric part of the classical Bregman divergence:
\[
D(x,z)=u_0(x)-u_0(z)-T(z)\cdot(x-z),\qquad A(x,z)=\tfrac12\bigl(D(x,z)-D(z,x)\bigr).
\]
Although each divergence is quadratic to leading order, their antisymmetric part satisfies $|A(x,z)|\le C|x-z|^3$ as the constant-Hessian contribution cancels. Taking the divided difference of the positive-part function converts the difference of the two slacks into the difference $k(x,z)-k(z,x)$. The row and column identities then give an equation
\[
\int_XJ_\eps(x,z)\bigl(\eta(z)-\eta(x)-A(x,z)\bigr)\dd\mu(z)=0
\]
with symmetric nonnegative weights $J_\eps$ that vanish unless $|x-z|\le C\ell$. Thus every pair that enters the equation satisfies $|A(x,z)|\le C\ell^3$, and $A(z,x)=-A(x,z)$. This cubic antisymmetric defect (rather than the quadratic size of each Bregman divergence) is the cancellation that yields the additional power of $\ell$ in the potential estimate.

Normalizing the weights gives a reversible Markov operator, but the estimates obtained so far do not provide a lower bound for its transition density near the diagonal. We obtain the required lower bound after two steps: nearby rows share a region of volume comparable with $\ell^d$ on which both densities are bounded below by $c\ell^{-d}$. This overlap gives a lower bound for the two-step transition density near the diagonal, while the prescribed term in the two-step equation remains antisymmetric and bounded by $C\ell^3$. The resulting equation is covered by \cref{end:thm:nonlocal}, which controls the oscillation of its solution by the size of this term divided by the interaction length. It therefore gives $\sup_X\eta-\inf_X\eta\le C\ell^2$. Retaining the antisymmetry, rather than estimating only the row integral of this term as a scalar source, is essential for this gain.

We prove the required nonlocal estimate by an energy argument and a finite Moser iteration. The available Sobolev estimate improves integrability only after averaging with the kernel, so the equation must be used at every step. This introduces a factor $1+\ell$ that would accumulate in an infinite iteration. We stop at an exponent $p$ comparable with $\log(1/\ell)$ and use the kernel's upper bound for the final passage to $L^\infty$: at this exponent, the loss $\ell^{-d/p}$ is bounded. The constants therefore remain uniform as $\ell\downarrow0$. The zero mean of $\eta$ and the bound for $e$ now give the two uniform potential estimates. This support-to-potential argument is formulated separately in \cref{end:thm:main}; once sharp localization is known, it does not require the uniform Hessian bounds for the regularized potentials.

\subsection{Organization}
The remainder of the paper is organized as follows. Section~\ref{sec:foundations} collects preliminaries to be used throughout, including regularity of the classical potentials, marginal and moving-section identities, convex geometry estimates, and reciprocal affine rescaling. Section~\ref{sec:boundary} develops the boundary localization argument: the rectangular comparison, pointwise support control from local excess, local regularity of the unregularized comparison maps, and the one-step excess improvement. Section~\ref{sec:uniformgeometry} iterates this estimate down to the regularization scale to obtain the section geometry and uniform curvature, completing the proof of \cref{thm:geometry} in Section~\ref{geom:sec:geometry}. Section~\ref{sec:conditional-regularity} proves the conditional covariance and transport-cost excess estimates in \cref{prop:conditional-covariance}, together with the ball inclusions centered at the conditional means in \eqref{eq:conditional-mean-section-balls}. Section~\ref{sec:localization} constructs a potential for the conditional mean and combines coarse affine approximation from the stopped iteration with elliptic duality and BMO absorption to prove the support and gradient assertions of \cref{thm:sharp}, as well as the BMO bounds \eqref{eq:bmosharp-main}. Section~\ref{end:sec:criterion} derives the sharp potential estimates in all $L^p$ norms from support localization through a reversible nonlocal equation and a uniform nonlocal $L^\infty$ estimate, completing the proof of \cref{thm:sharp} in Section~\ref{end:sec:completion}. Finally, Section~\ref{sec:consequences} completes the proofs of \cref{cor:consequences,cor:conditional-concentration}.

\section{Preliminaries}
\label{sec:foundations}
We use $B(x,r)$ for the open Euclidean ball, $B_r=B(0,r)$, $|D|$ for Lebesgue volume in the relevant dimension, $\id$ for the identity map and $\Id$ for the identity matrix. Matrix norms are operator norms. For $f\in C^{1,1}(D)$ on a convex body $D$, the norm
$\|D^2f\|_\infty$ is the Lebesgue essential supremum on
$\intr D$; it equals the Lipschitz constant of $\nabla f$ on $D$. A convex body is a compact convex set with nonempty interior. Surface integrals are taken with respect to $(d-1)$-dimensional Hausdorff measure $\HH$. For $d=1$, $\HH=\mathcal H^0$ is counting measure and boundary charts use $\R^0=\{0\}$. Constants denoted by $c,C$ may change from line to line and depend only on the dimension, the fixed domains, the density bounds and the stated regularity parameters and norms. Local rescalings use finite measures of equal positive mass, with the same objective as \eqref{eq:problem}.

\subsection{Regularity of the classical potentials}
\begin{lemma}\label{lem:classicalfinite}
Under \cref{ass:geometrydata}, there are $\beta\in(0,\alpha)$ and $0<\lambda\le\Lambda<\infty$ such that
\begin{equation}\label{eq:finiteclassical}
\begin{gathered}
 u_0\in C^{3,\beta}(X),\qquad v_0\in C^{3,\beta}(Y),\\
 \lambda\Id\preceq D^2u_0\preceq\Lambda\Id\quad\text{on }X,\qquad \lambda\Id\preceq D^2v_0\preceq\Lambda\Id\quad\text{on }Y.
\end{gathered}
\end{equation}
In particular, $T:X\to Y$ and $S:Y\to X$ are mutually inverse $C^{2,\beta}$ diffeomorphisms, $A:=\rho_0(DT)^{-1}\in C^{1,\beta}(X)$ is uniformly elliptic, and $D^2u_0,D^2v_0$ are Lipschitz on $X,Y$, respectively.
\end{lemma}
\begin{proof}
The case $d=1$ is straightforward, so let $d\ge2$. By Caffarelli's global boundary regularity \cite{Caffarelli} in the form of \cite[Theorem~1.1]{CF}, applied to both transport maps, $T=\nabla u_0$ and $S=\nabla v_0=T^{-1}$ belong to $C^{1,\alpha}$ up to the boundary. The identity $DS(T(x))DT(x)=\Id$ and compactness then give the two-sided Hessian bounds in \eqref{eq:finiteclassical}.

Let $n_X,n_Y$ denote the outward unit normals. Differentiating $T(\partial X)=\partial Y$ tangentially gives $D^2u_0(x)n_Y(T(x))=a(x)n_X(x)$, where $a(x)>0$ because $T$ maps interior points to interior points and $DT$ is invertible. Thus $n_Y(T(x))\cdot n_X(x)\ge\lambda/\Lambda$, so the second boundary condition is uniformly oblique. 
The oblique Schauder estimates \cite[Lemma~6.29 and Theorem~6.30]{GT}, applied to tangential difference quotients in $C^{3,\alpha}$ boundary charts, give uniform $C^{2,\alpha}$ bounds on smaller charts. Indeed, the interior coefficients and right-hand sides are uniformly $C^{0,\alpha}$, the boundary coefficients and data are uniformly $C^{1,\alpha}$, and obliqueness persists for sufficiently small increments. Passing to the limit gives $C^{2,\beta}$ regularity of the tangential derivatives for some $\beta\in(0,\alpha)$. Solving the transformed determinant equation for the second normal derivative then yields $u_0\in C^{3,\beta}(X)$. 
The same argument applies to $v_0$. The remaining assertions follow.
\end{proof}

\subsection{Marginal identities and support sections}
We first summarize some well known facts. The transformed potentials $u_\eps,v_\eps$ are convex and belong to $C^{1,1}(X)$, $C^{1,1}(Y)$, respectively, for each fixed $\eps>0$; see \cite[Lemma~2.2 and Corollary~4.2]{GdBN}. A priori, the Lipschitz constants of their gradients may depend on $\eps$. When $\eps$ is fixed, we write $S_x,K_y,z_x,\zeta_y$ for the quantities in \eqref{eq:sections}. The following row and column marginal equations hold at every point of the closed supports:
\begin{equation}\label{eqv:eq:marginal}
 \int_Y(q_\eps(x,y))_+\dd\nu(y)=\eps\quad(x\in X),\qquad \int_X(q_\eps(x,y))_+\dd\mu(x)=\eps\quad(y\in Y).
\end{equation}
The marginal equations and concavity imply that $S_x,K_y$ are convex bodies and that the zero-slack sets are contained in their boundaries. In particular, these zero sets have zero marginal measure. Differentiating \eqref{eqv:eq:marginal} yields
\begin{equation}\label{eqv:eq:centers}
 z_x=\frac{\int_{S_x}y\dd\nu(y)}{\nu(S_x)},\qquad \zeta_y=\frac{\int_{K_y}x\dd\mu(x)}{\mu(K_y)},
\end{equation}
with the identities at boundary base points obtained by continuity. Each section contains a point of positive slack. Concavity therefore implies that every zero-slack point is a limit of positive-slack points in the same section. Positivity of the marginal densities gives
\begin{equation}\label{eqv:eq:support}
 \spt\pi_\eps=\{(x,y)\in X\times Y:q_\eps(x,y)\ge0\}.
\end{equation}

For a.e. $x\in\intr X$ and every unit vector $e$,
\begin{equation}\label{eqv:eq:hessian}
 e^TD^2u_\eps(x)e=\frac1{\nu(S_x)}\int_{\partial S_x\cap\intr Y}\frac{((y-z_x)\cdot e)^2\rho_1(y)}{|x-\zeta_y|}\dd\HH(y).
\end{equation}
The analogous identity holds with the marginals interchanged. The denominator cannot vanish on the free boundary since a stationary point of the concave function $q_\eps(x,\cdot)$ with value zero would contradict the positive row integral in \eqref{eqv:eq:marginal}. To derive the formula, differentiate the first identity in \eqref{eqv:eq:centers}. As $x$ varies in direction $e$, the outward normal velocity on the free boundary $\partial S_x\cap\intr Y$, where $q_\eps(x,y)=0$, is
$
 \frac{(y-z_x)\cdot e}{|x-\zeta_y|}.
$
Differentiating the quotient gives the quadratic numerator in \eqref{eqv:eq:hessian}. The portion on the fixed boundary $\partial Y$ contributes no normal velocity. Moreover, for each $y\in\partial Y$, the column marginal equation
and concavity imply
$\{x\in X:q_\eps(x,y)=0\}\subset\partial K_y$, so this set has
zero Lebesgue measure. Fubini's theorem therefore gives
$\mathcal H^{d-1}\bigl(\{y\in\partial Y:q_\eps(x,y)=0\}\bigr)=0$
for a.e.\ $x\in X$.
The identity then follows from the coarea formula and distributional differentiation as in \cite[Proposition~4.3]{GGK}.

\subsection{Convex geometry estimates}
\ifjournalversion
This section provides two geometric lemmas in $\R^d$. The proofs are omitted for the sake of brevity. (They are included in the arXiv version of this paper.)
\fi

\begin{lemma}[Centroid geometry]\label{eqv:lem:centroid}
Let $K\subset\R^d$ be a convex body and let $\bar z$ be its centroid with respect to a density $\rho$ satisfying $0<m\le\rho\le M<\infty$ on $K$, i.e.,
\[
\bar z:=\frac{\int_K x\rho(x)\dd x}{\int_K\rho(x)\dd x}.
\]
If $|K|\ge c_0r^d$ and $\diam K\le C_0r$ for some $c_0,C_0,r>0$, then $B(\bar z,c_1r)\subset K$ for some $c_1>0$ depending only on $d,m/M,c_0$ and $C_0$.
\end{lemma}
\ifjournalversion\else
\begin{proof}
Write $\sigma(A):=\int_A\rho(x)\dd x$ for $A\subset K$. Fix a unit vector $e$, set $a:=\min_{x\in K}e\cdot x$ and $b:=\max_{x\in K}e\cdot x$, and denote the width in direction $e$ by $\operatorname{width}_e(K):=b-a$. Choose $x_+\in K$ with $e\cdot x_+=b$ and put $K_+:=(K+x_+)/2$. Convexity gives $K_+\subset K$, while $|K_+|=2^{-d}|K|$ and $e\cdot x-a\ge(b-a)/2$ on $K_+$. Hence
\[
e\cdot\bar z-a=\frac{1}{\sigma(K)}\int_K(e\cdot x-a)\rho(x)\dd x\ge\frac{b-a}{2}\frac{\sigma(K_+)}{\sigma(K)}\ge\frac mM\,2^{-(d+1)}\operatorname{width}_e(K).
\]
Every nonempty section of $K$ by a hyperplane perpendicular to $e$ is contained in a closed $(d-1)$-dimensional ball of radius $\diam K$. Integrating its measure over the projection interval gives
\[
|K|\le\omega_{d-1}(\diam K)^{d-1}\operatorname{width}_e(K),
\]
where $\omega_{d-1}$ is the volume of the unit ball in $\R^{d-1}$, with $\omega_0:=1$. Thus
\[
\operatorname{width}_e(K)\ge\frac{c_0}{\omega_{d-1}C_0^{d-1}}r.
\]
Combining the two estimates yields $e\cdot\bar z-a\ge c_1r$ for every unit vector $e$, with $c_1:=(m/M)2^{-(d+1)}c_0/(\omega_{d-1}C_0^{d-1})$. The supporting-halfspace representation of $K$ therefore gives $B(\bar z,c_1r)\subset K$.
\end{proof}
\fi

\begin{lemma}[Free-boundary second moments]\label{eqv:lem:cone}
Fix a convex body $\Omega\subset\R^d$. There is $\delta_*>0$, depending only on $\Omega$, with the following property. Let $K\subset\Omega$ be a convex body, and let $\bar z\in\intr K$, $\kappa>0$ and $r>0$ satisfy
\begin{equation}\label{m:display:003}
 B(\bar z,\kappa r)\subset K\subset B(\bar z,r),\qquad r<\delta_*.
\end{equation}
Then, for every unit vector $e$,
\begin{equation}\label{eqv:eq:freemoment}
 c r^{d+1}\le\int_{\partial K\cap\intr\Omega}((z-\bar z)\cdot e)^2\dd\HH(z)\le C r^{d+1}.
\end{equation}
The constants $c,C>0$ may depend on $\Omega,d,\kappa$, but not on $K,\bar z,r$ or $e$. %
\end{lemma}

\ifjournalversion\else
\begin{proof}
The case $d=1$ is straightforward, so we assume $d\ge2$. Choose $x_0\in\intr\Omega$ and $a>0$ such that $B(x_0,2a)\subset\intr\Omega$, and set $\delta_*:=a/4$.  We first construct a uniform cone of directions pointing into $\Omega$. For every $x\in\Omega$, there is a unit vector $v$ such that $x+t\theta\in\intr\Omega$ whenever $0<t\le a/2$ and $\theta\in\mathbb S^{d-1}$ satisfies $|\theta-v|<a/\diam\Omega$. Indeed, if $|x-x_0|<a$, then $B(x,a)\subset B(x_0,2a)$, and any $v$ works. Otherwise, take $v:=(x_0-x)/|x_0-x|$. For every such $\theta$, the point $y:=x+|x_0-x|\theta$ satisfies $|y-x_0|=|x_0-x||\theta-v|<a$, so $y\in\intr\Omega$. Moreover, $x+t\theta$ lies on the segment from $x$ to $y$, strictly away from $x$, since $0<t\le a/2<|x_0-x|$. Convexity therefore gives $x+t\theta\in\intr\Omega$.

Apply this construction at $x=\bar z$, and denote the resulting spherical cap by $\mathcal C:=\{\theta\in\mathbb S^{d-1}:|\theta-v|<a/\diam\Omega\}$. Its aperture depends only on $\Omega$. For $\theta\in\mathbb S^{d-1}$, let $R(\theta):=\max\{t\ge0:\bar z+t\theta\in K\}$ and $F(\theta):=\bar z+R(\theta)\theta$. Convexity and \eqref{m:display:003} imply that each ray meets $\partial K$ at the unique point $F(\theta)$, with $\kappa r\le R(\theta)\le r$. Since $r<\delta_*<a/2$, the cone construction gives $F(\mathcal C)\subset\partial K\cap\intr\Omega$.

We next justify the change of variables on this boundary portion. If $n$ is any outward unit supporting normal at $F(\theta)$, the inner-ball inclusion gives $R(\theta)\theta\cdot n\ge\kappa r$, and hence $\theta\cdot n\ge\kappa$. These inequalities also imply that $R$ is Lipschitz. Indeed, if $R(\phi)\ge R(\theta)$, the supporting-plane inequality at $F(\theta)$ gives $(R(\phi)-R(\theta))\theta\cdot n\le R(\phi)(\theta-\phi)\cdot n\le r|\theta-\phi|$. Interchanging $\theta$ and $\phi$ when necessary yields $|R(\phi)-R(\theta)|\le(r/\kappa)|\phi-\theta|$. In particular, $F$ is an injective Lipschitz parametrization of $\partial K$.

Let $n(z)$ denote the outward unit normal, defined at $\HH$-almost every point of $\partial K$. At $z=F(\theta)$, the derivative of the radial projection $z\mapsto(z-\bar z)/|z-\bar z|$ is $R(\theta)^{-1}(\Id-\theta\otimes\theta)$. Orthogonal projection from the tangent hyperplane of $\partial K$ onto $\theta^\perp$ has area factor $\theta\cdot n(F(\theta))$. Consequently, the surface Jacobian of $F$ is
\begin{equation}\label{m:display:004}
 J_F(\theta)=\frac{R(\theta)^{d-1}}{\theta\cdot n(F(\theta))}\ge R(\theta)^{d-1}\qquad\text{for almost every }\theta.
\end{equation}
The denominator belongs to $[\kappa,1]$ by the supporting-plane estimate above.

Write $\sigma$ for surface measure on $\mathbb S^{d-1}$. Since $(F(\theta)-\bar z)\cdot e=R(\theta)(\theta\cdot e)$, the area formula and \eqref{m:display:004} give
\begin{equation}\label{m:display:005}
\begin{aligned}
 \int_{\partial K\cap\intr\Omega}((z-\bar z)\cdot e)^2\dd\HH(z)&\ge\int_{\mathcal C}R(\theta)^2(\theta\cdot e)^2J_F(\theta)\dd\sigma(\theta)\\
 &\ge(\kappa r)^{d+1}\int_{\mathcal C}(\theta\cdot e)^2\dd\sigma(\theta)\ge c r^{d+1}.
\end{aligned}
\end{equation}
For the last inequality, rotate $\mathcal C$ to a fixed cap of the same aperture. The integral of $(\theta\cdot e)^2$ over this cap is a continuous, strictly positive function of $e\in\mathbb S^{d-1}$, since an open spherical cap is not contained in $e^\perp$. Compactness of $\mathbb S^{d-1}$ gives a positive minimum, independent of both $e$ and the orientation of $\mathcal C$. This proves the lower bound.

For the upper bound, \eqref{m:display:003} gives $|(z-\bar z)\cdot e|\le r$ on $\partial K$, so it remains to bound $\HH(\partial K)$. Denote by $\pi_j$ the orthogonal projection onto the hyperplane perpendicular to the $j$th coordinate vector. Cauchy's projection formula \cite[Section~4, equations~(10)--(11)]{HugSchneider2024} gives $\int_{\partial K}|n_j(z)|\dd\HH(z)=2|\pi_jK|$: almost every line parallel to that coordinate vector through the relative interior of $\pi_jK$ meets $\partial K$ in exactly two points. Since $\pi_jK$ lies in a $(d-1)$-dimensional ball of radius $r$, we have $|\pi_jK|\le\omega_{d-1}r^{d-1}$. Using $\sum_{j=1}^d|n_j|\ge1$, we obtain $\HH(\partial K)\le\sum_{j=1}^d\int_{\partial K}|n_j|\dd\HH\le2d\omega_{d-1}r^{d-1}$. Therefore $\int_{\partial K\cap\intr\Omega}((z-\bar z)\cdot e)^2\dd\HH(z)\le r^2\HH(\partial K)\le2d\omega_{d-1}r^{d+1}$, as required.
\end{proof}
\fi

\subsection{Reciprocal affine rescaling}
For local comparisons and rescalings, we extend the QOT functional \eqref{eq:problem} to marginals of equal (finite) positive mass. If their common mass is $m_0$, dividing the marginals, the coupling and the objective by $m_0$ reduces the problem to probability QOT with parameter $\eps/m_0$. In the original finite-mass variables, the density with respect to the product of the marginals still has row and column integrals equal to one, so the marginal equations remain \eqref{eqv:eq:marginal}. %

\begin{lemma}[Affine covariance and rescaling]\label{geom:lem:scaling}
Let $\mu,\nu$ be compactly supported finite measures of equal positive mass, and let $\pi\in\Pi(\mu,\nu)$ be a QOT optimizer at parameter $\eps>0$. Let $L\in\R^{d\times d}$ be invertible, $a,b\in\R^d$ and $R>0$, and define
\begin{equation}\label{m:display:006}
 F(\xi,\eta)=(a+RL\xi,b+R(L^{-1})^T\eta).
\end{equation}
Set $\widehat\pi:=R^{-d}(F^{-1})_\#\pi$, and let $\widehat\mu,\widehat\nu$ be its marginals. Then $\widehat\pi$ solves the QOT problem with marginals $\widehat\mu,\widehat\nu$ and regularization parameter
\begin{equation}\label{geom:eq:epsscale}
 \widehat\eps=\eps/R^{d+2};
\end{equation}
in particular, reciprocal affine changes ($R=1$) preserve the regularization parameter. If $\mu,\nu$ have densities $\rho_0,\rho_1$, the transformed marginal densities are
\begin{equation}\label{m:display:007}
 \widehat\rho_0(\xi)=|\det L|\rho_0(a+RL\xi),\qquad \widehat\rho_1(\eta)=|\det L|^{-1}\rho_1(b+R(L^{-1})^T\eta).
\end{equation}
\end{lemma}
\begin{proof}
We use the change-of-variables argument of \cite[Section~2.4]{GK}, keeping track of the mass normalization. The transformation gives a bijection between the original and transformed admissible couplings. If $h=\dd\pi/\dd(\mu\otimes\nu)$, the transformed density is $\widehat h=R^dh\circ F$, whereas $\widehat\mu\otimes\widehat\nu=R^{-2d}(F^{-1})_\#(\mu\otimes\nu)$. Consequently, $\int\widehat h^2\dd(\widehat\mu\otimes\widehat\nu)=\int h^2\dd(\mu\otimes\nu)$. 
Up to terms depending only on the marginals, the quadratic transport cost equals $-\int x\cdot y\dd\pi$. Since $(L\xi)\cdot((L^{-1})^T\eta)=\xi\cdot\eta$ and $\pi=R^dF_\#\widehat\pi$, this term becomes $-R^{d+2}\int\xi\cdot\eta\dd\widehat\pi$, again up to terms depending only on the marginals. Dividing the objective by $R^{d+2}$ proves \eqref{geom:eq:epsscale}. Finally, the Jacobians of the two coordinate maps are $R^d|\det L|$ and $R^d|\det L|^{-1}$, respectively, which gives \eqref{m:display:007}.
\end{proof}

\section{Boundary localization}\label{sec:boundary}
This section develops the boundary localization estimates that will be used to propagate local control of the regularized coupling down to the scale $\ell=\eps^{1/(d+2)}$. Again, we allow $\mu,\nu$ to have any common finite positive mass. Unless weaker assumptions are stated, the conditions of \cref{ass:geometrydata} are in force. The estimates below depend only on the bounds specified in the assumptions.

\subsection{Rectangular comparison}
\label{geom:sec:replacement}
In this subsection, the penalty is measured relative to the original product reference $\mu\otimes\nu$. More precisely, for submarginals $\alpha\le\mu$ and $\beta\le\nu$ of equal mass, we set
\begin{equation}\label{geom:eq:fixed-reference}
 \mathcal F_\eps(\gamma\mid\mu\otimes\nu)
 :=
 \int c(x,y)\dd\gamma(x,y)
 +\frac{\eps}{2}
 \int\left(\frac{\dd\gamma}{\dd(\mu\otimes\nu)}\right)^2
 \dd(\mu\otimes\nu),
 \qquad \gamma\in\Pi(\alpha,\beta),
\end{equation}
where $c(x,y)=|x-y|^2/2$, with value $+\infty$ if $\gamma\not\ll\mu\otimes\nu$ or its density is not square integrable. Thus restrictions of the plan have their own submarginals, while the reference measure in \eqref{geom:eq:fixed-reference} remains $\mu\otimes\nu$.

\begin{lemma}[Recovery by quantization]\label{geom:lem:recovery}
Let $\alpha\le\mu$ and $\beta\le\nu$ have equal finite mass $p$ and be supported in balls of radius $R$, possibly with different centers. Let $\gamma$ minimize quadratic transport cost between $\alpha$ and $\beta$. For every $0<r\le R$, there exists $\gamma_r\in\Pi(\alpha,\beta)$ such that
\begin{equation}\label{geom:eq:recovery}
\begin{aligned}
 0\le \int |x-y|^2\dd\gamma_r-\int |x-y|^2\dd\gamma
 &\le C_dp r^2,\\
 \int\left(\frac{\dd\gamma_r}{\dd(\mu\otimes\nu)}\right)^2
 \dd(\mu\otimes\nu)
 &\le C_d(R/r)^d.
\end{aligned}
\end{equation}
In particular, $\mathcal F_\eps(\gamma_r\mid\mu\otimes\nu) \le \int c\dd\gamma+C_d\bigl(pr^2+\eps(R/r)^d\bigr)$. The constant $C_d$ depends only on the dimension.
\end{lemma}

\begin{proof}
If $p=0$, take $\gamma_r=0$. Assume henceforth that $p>0$. Optimality for the quadratic cost implies that the support of $\gamma$ is monotone: $(x-x')\cdot(y-y')\ge0$ for any $(x,y),(x',y')\in\spt\gamma$. Consequently,
\begin{equation}\label{m:display:008}
 |x-x'|^2+|y-y'|^2
 \le |(x+y)-(x'+y')|^2.
\end{equation}
Partition $\R^d$ into half-open cubes of side length $r$ and group the points $(x,y)$ of $\spt\gamma$ according to the cube containing $x+y$. Since the two marginals are supported in balls of radius $R$, their sums lie in a ball of radius $2R$. Thus at most $N\le C_d(R/r)^d$ of these cells have positive $\gamma$-mass.

Let $\gamma_j$ be the restriction of $\gamma$ to the $j$th such cell, let $m_j>0$ be its mass, and denote its marginals by $\alpha_j$ and $\beta_j$. Define
$ \gamma_r:=\sum_{j=1}^N
 \frac{\alpha_j\otimes\beta_j}{m_j}.$
Each summand has marginals $\alpha_j$ and $\beta_j$. Since $\sum_j\alpha_j=\alpha$ and $\sum_j\beta_j=\beta$, we have $\gamma_r\in\Pi(\alpha,\beta)$.

By \eqref{m:display:008}, both coordinate projections of each cell intersected with $\spt\gamma$ have diameter at most $\sqrt d\,r$. Writing $\bar x_j:=m_j^{-1}\int x\dd\alpha_j$ and $\bar y_j:=m_j^{-1}\int y\dd\beta_j$, it follows that $|x-\bar x_j|,|y-\bar y_j|\le\sqrt d\,r$ for $\gamma_j$-almost every $(x,y)$. The terms $|x|^2$ and $|y|^2$ have the same integrals under $\gamma_j$ and $\alpha_j\otimes\beta_j/m_j$. Expanding the remaining bilinear term therefore gives
\begin{equation}\label{m:display:010}
 \int |x-y|^2\dd\left(
 \frac{\alpha_j\otimes\beta_j}{m_j}-\gamma_j\right)
 =
 2\int (x-\bar x_j)\cdot(y-\bar y_j)\dd\gamma_j
 \le 2dm_jr^2.
\end{equation}
Summing over $j$ proves the upper bound in the first line of \eqref{geom:eq:recovery}, and the lower bound follows from the optimality of $\gamma$.

For the penalty estimate, we retain the original product reference in \eqref{geom:eq:fixed-reference}. Set $a_j:=\dd\alpha_j/\dd\mu$ and $b_j:=\dd\beta_j/\dd\nu$. Since $\sum_j\alpha_j=\alpha\le\mu$ and $\sum_j\beta_j=\beta\le\nu$, we have $\sum_j a_j\le1$ and $\sum_j b_j\le1$ almost everywhere. Moreover, $\int a_j\dd\mu=\int b_j\dd\nu=m_j$. The density of $\gamma_r$ is $\sum_j a_j(x)b_j(y)/m_j$, so the weighted Cauchy--Schwarz inequality yields
\begin{equation}\label{m:display:011}
 \left(\sum_j\frac{a_j(x)b_j(y)}{m_j}\right)^2
 \le
 \left(\sum_j a_j(x)\right)
 \left(\sum_j\frac{a_j(x)b_j(y)^2}{m_j^2}\right)
 \le \sum_j\frac{a_j(x)b_j(y)^2}{m_j^2}.
\end{equation}
Integrating with respect to $\mu\otimes\nu$ bounds the squared norm by $\sum_j m_j^{-1}\int b_j^2\dd\nu$. Since $0\le b_j\le1$, each term is at most $m_j^{-1}\int b_j\dd\nu=1$. The second bound in \eqref{geom:eq:recovery} follows from $N\le C_d(R/r)^d$. Finally, the estimate for $\mathcal F_\eps$ follows from \eqref{geom:eq:fixed-reference} and $c(x,y)=|x-y|^2/2$.
\end{proof}

\begin{proposition}[Rectangular comparison with fixed marginals]\label{geom:prop:rectangle}
Let $\pi$ be a QOT optimizer at parameter $\eps$, with density $h_\eps:=\dd\pi/\dd(\mu\otimes\nu)$, and suppose that the densities of $\mu$ and $\nu$ are bounded above by $M$. Let $D=K\times L$ be a Borel rectangle, where $K$ and $L$ are contained in balls of radius $R$, and set $\pi_D:=\pi|_D,$ $\alpha:=(\mathrm{pr}_x)_\#\pi_D$
and $ \beta:=(\mathrm{pr}_y)_\#\pi_D.$ 
Let $\gamma\in\Pi(\alpha,\beta)$ minimize quadratic  transport cost. If $R\ge\ell:=\eps^{1/(d+2)}$, then
\begin{equation}\label{geom:eq:rectangle}
 0\le
 \underbrace{\int c\dd\pi_D-\int c\dd\gamma}_{\ge0}
 +\underbrace{\frac{\eps}{2}
 \int_D h_\eps^2\dd(\mu\otimes\nu)}_{\ge0}
 \le C_{d,M}R^d\ell^2. \vspace{-.5em}
\end{equation}
In particular, the transport-cost gap and the penalty each satisfy the upper bound $C_{d,M}R^d\ell^2$.
\end{proposition}

\begin{proof}
The measures $\alpha\le\mu$ and $\beta\le\nu$ have common mass $p:=\pi(D)$. Since $K$ is contained in a ball of radius $R$ and the density of $\mu$ is bounded by $M$, we have $p\le\mu(K)\le M\omega_dR^d$. If $p=0$, the assertion is immediate, so assume $p>0$.

We first explain why comparisons on $D$ preserve the original product reference. Every coupling $\widetilde\gamma\in\Pi(\alpha,\beta)$ is concentrated on $K\times L$, since its marginals are concentrated on $K$ and $L$. Consequently, $\widetilde\pi:=\pi|_{D^c}+\widetilde\gamma$ has marginals $\mu-\alpha+\alpha=\mu$ and $\nu-\beta+\beta=\nu$. If $\widetilde\gamma$ has finite penalty relative to $\mu\otimes\nu$, its density vanishes outside $D$, whereas that of $\pi|_{D^c}$ vanishes on $D$. Their squared densities therefore have no cross-term. Using the functional \eqref{geom:eq:fixed-reference}, optimality of $\pi$ and cancellation of the contribution on $D^c$ give $\mathcal F_\eps(\pi_D\mid\mu\otimes\nu) \le\mathcal F_\eps(\widetilde\gamma\mid\mu\otimes\nu)$.

Apply \cref{geom:lem:recovery} with $r=\ell\le R$ and take $\widetilde\gamma=\gamma_\ell$. The preceding comparison and \eqref{geom:eq:recovery} yield
\begin{equation}\label{m:display:013}
 \int c\dd\pi_D
 +\frac{\eps}{2}\int_D h_\eps^2\dd(\mu\otimes\nu)
 \le
 \int c\dd\gamma
 +C_dp\ell^2+C_d\eps(R/\ell)^d.
\end{equation}
Since $p\le M\omega_dR^d$ and $\eps=\ell^{d+2}$, the two error terms are bounded by $C_{d,M}R^d\ell^2$. Finally, $\pi_D\in\Pi(\alpha,\beta)$, so optimality of $\gamma$ gives $\int c\dd\pi_D-\int c\dd\gamma\ge0$. The penalty is also nonnegative, proving \eqref{geom:eq:rectangle} and the separate bounds.
\end{proof}

\subsection{Boundary support localization}
\label{geom:sec:long}

The following estimate converts small local excess into a displacement bound for every support pair with an endpoint in a smaller neighborhood. At a boundary point, the common supporting half-space permits the construction of a comparison ball carrying a definite amount of nearby mass. For a transport plan $\pi$ and $R>0$, we set
\begin{equation}\label{geom:eq:excess}
 \cE(\pi,R)=R^{-d-2}\int_{\#B_R}|x-y|^2\dd\pi,\qquad \#B_R:=(B_R\times\R^d)\cup(\R^d\times B_R).
\end{equation}

\begin{lemma}[Local support bound at a convex boundary]\label{geom:lem:long}
Let $\pi_\eps$ be the QOT optimizer at parameter $\eps>0$, and recall $\ell=\eps^{1/(d+2)}$. Suppose that $X\cap B_1=B_1\cap\{(z',z_d):z_d\ge P(z')\}$ and $Y\cap B_1=B_1\cap\{(z',z_d):z_d\ge Q(z')\}$, where $P,Q\in C^2(B'_1)$, $B'_1:=\{z'\in\R^{d-1}:|z'|<1\}$, satisfy
\begin{equation}\label{m:display:014}
 P(0)=Q(0)=0,\qquad DP(0)=DQ(0)=0,\qquad
 \|P\|_{C^2(B'_1)}+\|Q\|_{C^2(B'_1)}\le\kappa.
\end{equation}
Assume $m\le\rho_0\le M$ on $X\cap B_1$ and $m\le\rho_1\le M$ on $Y\cap B_1$. There exist $C,\delta>0$, depending only on $d,m,M$, such that $\cE(\pi_\eps,1)+\kappa+\ell\le\delta$ implies
\begin{equation}\label{geom:eq:long}
 \sup_{(x,y)\in\spt\pi_\eps\cap\#B_{1/2}}|x-y|
 \le C\bigl(\cE(\pi_\eps,1)^{1/(d+2)}+\kappa+\ell\bigr).
\end{equation}
For an unregularized optimal plan $\pi_0$, the same estimate holds with $\pi_\eps$ replaced by $\pi_0$ and the $\ell$-term omitted, provided $\cE(\pi_0,1)+\kappa\le\delta$. In this case, the local graph descriptions, the local density bounds, and $X\cup Y\subset\{x_d\ge0\}$ suffice as assumptions.

If the graph and half-space assumptions are replaced by $B_1\subset X\cap Y$, the same estimate holds with the $\kappa$-term omitted, provided $\cE(\pi_\eps,1)+\ell\le\delta$. For unregularized optimal transport, the $\ell$-term is also omitted, and only the local density bounds on $B_1$ are required.
\end{lemma}

\begin{proof}
Fix $(x,y)\in\spt\pi_\eps$ with $x\in B_{1/2}$, and set $r:=|x-y|$ and $s:=\min\{r,1/8\}$. Convexity and the tangent-plane assumptions give $X\cup Y\subset\{x_d\ge0\}$. It suffices to consider $r>16\kappa$. Taking $\delta\le1/128$ gives $s\ge16\kappa$. Set $\zeta:=x+\frac{s}{2r}(y-x)+\frac{s}{8}e_d$, where $e_d$ is the last coordinate vector. Since $s/(2r)\le1/2$ and $x_d,y_d\ge0$, $\zeta_d\ge s/8$. For $z\in B(\zeta,s/64)$, we have $z_d>7s/64\ge\kappa\ge P(z')$ and $|z|<1/2+41s/64<1$. Hence $B(\zeta,s/64)\subset X\cap B_1$.

Using $s\le r$ and $|y_d-x_d|\le r$, we obtain
\begin{equation}\label{m:display:016}
\begin{aligned}
 (\zeta-x)\cdot(y-\zeta)
 &=
 \frac{sr}{2}\left(1-\frac{s}{2r}\right)
 +\frac{s}{8}\left(1-\frac{s}{r}\right)(y_d-x_d)
 -\frac{s^2}{64}\ge \frac{7}{64}sr.
\end{aligned}
\end{equation}
For $x'\in B(\zeta,s/64)$ and $|y'-x'|\le s/64$, the bounds $|\zeta-x|\le5s/8$ and \eqref{m:display:016}  give
\begin{equation}\label{geom:eq:crossgeom}
\begin{aligned}
 \big|(x'-x)\cdot(y-y')-(\zeta-x)\cdot(y-\zeta)\big|
 &\le \frac{s}{64}
       \left(r+\frac{15s}{8}+\frac{s}{32}\right)
 \le \frac{sr}{16},\\
 (x'-x)\cdot(y-y')&\ge\frac{sr}{32}.
\end{aligned}
\end{equation}

Let $G:=\{(x',y')\in\spt\pi_\eps:x'\in B(\zeta,s/64),\ |y'-x'|\le s/64\}$. Since $B(\zeta,s/64)\subset X\cap B_1$,
\[
 \mu(B(\zeta,s/64))=\int_{B(\zeta,s/64)}\rho_0\dd x\ge m\omega_d(s/64)^d.
\]
Since $\mu$ is the first marginal of $\pi_\eps$,
\[
\begin{aligned}
 \mu(B(\zeta,s/64))-\pi_\eps(G)
 &=\pi_\eps\{(x',y'):x'\in B(\zeta,s/64),\ |x'-y'|>s/64\}\\
 &\le \frac{64^2}{s^2}
       \int_{B(\zeta,s/64)\times\R^d}|x'-y'|^2\dd\pi_\eps(x',y')\le \frac{64^2}{s^2}\cE(\pi_\eps,1),
\end{aligned}
\]
where the last inequality follows from $B(\zeta,s/64)\times\R^d\subset\#B_1$. Hence
\begin{equation}
    \label{eq:mass-estimates-local-support}
    m\omega_d(s/64)^d
\le\pi_\eps(G)+64^2s^{-2}\cE(\pi_\eps,1).
\end{equation}

Continuity of $q_\eps$ and $\dd\pi_\eps/\dd(\mu\otimes\nu)=(q_\eps)_+/\eps$ give $q_\eps\ge0$ on $\spt\pi_\eps$. Since $q_\eps(x,y)=x\cdot y-u_\eps(x)-v_\eps(y)$, cancellation of the potential terms gives
\begin{equation}\label{m:display:017}
\begin{aligned}
 &q_\eps(x',y)+q_\eps(x,y')
 -q_\eps(x,y)-q_\eps(x',y')\\
 &\qquad=x'\cdot y+x\cdot y'-x\cdot y-x'\cdot y'
 =(x'-x)\cdot(y-y').
\end{aligned}
\end{equation}
The marginals of $\pi_\eps|_G$ are dominated by $\mu,\nu$. Thus \eqref{geom:eq:crossgeom}, \eqref{m:display:017}, and \eqref{eqv:eq:marginal} give
\begin{equation}\label{geom:eq:longcontradiction}
\begin{aligned}
 \frac{rs}{32}\pi_\eps(G)
 &\le
 \int_G\bigl((q_\eps(x',y))_+
             +(q_\eps(x,y'))_+\bigr)\dd\pi_\eps(x',y')\\
 &\le
 \int_X(q_\eps(x',y))_+\dd\mu(x')
 +\int_Y(q_\eps(x,y'))_+\dd\nu(y')
 =2\eps.
\end{aligned}
\end{equation}
By \eqref{geom:eq:longcontradiction}, $\pi_\eps(G)\le64\eps/(rs)$. Substituting this into \eqref{eq:mass-estimates-local-support} and multiplying by $s^2$ gives
\[
 \frac{m\omega_d}{64^d}s^{d+2}
 \le 64\eps\frac{s}{r}+64^2\cE(\pi_\eps,1)
 \le 64^2\bigl(\cE(\pi_\eps,1)+\eps\bigr),
\]
where the last inequality uses $s\le r$. Thus $s^{d+2}\le C\bigl(\cE(\pi_\eps,1)+\eps\bigr)$. If $r\ge1/8$, this gives $8^{-(d+2)}\le C(\delta+\delta^{d+2})$, which is excluded by taking $\delta$ sufficiently small. Hence $s=r$ and $r\le C\bigl(\cE(\pi_\eps,1)^{1/(d+2)}+\ell\bigr)$. Together with the case $r\le16\kappa$, this proves \eqref{geom:eq:long} for $x\in B_{1/2}$. Interchanging the marginals gives the estimate on $\#B_{1/2}$.

For unregularized optimal transport, define $G$ using $\pi_0$ in place of $\pi_\eps$. Monotonicity gives $(x'-x)\cdot(y-y')\le0$ for every $(x',y')\in\spt\pi_0$. By \eqref{geom:eq:crossgeom}, $G=\varnothing$. The mass estimate therefore gives $s^{d+2}\le C\cE(\pi_0,1)$, and the preceding argument applies with the $\ell$-term omitted. Only the local graph descriptions, the local lower density bounds, and the global half-space condition were used.

For the interior versions, omit the vertical shift and use $\zeta:=x+\frac{s}{2r}(y-x)$. Then $B(\zeta,s/64)\subset B_1\subset X$ and $(\zeta-x)\cdot(y-\zeta)\ge sr/4$. The same comparison and mass estimates apply without the $\kappa$-term, both for regularized and unregularized optimal transport.
\end{proof}

\subsection{Local regularity of the unregularized comparison maps}
\label{geom:sec:classical}

We next consider an unregularized optimal plan whose marginals agree with $\mu,\nu$ near the origin. We prove a local regularity estimate controlled by the square root of its excess and by the variation of the boundary graphs and densities.

\begin{lemma}[Boundary estimates for the comparison maps]
\label{geom:lem:classical}
Suppose that $X,Y\subset\{x_d\ge0\}$, that $X\cap B_1$ and $Y\cap B_1$ are convex, and that $X\cap B_1=B_1\cap\{(z',z_d):z_d\ge P(z')\}$, $Y\cap B_1=B_1\cap\{(z',z_d):z_d\ge Q(z')\}$, where $P,Q\in C^{3,\alpha}(B'_1)$ satisfy
\begin{equation}\label{m:display:018}
 P(0)=Q(0)=0,\qquad DP(0)=DQ(0)=0.
\end{equation}
Assume $m\le\rho_0\le M$ on $X\cap B_1$, $m\le\rho_1\le M$ on $Y\cap B_1$, $\rho_0(0)=\rho_1(0)$, and
\begin{equation}\label{geom:eq:smallclassicaldata}
 \|P\|_{C^{3,\alpha}(B'_1)}+\|Q\|_{C^{3,\alpha}(B'_1)}
 +\|\rho_0-\rho_0(0)\|_{C^{1,\alpha}(X\cap B_1)}
 +\|\rho_1-\rho_1(0)\|_{C^{1,\alpha}(Y\cap B_1)}
 \le\kappa.
\end{equation}
Let $\alpha_0\le\mu$ and $\alpha_1\le\nu$ have equal mass, with $\alpha_0|_{B_{1/2}}=\mu|_{B_{1/2}}$, $\alpha_1|_{B_{1/2}}=\nu|_{B_{1/2}}$, and $\spt\alpha_0\cup\spt\alpha_1\subset\overline B_{3/4}$. Let $\gamma\in\Pi(\alpha_0,\alpha_1)$ minimize quadratic transport, with a.e.-defined maps $\mathcal T=\nabla\phi$ and $\mathcal T^{-1}=\nabla\psi$. Then
\begin{equation}\label{m:display:019}
 \cE(\gamma,1)=\int|x-y|^2\dd\gamma
 =\int|\mathcal T(x)-x|^2\dd\alpha_0(x).
\end{equation}

There exist $r_c,\delta>0$, $\beta\in(0,\alpha)$, and $C<\infty$, depending only on $d,m,M,\alpha$, such that $\cE(\gamma,1)+\kappa^2\le\delta$ implies that the maps extend continuously to the boundary near the origin and
\begin{equation}\label{geom:eq:linearclassical}
 \|\mathcal T-\id\|_{C^{1,\beta}(X\cap B_{4r_c})}
 +\|\mathcal T^{-1}-\id\|_{C^{1,\beta}(Y\cap B_{4r_c})}
 \le C\bigl(\cE(\gamma,1)^{1/2}+\kappa\bigr).
\end{equation}
They are mutually inverse diffeomorphisms on corresponding neighborhoods and map the respective boundary portions onto each other. In particular, $b:=\mathcal T(0)\in\partial Y$ and $H:=D\mathcal T(0)$ satisfy
\begin{equation}\label{geom:eq:jet}
 |b|+\|H-\Id\|
 \le C\bigl(\cE(\gamma,1)^{1/2}+\kappa\bigr),
 \qquad H=H^T\succ0,\qquad
 \det H=\frac{\rho_0(0)}{\rho_1(b)}.
\end{equation}
Choosing $\delta$ smaller so that $|b|<r_c$, we also have
\begin{equation}\label{geom:eq:taylorclassical}
\begin{aligned}
 |\mathcal T(x)-b-Hx|
 &\le C\bigl(\cE(\gamma,1)^{1/2}+\kappa\bigr)
          |x|^{1+\beta},
 &&x\in X\cap B_{r_c},\\
 |\mathcal T^{-1}(y)-H^{-1}(y-b)|
 &\le C\bigl(\cE(\gamma,1)^{1/2}+\kappa\bigr)
          |y-b|^{1+\beta},
 &&y\in Y\cap B(b,r_c).
\end{aligned}
\end{equation}

The interior version holds with $B_1\subset X\cap Y$, without the graph or half-space assumptions and with the graph terms omitted from \eqref{geom:eq:smallclassicaldata}. In that case $b\in\intr Y$. The condition $\rho_0(0)=\rho_1(0)$ may be omitted if $\kappa$ is replaced throughout by $\kappa+|\rho_0(0)-\rho_1(0)|$.
\end{lemma}

\begin{proof}
By construction of $\gamma$,  \eqref{m:display:019} holds. We first prove the boundary assertion. Choose lower semicontinuous conjugate convex transport potentials $\phi,\psi$. The unregularized part of \cref{geom:lem:long}, applied to $\gamma$ at radius $1/2$, gives
\begin{equation}\label{geom:eq:classicalrough}
 \|\mathcal T-\id\|_{L^\infty(X\cap B_{1/4})}+\|\mathcal T^{-1}-\id\|_{L^\infty(Y\cap B_{1/4})}
 \le C\bigl(\cE(\gamma,1)^{1/(d+2)}+\kappa\bigr).
\end{equation}
Fix $s=1/64$ and let $e_*$ denote the right-hand side of \eqref{geom:eq:classicalrough}. Decrease $\delta$ so that $e_*\le s/10$. By positivity of the local densities, the potentials are finite on the interiors of the corresponding convex caps, where their differences from $|\cdot|^2/2$ have weak gradients bounded by $e_*$. These differences therefore have $e_*$-Lipschitz continuous extensions to the caps. Normalize $\phi(0)=0$, adjusting $\psi$ by the opposite constant, and write $W:=\phi-|\cdot|^2/2$. Then
\begin{equation}\label{geom:eq:sectionlocalization-value}
 \Big|\phi(x)-\frac{|x|^2}{2}\Big|\le e_*|x|,\qquad x\in X\cap\overline B_{4s}.
\end{equation}

Set $D:=X\cap\overline B_{4s}$, $U:=\intr Y\cap B_{2s}$, and $h:=s^2/2$. Define
\[
 \Psi(y):=\max_{x\in D}\{x\cdot y-\phi(x)\},\qquad
 \mathcal M(y):=\argmax_{x\in D}\{x\cdot y-\phi(x)\}.
\]
The maximizing sets are nonempty and compact, and $\Psi$ is finite, convex, and Lipschitz on $\R^d$. For almost every $y\in U$, \eqref{geom:eq:classicalrough} places $\mathcal T^{-1}(y)$ in $D$, and dual equality makes it a maximizer. Every element of $\mathcal M(y)$ belongs to $\partial\Psi(y)$. Hence, at almost every point of differentiability of $\Psi$ in $U$,
$\mathcal M(y)=\{\nabla\Psi(y)\}=\{\mathcal T^{-1}(y)\}.$ 
These statements hold Lebesgue-almost everywhere because $\alpha_1=\rho_1\dd y$ on $U$ and $\rho_1\ge m$ there. Thus $|\nabla\Psi(y)-y|\le e_*$ almost everywhere on $U$, and $\Psi-|\cdot|^2/2$ is $e_*$-Lipschitz on this convex set. For $y\in U$, $x\in\mathcal M(y)$, and a unit vector $v$, the subgradient inequality at $y\pm tv\in U$ gives $|(x-y)\cdot v|\le e_*+t/2$. Letting $t\downarrow0$ yields
\begin{equation}\label{geom:eq:sectionlocalization-maximizers}
 |x-y|\le e_*,\qquad y\in U,\quad x\in\mathcal M(y).
\end{equation}

We restrict the source by a sublevel set and define its target pointwise:
\begin{equation}\label{m:display:020}
\begin{aligned}
 C_0&:=\{x\in\intr X\cap B_{4s}:\phi(x)<h\},\quad
 C_1:=\{y\in U:\max_{x\in\mathcal M(y)}\phi(x)<h\}.
\end{aligned}
\end{equation}
The set $C_0$ is open and convex. The function $y\mapsto\max_{x\in\mathcal M(y)}\phi(x)$ is upper semicontinuous: if $y_j\to y$, maximizers realizing a limsup have a subsequence converging in $D$; its limit belongs to $\mathcal M(y)$, and continuity of $\phi$ gives the assertion. Thus $C_1$ is open. By \eqref{geom:eq:sectionlocalization-value} and $e_*\le s/10$, we have $\phi(x)<h$ on $D\cap B_{4s/5}$ and $\phi(x)>h$ on $D\setminus B_{6s/5}$. Together with \eqref{geom:eq:sectionlocalization-maximizers}, this gives
\begin{equation}\label{geom:eq:classicalimage}
\begin{aligned}
 \intr X\cap B_{4s/5}&\subset C_0\subset\intr X\cap B_{6s/5},\quad 
 \intr Y\cap B_{7s/10}\subset C_1\subset\intr Y\cap B_{13s/10}.
\end{aligned}
\end{equation}

The set $C_1$ is also connected. Indeed, fix $y\in C_1$ and $0<t<1$. Convexity of $Y\cap B_1$ and $0\in\partial Y$ give $ty\in U$. For $x\in\mathcal M(y)$ and $z\in\mathcal M(ty)$, the two maximizing inequalities imply
\[
 y\cdot(z-x)\le\phi(z)-\phi(x)\le t\,y\cdot(z-x).
\]
Hence $y\cdot(z-x)\le0$ and $\phi(z)\le\phi(x)<h$. Since this holds for every $z\in\mathcal M(ty)$, we have $ty\in C_1$. Every point of $C_1$ can therefore be joined within $C_1$ to the convex inner cap $\intr Y\cap B_{7s/10}$, proving path connectedness.

For almost every $y\in U$, the singleton identity above gives
\[
 y\in C_1\quad\Longleftrightarrow\quad\phi(\mathcal T^{-1}(y))<h\quad\Longleftrightarrow\quad\mathcal T^{-1}(y)\in C_0.
\]
Here $\mathcal T^{-1}(y)\in\intr X$ almost everywhere, since the inverse pushes $\alpha_1$ to $\alpha_0$ and $\alpha_0(\partial X\cap B_{4s})=0$. The forward bound in \eqref{geom:eq:classicalrough} and \eqref{geom:eq:classicalimage} also give $\gamma(C_0\times U^c)=0$. Consequently, $\gamma|_{C_0\times\R^d}$ has marginals $\rho_0\mathbf1_{C_0}\dd x$ and $\rho_1\mathbf1_{C_1}\dd y$. The restriction $\gamma|_{C_0\times\R^d}$ is optimal
between its marginals, since its support is contained
in the cyclically monotone set $\spt\gamma$. Moreover, $\phi(x)=\sup_{y\in C_1}\{x\cdot y-\Psi(y)\},$ for  $ x\in C_0.$ 
The right-hand side is at most $\phi$ by definition of~$\Psi$. Equality holds at almost every source point by dual equality and the restricted marginal identity, and then everywhere by continuity; the supremum is Lipschitz because $C_1$ is bounded. Thus $\phi$ is a $c$-convex transport potential for the restricted problem on the two domains $C_0,C_1$, with $c(x,y)=-x\cdot y$. Finally, \eqref{geom:eq:sectionlocalization-value} and \eqref{geom:eq:classicalimage} give
\begin{equation}\label{geom:eq:potentialcloseclassical}
 \|\phi-|\cdot|^2/2\|_{L^\infty(C_0)}\le Cs e_*\le Cs\bigl(\cE(\gamma,1)^{1/(d+2)}+\kappa\bigr).
\end{equation}

For dimension $d\ge2$, we apply \cite[Theorems~2.1--2.2]{CF} after the change of variables $(x,y)=s(\xi,\eta)$. Some preparations are needed to fit the setting in \cite{CF}. Set $\widehat C_i:=s^{-1}C_i$ and
\[
 \phi_s(\xi):=s^{-2}\phi(s\xi),\qquad
 f_s(\xi):=\frac{\rho_0(s\xi)}{\rho_0(0)},\qquad
 g_s(\eta):=\frac{\rho_1(s\eta)}{\rho_0(0)}.
\]
The restricted marginal identity becomes
\[
 (\nabla\phi_s)_\#(f_s\mathbf1_{\widehat C_0}\dd\xi)=g_s\mathbf1_{\widehat C_1}\dd\eta.
\]
Multiply the rescaled functions $s^{-1}P(s\cdot)$ and $s^{-1}Q(s\cdot)$ by a fixed smooth cutoff equal to one on $B'_2$ and compactly supported in $B'_3$, and extend them by zero. Denote the resulting global functions by $P_s,Q_s$. They satisfy $P_s(0)=Q_s(0)=0$, $DP_s(0)=DQ_s(0)=0$, and $\|P_s\|_{C^2}+\|Q_s\|_{C^2}\le C\kappa$. The inclusions \eqref{geom:eq:classicalimage} therefore give the graph inclusions in \cite[equation~(2.6)]{CF} for $\widehat C_0,\widehat C_1$.

The identity $\rho_0(0)=\rho_1(0)\ge m$ and \eqref{geom:eq:smallclassicaldata} give
\[
 \|f_s-1\|_{L^\infty(\widehat C_0)}+\|g_s-1\|_{L^\infty(\widehat C_1)}\le C\kappa,
\]
with uniformly bounded $C^\alpha$ norms on the corresponding closed graph caps in $B_{1/2}$. The rescaled cost is $c_s(\xi,\eta)=-\xi\cdot\eta$, so the smoothness, twist, and nondegeneracy assumptions hold with zero cost perturbation. The supremum representation above gives $c_s$-convexity of $\phi_s$. Since $C_0$ is open, the supporting-plane inequality and the Lipschitz bound for $W$ give $|p-x|\le e_*$ for every $x\in C_0$ and every slope $p$ supporting $\phi$ there. Hence $\partial_{c_s}\phi_s(\widehat C_0)\subset B_{13/10}\subset B_2$. Finally,
\[
 \|\phi_s-|\cdot|^2/2\|_{L^\infty(\widehat C_0)}\le Ce_*/s.
\]
Since $s$ is fixed, all the required perturbation norms are sufficiently small after decreasing $\delta$. Thus \cite[Theorem~2.2 and Step~3 of its proof]{CF} gives a uniform local $C^{2,\beta}$ bound for $\phi$ up to the boundary, for some $\beta\in(0,\alpha)$, with radius and bound depending only on $d,m,M,\alpha$. Repeating the localization with the marginals interchanged, and subtracting $\psi(0)$ when normalizing the inverse potential, gives the same bound for $\psi$. By \cite[Remark~4.4]{CF}, their gradients send the respective boundary portions into each other near zero.

For $d=1$, the same estimates hold for any fixed $\beta\in(0,\alpha)$ by the increasing rearrangement. The cumulative distribution functions, positivity of the densities, and \eqref{geom:eq:classicalrough} give
\[
 \mathcal T'=\frac{\rho_0}{\rho_1\circ\mathcal T},\qquad
 (\mathcal T^{-1})'=\frac{\rho_1}{\rho_0\circ\mathcal T^{-1}}
\]
on the intervals under consideration. Both supports start at zero, so $\mathcal T(0)=\mathcal T^{-1}(0)=0$, giving boundary correspondence.

In either case, choose $r_c>0$ so that the local $C^{2,\beta}$
estimates, modulo additive constants, hold on
$X\cap B_{32r_c}$ and $Y\cap B_{32r_c}$, and the boundary
correspondence holds on
$\partial X\cap B_{32r_c}$ and $\partial Y\cap B_{32r_c}$.
Require also $32r_c<1/4$, and decrease $\delta$ so that
$e_*<r_c$. The almost-everywhere inverse identities then extend by continuity on the smaller neighborhoods. The transport equations give
\[
 \det D^2\phi(x)=\frac{\rho_0(x)}{\rho_1(\nabla\phi(x))},\qquad
 \det D^2\psi(y)=\frac{\rho_1(y)}{\rho_0(\nabla\psi(y))};
\]
see \cite[equation~(2.3)]{CF}. These identities extend to the boundary by continuity. The Hessians are nonnegative, and their uniform upper bounds and determinant lower bound $m/M$ give uniform positive lower bounds for their eigenvalues. Consequently,
\begin{equation}\label{eq:replacementabsolute}
\begin{gathered}
 C^{-1}\Id\preceq D^2\phi,D^2\psi\preceq C\Id,\qquad
 \|D^2\phi\|_{C^\beta(X\cap B_{16r_c})}+\|D^2\psi\|_{C^\beta(Y\cap B_{16r_c})}\le C,
\end{gathered}
\end{equation}
where each matrix inequality is understood on its respective cap. Thus the maps are mutually inverse diffeomorphisms on corresponding neighborhoods, up to the boundary.

We now estimate the dependence on the excess. In $\intr X\cap B_{8r_c}$, linearization of $\log\det D^2\phi$ gives
\begin{equation}\label{geom:eq:linearizedMA}
\begin{aligned}
 a^{ij}W_{ij}&=F,\\
 a(x)&:=\int_0^1(\Id+tD^2W(x))^{-1}\dd t,\\
 F(x)&:=\log\rho_0(x)-\log\rho_1(x+\nabla W(x)).
\end{aligned}
\end{equation}
By \eqref{eq:replacementabsolute}, $C^{-1}\Id\preceq a\preceq C\Id$ and $\|a\|_{C^\beta}\le C$. Since $\rho_0(0)=\rho_1(0)$, \eqref{geom:eq:smallclassicaldata} also gives
\begin{equation}\label{m:display:023}
 \|F\|_{C^\beta(X\cap B_{8r_c})}\le C\kappa.
\end{equation}

On $\partial X\cap B_{8r_c}$, boundary correspondence means $x_d+W_d=Q(x'+\nabla'W)$. Since $x_d=P(x')$, the fundamental theorem of calculus yields
\begin{equation}\label{geom:eq:oblique}
 \mathfrak b(x)\cdot\nabla W(x)=Q(x')-P(x'), \qquad
 \mathfrak b(x):=\left(-\int_0^1DQ(x'+t\nabla'W(x))\dd t,\ 1\right).
\end{equation}
The data and absolute estimates give $\|\mathfrak b-e_d\|_{C^{1,\beta}}\le C\kappa$ and $\|Q-P\|_{C^{1,\beta}}\le C\kappa$. In particular,
$ \mathfrak b\cdot\frac{(-DP,1)}{\sqrt{1+|DP|^2}}\ge\frac12$ 
for $\delta$ sufficiently small. For $d=1$, the tangential terms are absent and the boundary condition is simply $W'(0)=0$.

Let $\overline W$ be the mean of $W$ on $X\cap B_{8r_c}$. The small graph norm gives the uniform volume bounds required by \cref{geom:lem:localoblique} on these convex caps. Applying that lemma to \eqref{geom:eq:linearizedMA} and \eqref{geom:eq:oblique},
\begin{equation}\label{geom:eq:localschauder}
 \|W-\overline W\|_{C^{2,\beta}(X\cap B_{4r_c})}
 \le C\left(\|W-\overline W\|_{L^2(X\cap B_{8r_c})}+\kappa\right).
\end{equation}
Poincar\'e's inequality on $X\cap B_{8r_c}$, the lower density bound, and $\alpha_0|_{B_{1/2}}=\mu|_{B_{1/2}}$ give
\begin{equation}\label{m:display:024}
\begin{aligned}
 \|W-\overline W\|_{L^2(X\cap B_{8r_c})}
 &\le C\|\nabla W\|_{L^2(X\cap B_{8r_c})}\le C\left(\int|\mathcal T(x)-x|^2\dd\alpha_0(x)\right)^{1/2}
 =C\cE(\gamma,1)^{1/2}.
\end{aligned}
\end{equation}
This proves the forward estimate in \eqref{geom:eq:linearclassical}, and  interchanging the marginals proves the inverse.

Set $b:=\mathcal T(0)$ and $H:=D\mathcal T(0)$. Boundary correspondence gives $b\in\partial Y$. The bounds and positivity in \eqref{geom:eq:jet} follow from \eqref{geom:eq:linearclassical} and \eqref{eq:replacementabsolute}, while the transport equation at zero gives the determinant identity. Decreasing $\delta$ ensures $|b|<r_c$. Since $D\mathcal T^{-1}(b)=H^{-1}$, integration of the derivatives along $[0,x]$ and $[b,y]$ gives \eqref{geom:eq:taylorclassical}. These segments lie in the respective convex caps, and $B(b,r_c)\subset B_{2r_c}$.

For the interior assertion, use the interior version of \cref{geom:lem:long} and retain the original potential on $B_{3s}$. Take $C_0:=\overline B_s$ and $C_1:=\partial\phi(C_0)$, with subdifferentials computed for the original potential (not its restriction to $C_0$). The supporting-plane argument and the rough gradient estimate give $|p-x|\le e_*\le s/10$ for every $x\in\overline B_s$ and $p\in\partial\phi(x)$. Thus $C_1$ is compact and $C_1\subset\overline B_{11s/10}$. The inverse estimate gives $B_{9s/10}\subset C_1$ up to null sets, and hence everywhere by closedness.

Almost-everywhere differentiability of the conjugate and the inverse identities imply that $C_1$ agrees, up to $\alpha_1$-null sets, with $\{y:\mathcal T^{-1}(y)\in B_s\}$; possible preimages on $\partial B_s$ have zero $\alpha_0$-mass. The restricted marginals are therefore $\rho_0\mathbf1_{B_s}\dd x$ and $\rho_1\mathbf1_{C_1}\dd y$. After the same rescaling and common density normalization, both closed sets contain $B_{1/3}$ and are contained in $B_3$, the density errors are bounded by $C\kappa$, and the cost perturbation vanishes. Moreover, the original potential satisfies
$ \|s^{-2}\phi(s\cdot)-|\cdot|^2/2\|_{L^\infty(B_3)}\le3e_*/s.$
For $d\ge2$, the closed-set localization and regularity theorems \cite[Section~3 and Theorems~4.3, 5.3]{DPF} give uniform absolute interior estimates for $\phi$ and, by symmetry, for $\psi$. For $d=1$, simply use the scalar transport identities above. The preceding linearization and the interior version of \cref{geom:lem:localoblique} then yield \eqref{geom:eq:linearclassical}. The inverse estimate, the conclusions for $b=\mathcal T(0)$ and $H=D\mathcal T(0)$ in \eqref{geom:eq:jet}, and the Taylor estimates follow as above, with $b\in\intr Y$.

Finally, without $\rho_0(0)=\rho_1(0)$, set $\Delta:=|\rho_0(0)-\rho_1(0)|$. The normalized density errors and the bound in \eqref{m:display:023} increase by at most $C\Delta$. Every step remains valid with $\kappa$ replaced by $\kappa+\Delta$, proving the final assertion.
\end{proof}

We record the local oblique estimate used in the preceding proof.

\begin{lemma}[Local Schauder estimate for an oblique boundary problem]\label{geom:lem:localoblique}
Fix $0<r<R$ and $\beta\in(0,1)$. Let $\Omega$ be a convex domain meeting $B_r$, with $C^{2,\beta}$ graph boundary in $B_R$, and set $D_t:=\Omega\cap B_t$. Assume that, for some $c_0,r_*>0$, $|D_t\cap B_\rho(z)|\ge c_0\rho^d$ whenever $r\le t\le R$, $z\in\overline D_t$, and $0<\rho\le r_*$. Let $w\in C^{2,\beta}(\overline D_R)$ satisfy
\begin{equation}\label{m:display:109}
 a^{ij}w_{ij}=F\quad\text{in }D_R,
 \qquad B\cdot Dw=g\quad\text{on }\partial\Omega\cap B_R,
\end{equation}
where $a\in C^\beta(\overline D_R)$ is uniformly elliptic, $B\in C^{1,\beta}(\partial\Omega\cap B_R)$, and $B\cdot n_{\rm in}\ge b_0>0$, with $n_{\rm in}$ the inward unit normal. Then
\begin{equation}\label{geom:eq:localLtwoappendix}
 \|w\|_{C^{2,\beta}(\overline D_r)}
 \le C\left(
 \|w\|_{L^2(D_R)}
 +\|F\|_{C^\beta(\overline D_R)}
 +\|g\|_{C^{1,\beta}(\partial\Omega\cap B_R)}
 \right).
\end{equation}
The constant depends only on $d,\beta,r,R,c_0,r_*$, the boundary chart bounds, the coefficient norms, and the ellipticity and obliqueness constants. In particular, for $\bar w:=|D_R|^{-1}\int_{D_R}w(x)\dd x$,
\begin{equation}\label{m:display:110}
 \|w-\bar w\|_{C^{2,\beta}(\overline D_r)}
 \le C\left(
 \|Dw\|_{L^2(D_R)}
 +\|F\|_{C^\beta(\overline D_R)}
 +\|g\|_{C^{1,\beta}(\partial\Omega\cap B_R)}
 \right).
\end{equation}
For interior balls, the same estimates hold without the terms involving $g$.
\end{lemma}

\begin{proof}
Set $K:=\|F\|_{C^\beta(\overline D_R)}+\|g\|_{C^{1,\beta}(\partial\Omega\cap B_R)}$ and $A(t):=\|w\|_{C^{2,\beta}(\overline D_t)}$. The local oblique Schauder estimate \cite[Theorem~6.30 and Section~6.7]{GT} gives, for $r\le t<s<R$ and some fixed $p>0$,
\[
 A(t)\le C(s-t)^{-p}\bigl(\|w\|_{L^\infty(D_s)}+K\bigr).
\]
Since $\dist(\overline D_t,\partial B_s)\ge s-t>0$,
the local estimates use only the boundary condition
on $\partial\Omega\cap B_s$;
no boundary condition on $\partial B_s$ is required. Averaging over $D_s\cap B_\tau(z)$ and using convexity and the volume lower bound gives, for $0<\tau\le r_*$,
\[
 \|w\|_{L^\infty(D_s)}\le \tau A(s)+C\tau^{-d/2}\|w\|_{L^2(D_R)}.
\]
For any sufficiently small $\eta>0$, choose $\tau$ proportional to $\eta(s-t)^p$. Substitution yields, for some $q>0$,
\[
 A(t)\le\eta A(s)+C_\eta(s-t)^{-q}\bigl(\|w\|_{L^2(D_R)}+K\bigr).
\]
Iterate over radii increasing from $r$ to $R':=(r+R)/2$, with successive gaps proportional to $2^{-j}$. Choosing $\eta<2^{-q}$ makes the resulting series converge; the terminal term tends to zero because $A(R')<\infty$. This proves \eqref{geom:eq:localLtwoappendix}.

The function $w-\bar w$ satisfies the same equation and boundary condition. Since $D_R$ is convex, Poincar\'e's inequality gives $\|w-\bar w\|_{L^2(D_R)}\le CR\|Dw\|_{L^2(D_R)}$. Applying \eqref{geom:eq:localLtwoappendix} to $w-\bar w$ proves \eqref{m:display:110}. The interior proof is identical, using the interior Schauder estimate \cite[Section~6.1]{GT} and omitting the boundary terms.
\end{proof}

\subsection{Excess improvement}
\label{geom:sec:improvement}

Assume the hypotheses on $X,Y,\rho_0,\rho_1$ in \cref{geom:lem:classical}, including $\rho_0(0)=\rho_1(0)$ and \eqref{geom:eq:smallclassicaldata}, and assume in the boundary case that $X,Y$ are globally convex. Let $\pi$ be a QOT optimizer whose marginals have equal finite positive mass and densities bounded above by $M$ on their entire supports. Write $E:=\cE(\pi,1)$ and recall $\ell=\eps^{1/(d+2)}$.

\begin{proposition}[Boundary excess improvement]
\label{geom:prop:improvement}
There exist $\beta\in(0,\alpha)$ and $C,\theta_0>0$, depending only on $d,m,M,\alpha$, with the following property. For each $\theta\in(0,\theta_0]$, if $E+\kappa^2+\ell^2$ is sufficiently small, there exist $b\in\partial Y$, $H=H^T\succ0$, and an orthogonal matrix $O$ such that
\begin{equation}\label{geom:eq:changesmall}
 |b|+\|H-\Id\|\le C(\sqrt E+\kappa).
\end{equation}
The change of coordinates
\begin{equation}\label{geom:eq:jettransform}
 x=H^{-1/2}O\widehat x,\qquad
 y=b+H^{1/2}O\widehat y
\end{equation}
places both transformed supports in $\{z_d\ge0\}$, with common boundary point $0$ and tangent plane $\{z_d=0\}$. The transformed densities satisfy $\widehat\rho_0(0)=\widehat\rho_1(0)$. Writing $\widehat\pi$ for the pushforward of $\pi$ under the inverse coordinate map, we have
\begin{equation}\label{geom:eq:onestep}
 \cE(\widehat\pi,\theta)
 \le C\theta^{2\beta}(E+\kappa^2)
       +C\theta^{-d-2}\ell^2.
\end{equation}
The smallness threshold may depend on $\theta$, whereas $C$ is independent of $\theta$.

Under the interior hypotheses of \cref{geom:lem:classical}, the same estimates and density normalization hold with $b\in\intr Y$ and $O=\Id$. %
\end{proposition}

\begin{proof}
By \cref{geom:lem:long}, if $E+\kappa^2+\ell^2$ is sufficiently small, every pair in $\spt\pi$ with one endpoint in $B_{1/2}$ has the other in $B_{3/4}$. Set $D:=B_{3/4}\times B_{3/4}$ and  $ \pi_D:=\pi|_D,$ 
and let $\alpha_0,\alpha_1$ be its marginals. Then $\alpha_0|_{B_{1/2}}=\mu|_{B_{1/2}}$ and $\alpha_1|_{B_{1/2}}=\nu|_{B_{1/2}}$. Let $\gamma\in\Pi(\alpha_0,\alpha_1)$ minimize quadratic transport. Since $\gamma$ is supported in $\overline B_{3/4}\times\overline B_{3/4}$,
\begin{equation}\label{m:display:027}
 \cE(\gamma,1)
 =\int|x-y|^2\dd\gamma
 \le\int_D|x-y|^2\dd\pi
 \le E.
\end{equation}
Thus \cref{geom:lem:classical} applies to $\gamma$. Write $\mathcal T=\nabla\phi$, $\mathcal T^{-1}=\nabla\psi$, where $\phi,\psi$ are conjugate convex potentials, and set $b:=\mathcal T(0)$ and $H:=D\mathcal T(0)$. Equation~\eqref{geom:eq:jet} gives \eqref{geom:eq:changesmall}.

Define
\begin{equation}\label{m:display:028}
 G(x,y):=\phi(x)+\psi(y)-x\cdot y\ge0.
\end{equation}
Since $G=0$ $\gamma$-almost everywhere and $\pi_D,\gamma$ have the same marginals, \cref{geom:prop:rectangle} gives
\begin{equation}\label{geom:eq:gapintegral}
 \int_D G\dd\pi
 =
 \int\tfrac12|x-y|^2\dd\pi_D
 -\int\tfrac12|x-y|^2\dd\gamma
 \le C\ell^2.
\end{equation}

We next choose $O$. Let $n_0,n_1(b)$ denote the inward unit normals to $\partial X$ at $0$ and $\partial Y$ at $b$. The boundary identity for $\mathcal T$ and the symmetry of $H$ imply
\begin{equation}\label{m:display:030}
 H(T_0\partial X)=T_b\partial Y,\qquad
 Hn_1(b)=\lambda n_0,\quad \lambda>0.
\end{equation}
The sign follows because $\mathcal T$ maps the interior into the interior. Choose an orthogonal matrix $O$ satisfying
\begin{equation}\label{m:display:031}
 Oe_d=\frac{H^{-1/2}n_0}{|H^{-1/2}n_0|}.
\end{equation}
Under \eqref{geom:eq:jettransform}, the inward normals are proportional to $O^TH^{-1/2}n_0$ and $O^TH^{1/2}n_1(b)$, respectively. Both are positive multiples of $e_d$ by \eqref{m:display:030}--\eqref{m:display:031}. Convexity therefore places both transformed supports in $\{z_d\ge0\}$.

By \cref{geom:lem:scaling}, the transformed plan $\widehat\pi$ is a QOT optimizer at the same parameter $\eps$. The Jacobians and \eqref{geom:eq:jet} give
\begin{equation}\label{m:display:032}
 \widehat\rho_0(0)
 =(\det H)^{-1/2}\rho_0(0)
 =(\det H)^{1/2}\rho_1(b)
 =\widehat\rho_1(0).
\end{equation}

We now localize the pairs in $\spt\widehat\pi\cap\#B_\theta$. By \eqref{geom:eq:changesmall}, we may assume $\|H^{1/2}\|,\|H^{-1/2}\|\le2$. For their original coordinates, $\min\{|x|,|y|\}\le |b|+2\theta$. Choose $\theta_0$ sufficiently small compared with $r_c$, and then decrease the smallness threshold. The estimate in \cref{geom:lem:long} yields
\[
 \max\{|x|,|y|\}
 \le |b|+2\theta
      +C\bigl(E^{1/(d+2)}+\kappa+\ell\bigr)
 \le r_c.
\]
Moreover, \eqref{geom:eq:linearclassical} gives $|\mathcal T(x)|,|\mathcal T^{-1}(y)|\le2r_c$. In particular, these pairs belong to $D$.

The segments $[\mathcal T(x),y]$ and $[\mathcal T^{-1}(y),x]$ lie in $Y\cap B_{2r_c}$ and $X\cap B_{2r_c}$, respectively. Using \eqref{eq:replacementabsolute} and the inverse identities, we obtain
\begin{equation}\label{m:display:029}
\begin{aligned}
 G(x,y)
 &=\psi(y)-\psi(\mathcal T(x))
       -x\cdot(y-\mathcal T(x))
 \ge c|y-\mathcal T(x)|^2,\\
 G(x,y)
 &=\phi(x)-\phi(\mathcal T^{-1}(y))
       -y\cdot(x-\mathcal T^{-1}(y))
 \ge c|x-\mathcal T^{-1}(y)|^2.
\end{aligned}
\end{equation}

If $\widehat x\in B_\theta$, then $|x|\le2\theta$ and $\widehat y-\widehat x =O^TH^{-1/2}(y-b-Hx)$. The forward estimate in \eqref{geom:eq:taylorclassical} therefore gives
\begin{equation}\label{m:display:033}
 |\widehat y-\widehat x|^2
 \le C|y-\mathcal T(x)|^2
      +C(E+\kappa^2)\theta^{2+2\beta}.
\end{equation}
If $\widehat y\in B_\theta$, then $|y-b|\le2\theta$ and $\widehat y-\widehat x =-O^TH^{1/2}(x-H^{-1}(y-b))$. The inverse estimate gives
\begin{equation}\label{m:display:034}
 |\widehat y-\widehat x|^2
 \le C|x-\mathcal T^{-1}(y)|^2
      +C(E+\kappa^2)\theta^{2+2\beta}.
\end{equation}
The transformed marginal densities are bounded above by $CM$, so each of $B_\theta\times\R^d$ and $\R^d\times B_\theta$ has $\widehat\pi$-mass at most $C\theta^d$. Integrating \eqref{m:display:033}--\eqref{m:display:034} and using \eqref{m:display:029} and \eqref{geom:eq:gapintegral}, we conclude
\[
 \int_{\#B_\theta}
       |\widehat x-\widehat y|^2\dd\widehat\pi
 \le C\ell^2
      +C(E+\kappa^2)\theta^{d+2+2\beta}.
\]
Division by $\theta^{d+2}$ proves \eqref{geom:eq:onestep}.

For the interior assertion, use the interior versions of \cref{geom:lem:long,geom:lem:classical} and take $O=\Id$. The density identity and the excess estimate follow from the same calculations.
\end{proof}

\section{Geometry of the sections}\label{sec:uniformgeometry}
This section establishes uniform control of the size and geometry of the support sections by iterating the estimates from the previous section down to scale $\ell$. Throughout the section, \cref{ass:geometrydata} holds, constants depend only on its data, and $\ell$ is defined by \eqref{eq:length}.
\subsection{A stopped excess iteration}
\label{geom:sec:iteration}

We prove the section-diameter bound \eqref{geom:eq:diameters} by iterating \cref{geom:prop:improvement} down to a fixed multiple of $\ell$. We first treat boundary points, then use the boundary excess estimates to initialize interior iterations at scales comparable with the distance to the boundary.

\subsubsection{Initialization from the classical Brenier map}
By \cref{lem:classicalfinite}, the Brenier map $T$ and its inverse are $C^{1,\alpha_*}$ on $X$ and $Y$, respectively, for some $\alpha_*\in(0,1)$, with uniformly bounded and uniformly positive definite derivatives. The global recovery construction in \cref{geom:lem:recovery} gives
\begin{equation}\label{geom:eq:globalenergy}
\QOT_\eps-\OT\le C\ell^2,\qquad \ell=\eps^{1/(d+2)}.
\end{equation}
Uniform convexity of the classical potentials yields
\begin{equation}\label{geom:eq:globalLtwo}
\int\bigl(|y-T(x)|^2+|x-T^{-1}(y)|^2\bigr)\dd\pi_\eps\le C\ell^2.
\end{equation}
Indeed, classical dual equality and the marginal constraints give
\begin{equation}\label{eq:classicalgap-exact}
\int\bigl(u_0(x)+v_0(y)-x\cdot y\bigr)\dd\pi_\eps=\frac12\int|x-y|^2\dd\pi_\eps-\OT.
\end{equation}
Uniform convexity of $v_0$ bounds the integrand on the left below by a constant times $|y-T(x)|^2$ and uniform convexity of $u_0$ gives the corresponding bound for $|x-T^{-1}(y)|^2$. The nonnegative regularization penalty and \eqref{geom:eq:globalenergy} then imply \eqref{geom:eq:globalLtwo}.

Fix $p\in\partial X$, and set $q=T(p)$ and $A=DT(p)$. Let $n_p$ be the inward unit normal at $p$, and choose an orthogonal matrix $O$ with $Oe_d=A^{-1/2}n_p/|A^{-1/2}n_p|$. The reciprocal affine change
\begin{equation}\label{m:display:035}
x=p+A^{-1/2}O\xi,\qquad y=q+A^{1/2}O\eta
\end{equation}
places both supports in $\{z_d\ge0\}$ with tangent plane $\{z_d=0\}$ at the origin, as in the proof of \cref{geom:prop:improvement}. The transformed classical map has derivative $\Id$ there, and the transport Jacobian identity gives equality of the transformed density values at the origin. Set $L_{\rm init}=A^{-1/2}O$; this matrix and its inverse are uniformly bounded in $p$.

Let $\pi^{\rm init}$ be the pushforward of $\pi_\eps$ under the inverse coordinate change \eqref{m:display:035}. The map
\[
T^{\rm init}(\xi) := O^TA^{-1/2} \bigl(T(p+A^{-1/2}O\xi)-T(p)\bigr)
\]
is the Brenier map from $\mu^{\rm init}$ to $\nu^{\rm init}$, where
\[
\mu^{\rm init} :=\bigl[x\mapsto O^TA^{1/2}(x-p)\bigr]_\#\mu, \qquad \nu^{\rm init} :=\bigl[y\mapsto O^TA^{-1/2}(y-T(p))\bigr]_\#\nu.
\]
Since $A=DT(p)$, we have $T^{\rm init}(0)=0$ and $DT^{\rm init}(0)=O^TA^{-1/2}AA^{-1/2}O=\Id$. Its inverse also fixes $0$ and has derivative $\Id$ there. The uniform $C^{1,\alpha_*}$ bounds therefore give, on the respective transformed supports,
\begin{equation}\label{estimates-T-init}
|T^{\rm init}(\xi)-\xi| \le C|\xi|^{1+\alpha_*}, \qquad |(T^{\rm init})^{-1}(\eta)-\eta| \le C|\eta|^{1+\alpha_*}.
\end{equation}
These estimates hold up to the boundary: the segments from $0$ to $\xi$ and $\eta$ remain in the corresponding convex supports.

For corresponding pairs $(x,y)$ and $(\xi,\eta)$, we have $\eta-T^{\rm init}(\xi) =O^TA^{-1/2}(y-T(x))$ and $\xi-(T^{\rm init})^{-1}(\eta) =O^TA^{1/2}(x-T^{-1}(y))$. The uniform bounds on $A^{1/2},A^{-1/2}$ and \eqref{geom:eq:globalLtwo} thus imply
\begin{equation}\label{geom:eq:globalLtwo-2}
\int\Bigl( |\eta-T^{\rm init}(\xi)|^2 +|\xi-(T^{\rm init})^{-1}(\eta)|^2 \Bigr)\dd\pi^{\rm init} \le C\ell^2.
\end{equation}
Fix a small radius $R_0>0$, independent of $\eps$. If $|\xi|<R_0$, the first estimate in \eqref{estimates-T-init} gives
\[
|\xi-\eta|^2 \le 2|\eta-T^{\rm init}(\xi)|^2 +CR_0^{2+2\alpha_*}.
\]
If $|\eta|<R_0$, the second estimate in \eqref{estimates-T-init} gives
\[
|\xi-\eta|^2 \le 2|\xi-(T^{\rm init})^{-1}(\eta)|^2 +CR_0^{2+2\alpha_*}.
\]
The densities of $\mu^{\rm init}$ and $\nu^{\rm init}$ are uniformly bounded above, so $\pi^{\rm init}(B_{R_0}\times\R^d)\le CR_0^d$ and $\pi^{\rm init}(\R^d\times B_{R_0})\le CR_0^d$. Summing the integrals over the two parts of $\#B_{R_0}$ gives
\[
\begin{aligned}
\int_{\#B_{R_0}}|\xi-\eta|^2\dd\pi^{\rm init} &\le 2\int\Bigl( |\eta-T^{\rm init}(\xi)|^2 +|\xi-(T^{\rm init})^{-1}(\eta)|^2 \Bigr)\dd\pi^{\rm init}\\
&\quad+ CR_0^{2+2\alpha_*} \Bigl( \pi^{\rm init}(B_{R_0}\times\R^d) +\pi^{\rm init}(\R^d\times B_{R_0}) \Bigr)\\
&\le C\ell^2+CR_0^{d+2+2\alpha_*},
\end{aligned}
\]
where we also used \eqref{geom:eq:globalLtwo-2}. Dividing by $R_0^{d+2}$, as prescribed by \eqref{geom:eq:excess}, proves
\begin{equation}\label{geom:eq:initial}
E_0:=\cE(\pi^{\rm init},R_0) \le CR_0^{2\alpha_*} +C\ell^2R_0^{-d-2}.
\end{equation}

Here and below, $K\ge1$ is a fixed bound for the relevant boundary chart and density norms. The constants $C,K$ depend only on the data and are independent of $p,R_0,\eps$. After rescaling by $R_0$, the graph and density bounds in \eqref{geom:eq:smallclassicaldata} are at most $KR_0$. Thus, given $\delta>0$, choose $R_0>0$ so that $CR_0^{2\alpha_*}+KR_0\le\delta/2$, and then $\eps_0>0$ so that $C\eps_0^{2/(d+2)}R_0^{-d-2}\le\delta/2$. By \eqref{geom:eq:initial}, $E_0+KR_0\le\delta$ for every $0<\eps\le\eps_0$, uniformly in $p$.

\subsubsection{A finite Campanato iteration}

Use common constants $\beta,C$ in the interior and boundary versions of \cref{geom:prop:improvement}, with density bounds $m\sqrt{m/M}$ and $M\sqrt{M/m}$. Fix $0<\gamma<\min\{\beta,1\}$ and $0<\theta<\min\{1/8,\theta_0\}$ such that
\begin{equation}\label{m:display:036}
C\theta^{2\beta}\le\theta^{2\gamma}=:a<1.
\end{equation}
Constants below may depend on this fixed $\theta$. Let $\delta_\theta$ be the threshold in \cref{geom:prop:improvement}, and
\begin{equation}\label{m:display:037}
R_n:=\theta^nR_0.
\end{equation}
Fix a large $\Lambda$, to be chosen below. For $\Lambda\ell\le R_0$, let $N$ be the largest index such that $R_N\ge\Lambda\ell$. Then
\begin{equation}\label{geom:eq:stopradius}
\Lambda\ell\le R_N<\Lambda\ell/\theta.
\end{equation}

Starting from $L_0:=L_{\rm init}$ and $q_0:=T(p)$, we construct $L_n,q_n$ inductively. Given these quantities at step $n$, use the coordinates
\begin{equation}\label{geom:eq:iterationcoordinates}
x=p+R_nL_n\xi,\qquad y=q_n+R_n(L_n^{-1})^T\eta,
\end{equation}
and define
\[
\pi^{(n)} :=R_n^{-d} \left[ (x,y)\mapsto \left( R_n^{-1}L_n^{-1}(x-p), R_n^{-1}L_n^T(y-q_n) \right) \right]_\#\pi_\eps, \qquad E_n:=\cE(\pi^{(n)},1).
\]
This agrees with the preceding definition of $E_0$. The marginal supports are
\begin{equation}\label{m:display:038}
X_n:=R_n^{-1}L_n^{-1}(X-p),\qquad Y_n:=R_n^{-1}L_n^T(Y-q_n),
\end{equation}
and their densities are
\[
\rho_{0,n}(\xi) =|\det L_n|\rho_0(p+R_nL_n\xi),\qquad \rho_{1,n}(\eta) =|\det L_n|^{-1}\rho_1(q_n+R_n(L_n^{-1})^T\eta).
\]
By \cref{geom:lem:scaling}, $\pi^{(n)}$ is a QOT optimizer with parameter $\eps/R_n^{d+2}$ and regularization length $\ell/R_n$.

The construction preserves $q_n\in\partial Y$, the common tangent plane $\{z_d=0\}$ at $0$, containment of both supports in $\{z_d\ge0\}$, and $\rho_{0,n}(0)=\rho_{1,n}(0)$. These properties hold at $n=0$ by the initialization. Density matching gives
\[
|\det L_n|^2 =\frac{\rho_1(q_n)}{\rho_0(p)} \in[m/M,M/m].
\]
Thus the marginal densities remain between the fixed bounds specified above.

Assume temporarily that $\|L_{\rm init}^{-1}L_n\|\le2$ and $\|L_n^{-1}L_{\rm init}\|\le2$. Let $D_0\ge1$ be a uniform bound for $\|L_n\|,\|L_n^{-1}\|$ under this assumption. Before dilation by $R_n$, the supports $L_n^{-1}(X-p)$ and $L_n^T(Y-q_n)$ have boundary graphs $\widetilde P_n,\widetilde Q_n$ on a fixed neighborhood of $0$. The original chart bounds, the bounds on the matrices and their inverses, and the implicit function theorem give
\[
\widetilde P_n(0)=\widetilde Q_n(0)=0,\qquad D\widetilde P_n(0)=D\widetilde Q_n(0)=0,\qquad \|\widetilde P_n\|_{C^{3,\alpha}} +\|\widetilde Q_n\|_{C^{3,\alpha}}\le CK.
\]
For $R_0$ sufficiently small, the boundary graphs of $X_n,Y_n$ on $|\xi'|<1$ are $P_n(\xi')=R_n^{-1}\widetilde P_n(R_n\xi')$ and $Q_n(\xi')=R_n^{-1}\widetilde Q_n(R_n\xi')$. Hence
\[
D^kP_n(\xi') =R_n^{k-1}D^k\widetilde P_n(R_n\xi') \quad(k=2,3),\qquad [D^3P_n]_\alpha \le R_n^{2+\alpha}[D^3\widetilde P_n]_\alpha.
\]
Since $P_n(0)=DP_n(0)=0$, Taylor's formula gives $\|P_n\|_\infty+\|DP_n\|_\infty\le CKR_n$. The same estimates hold for $Q_n$, so $\|P_n\|_{C^{3,\alpha}}+\|Q_n\|_{C^{3,\alpha}} \le CKR_n$.

The density formulas and the chain rule give
\[
\begin{aligned}
\|\rho_{0,n}-\rho_{0,n}(0)\|_{L^\infty(X_n\cap B_1)}+\|D\rho_{0,n}\|_{L^\infty(X_n\cap B_1)}&\le CKR_n,\\
[D\rho_{0,n}]_{\alpha;X_n\cap B_1}&\le CKR_n^{1+\alpha},
\end{aligned}
\]
with the analogous estimates for $\rho_{1,n}$ on $Y_n\cap B_1$. Therefore the quantity in \eqref{geom:eq:smallclassicaldata} for the marginals of $\pi^{(n)}$ satisfies
\begin{equation}\label{geom:eq:kappan}
\begin{aligned}
\kappa_n &:= \|P_n\|_{C^{3,\alpha}}+\|Q_n\|_{C^{3,\alpha}}\\
&\quad+ \|\rho_{0,n}-\rho_{0,n}(0)\|_{C^{1,\alpha}(X_n\cap B_1)} +\|\rho_{1,n}-\rho_{1,n}(0)\|_{C^{1,\alpha}(Y_n\cap B_1)} \le CKR_n.
\end{aligned}
\end{equation}

Suppose $n<N$ and $E_n+\kappa_n^2+(\ell/R_n)^2\le\delta_\theta$. Apply \cref{geom:prop:improvement} to $\pi^{(n)}$, and let $H_n,b_n,O_n$ be the resulting matrices and translation. They satisfy
\begin{equation}\label{geom:eq:increments}
\|H_n-\Id\|+|b_n| \le C(\sqrt{E_n}+KR_n),\qquad L_{n+1}:=L_nH_n^{-1/2}O_n.
\end{equation}
Define also
\begin{equation}\label{m:display:040}
q_{n+1}:=q_n+R_n(L_n^{-1})^Tb_n.
\end{equation}
The coordinate change supplied by the proposition, followed by dilation by $\theta$, is
\[
\xi=\theta H_n^{-1/2}O_n\xi_{n+1},\qquad \eta=b_n+\theta H_n^{1/2}O_n\eta_{n+1}.
\]
Substitution into \eqref{geom:eq:iterationcoordinates} gives the same coordinate representation at level $n+1$, with $R_{n+1}=\theta R_n$. The plan $\pi^{(n+1)}$ is consequently obtained from $\pi^{(n)}$ by the inverse of this coordinate change and multiplication by $\theta^{-d}$. The proposition preserves the supporting planes, half-spaces, and density matching.

Let $\widehat\pi^{(n)}$ denote the plan after the change using $H_n,b_n,O_n$, before dilation. The next plan is
\[
\pi^{(n+1)} =\theta^{-d} \bigl[(\xi,\eta)\mapsto(\xi/\theta,\eta/\theta)\bigr]_\# \widehat\pi^{(n)}.
\]
Consequently, \eqref{geom:eq:excess} gives
\[
E_{n+1}=\cE(\pi^{(n+1)},1)=\theta^{-d-2}\int_{\#B_\theta}|\xi-\eta|^2\dd\widehat\pi^{(n)}
=\cE(\widehat\pi^{(n)},\theta).
\]
By \eqref{geom:eq:onestep},
\[
E_{n+1} =\cE(\widehat\pi^{(n)},\theta) \le C\theta^{2\beta}(E_n+\kappa_n^2) +C\theta^{-d-2}(\ell/R_n)^2.
\]
Using \eqref{m:display:036} and \eqref{geom:eq:kappan}, we obtain
\begin{equation}\label{geom:eq:recurrence}
E_{n+1}\le aE_n+CK^2R_n^2+C(\ell/R_n)^2.
\end{equation}

We now verify that this construction reaches $N$. For every constructed level $n\le N$, iteration of \eqref{geom:eq:recurrence} gives
\[
E_n\le a^nE_0+ CK^2R_0^2\sum_{j=0}^{n-1}a^{n-1-j}\theta^{2j} +C(\ell/R_0)^2 \sum_{j=0}^{n-1}a^{n-1-j}\theta^{-2j}.
\]
Since $\theta^2/a<1$ and $a\theta^2<1$,
\begin{equation}\label{geom:eq:iterationbound}
E_n\le C\theta^{2\gamma n}(E_0+K^2R_0^2) +C(\ell/R_n)^2.
\end{equation}
Taking square roots in \eqref{geom:eq:iterationbound} gives
\[
\sqrt{E_j} \le C\theta^{\gamma j}(\sqrt{E_0}+KR_0) +C\ell/R_j.
\]
Since $R_j=\theta^jR_0$, we have
\[
\sum_{j=0}^n\theta^{\gamma j} \le\frac1{1-\theta^\gamma},\qquad \sum_{j=0}^n KR_j\le\frac{KR_0}{1-\theta},
\]
and
\[
\sum_{j=0}^n\frac{\ell}{R_j} =\frac{\ell}{R_n}\sum_{j=0}^n\theta^{n-j} \le\frac{\ell}{(1-\theta)R_n} \le\frac1{(1-\theta)\Lambda},
\]
where the last inequality uses $n\le N$. Adding these estimates yields
\begin{equation}\label{geom:eq:summability}
\sum_{j=0}^n(\sqrt{E_j}+KR_j) \le C\bigl(\sqrt{E_0}+KR_0+\Lambda^{-1}\bigr), \qquad n\le N,
\end{equation}
with $C$ independent of $n,N$ and $\eps$. Moreover, \eqref{geom:eq:increments} and orthogonality of $O_j$ imply
\[
\max\bigl\{ \|L_{\rm init}^{-1}L_n\|, \|L_n^{-1}L_{\rm init}\| \bigr\} \le \exp\!\left( C\sum_{j<n}(\sqrt{E_j}+KR_j)\right).
\]
Choose $\Lambda$ large, then $R_0$ small, and finally $\eps$ small using \eqref{geom:eq:initial}, so that
\[
C(\sqrt{E_0}+KR_0+\Lambda^{-1})<\log2, \qquad C(E_0+K^2R_0^2+\Lambda^{-2})<\delta_\theta,
\]
with $C$ large enough for the preceding estimates. The first inequality ensures the two relative matrix norms remain below $2$. Thus \eqref{geom:eq:kappan} remains valid, and the second inequality gives $E_n+\kappa_n^2+(\ell/R_n)^2<\delta_\theta$. Induction therefore constructs every level $0\le n\le N$.

The target centers satisfy
\begin{equation}\label{m:display:041}
|q_n-T(p)| \le C\sum_{j<n}R_j(\sqrt{E_j}+KR_j) \le CR_0\sum_{j<n}(\sqrt{E_j}+KR_j).
\end{equation}
In particular, they remain in a fixed neighborhood of $T(p)$.

Let $\delta_{\rm long}>0$ be the threshold in \cref{geom:lem:long}. At the last step $n=N$, where $\Lambda\ell\le R_N<\Lambda\ell/\theta$, \eqref{geom:eq:iterationbound}, \eqref{geom:eq:kappan}, and \eqref{geom:eq:stopradius} give
\[
E_N\le C(E_0+K^2R_0^2+\Lambda^{-2}),\qquad \kappa_N\le CKR_0,\qquad \frac{\ell}{R_N}\le\Lambda^{-1}.
\]
Choose $\Lambda$ large, then $R_0$ small, and finally $\eps$ small using \eqref{geom:eq:initial}, so that $E_N+\kappa_N+\ell/R_N\le\delta_{\rm long}$. The supports of the marginals of $\pi^{(N)}$ are convex, lie in $\{z_d\ge0\}$, and have the required graph representations and density bounds. Thus \cref{geom:lem:long} applies and gives
\[
\sup_{(\xi,\eta)\in\spt\pi^{(N)}\cap\#B_{1/2}} |\xi-\eta| \le C\left(E_N^{1/(d+2)} +\kappa_N+\frac{\ell}{R_N}\right) \le C.
\]

Fix $c_b>0$ such that $D_0c_b<1/2$, and take $x\in X$ with $|x-p|\le c_bR_N$. In the level-$N$ coordinates,
\[
\xi:=R_N^{-1}L_N^{-1}(x-p),\qquad |\xi|\le D_0c_b<1/2.
\]
For any $y,y'\in S_{\eps,x}$, define $\eta:=R_N^{-1}L_N^T(y-q_N)$ and $\eta':=R_N^{-1}L_N^T(y'-q_N)$. By \eqref{eqv:eq:support} and the definition of $\pi^{(N)}$, both $(\xi,\eta)$ and $(\xi,\eta')$ belong to $\spt\pi^{(N)}\cap\#B_{1/2}$. Consequently,
\[
|y-y'| =R_N|(L_N^{-1})^T(\eta-\eta')|
\le D_0R_N\bigl(|\eta-\xi|+|\eta'-\xi|\bigr) \le CR_N \le C\frac{\Lambda}{\theta}\ell.
\]
Taking the supremum over $y,y'\in S_{\eps,x}$ proves $\diam S_{\eps,x}\le C\ell$, with the fixed factor $\Lambda/\theta$ absorbed into $C$.

All constants are uniform in $p\in\partial X$. If $\dist(x,\partial X)\le c_b\Lambda\ell$, choose $p\in\partial X$ attaining this distance. Then $|x-p|\le c_b\Lambda\ell\le c_bR_N$, so
\begin{equation}\label{geom:eq:boundarystrip}
\diam S_{\eps,x}\le C\ell \qquad\text{if }\dist(x,\partial X) \le c_b\Lambda\ell.
\end{equation}
Applying the same argument with the marginals interchanged gives the  bound for $K_{\eps,y}$ near~$\partial Y$.

\subsubsection{Interior localization from boundary estimates}
The passage from boundary estimates to an interior iteration follows the argument of Miura--Otto for unregularized transport in \cite[Section~6.3, proof of Theorem~1.1]{MO}. We adapt it to QOT, stopping at a fixed multiple of $\ell$, as in the regularized interior iteration of \cite[Theorem~15]{GK}. Recall $D_0,c_b$ from the proof of \eqref{geom:eq:boundarystrip}, with $D_0c_b<1/2$. Choose $M_0$ sufficiently large and then $r_*>0$ sufficiently small that
\begin{equation}\label{geom:eq:coverconstants}
M_0c_b>2/\theta,\qquad D_0/M_0<1/16,\qquad 0<r_*<\frac{\theta}{64D_0M_0}.
\end{equation}
These constants are independent of $\Lambda$, which may be enlarged below. Take $x\in X$ with
\begin{equation}\label{m:display:042}
t:=\dist(x,\partial X),\qquad c_b\Lambda\ell<t<R_0/M_0.
\end{equation}
Choose $p\in\partial X$ such that $|x-p|=t$. Then $x-p=tn_p$, where $n_p$ is the inward unit normal. In the boundary iteration based at $p$, let $n$ be the largest index with $R_n\ge M_0t$. Thus
\begin{equation}\label{geom:eq:coverlevel}
M_0t\le R_n<M_0t/\theta.
\end{equation}
This index exists because $M_0t<R_0$ and $R_k\to0$. Moreover, $M_0t>M_0c_b\Lambda\ell>2\Lambda\ell/\theta>R_N$, so $n<N$.

The coordinate of $x$ at this level is
\begin{equation}\label{m:display:043}
\xi:=R_n^{-1}L_n^{-1}(x-p).
\end{equation}
Since the inward normal to $X_n$ at $0$ is $e_d$, we have $e_d=\frac{L_n^Tn_p}{|L_n^Tn_p|}.$ 
Consequently, \eqref{geom:eq:coverlevel} gives
\begin{equation}\label{geom:eq:normaldistance}
|\xi|\le\frac{D_0}{M_0},\qquad \xi_d =\frac{L_n^Tn_p\cdot L_n^{-1}(x-p)} {R_n|L_n^Tn_p|} =\frac{t}{R_n|L_n^Tn_p|} \ge\frac{\theta}{D_0M_0}.
\end{equation}

The graph representations of $X_n,Y_n$ are $X_n\cap B_1=\{z\in B_1:z_d\ge P_n(z')\}$ and $Y_n\cap B_1=\{z\in B_1:z_d\ge Q_n(z')\}$. By \eqref{geom:eq:kappan}, both graphs vanish to first order at $0$ and have second derivatives bounded by $CKR_n$. For $z\in B(\xi,4r_*)$, we have $|z|<D_0/M_0+4r_*<1$, and Taylor's formula gives
\[
\max\{P_n(z'),Q_n(z')\} \le CKR_0(D_0/M_0+4r_*)^2 \le\frac{\theta}{2D_0M_0} <\xi_d-4r_*<z_d,
\]
after decreasing $R_0$. Together with \eqref{geom:eq:coverlevel}, this yields
\begin{equation}\label{geom:eq:interiorstartballs}
B(\xi,4r_*)\subset\intr X_n\cap\intr Y_n,\qquad r_*M_0t\le r_*R_n<r_*M_0t/\theta.
\end{equation}
In particular, $r_*R_n\asymp t$, with constants independent of $x$ and $\eps$.

Since $\rho_{0,n}(0)=\rho_{1,n}(0)$, $|\rho_{0,n}(\xi)-\rho_{1,n}(\xi)|\le C\kappa_n$. Define the positive scalar matrix
\begin{equation}\label{m:display:045}
A_*:= \left(\frac{\rho_{0,n}(\xi)} {\rho_{1,n}(\xi)}\right)^{1/d}\Id.
\end{equation}
The density bounds imply $\|A_*-\Id\|\le C\kappa_n$. For $R_0$ sufficiently small, $\|A_*^{1/2}\|,\|A_*^{-1/2}\|<4/3$.

For a pair $(x_n,y_n)$ in the level-$n$ coordinates, use
\begin{equation}\label{m:display:046}
x_n=\xi+r_*A_*^{-1/2}\zeta,\qquad y_n=\xi+r_*A_*^{1/2}\eta,
\end{equation}
and define
\[
\widetilde\pi :=r_*^{-d} \left[ (x_n,y_n)\mapsto \left( r_*^{-1}A_*^{1/2}(x_n-\xi), r_*^{-1}A_*^{-1/2}(y_n-\xi) \right) \right]_\#\pi^{(n)}.
\]
Set $\widetilde E_0:=\cE(\widetilde\pi,1)$. Substitution into \eqref{geom:eq:iterationcoordinates} gives
\[
x'=x+r_*R_n(L_nA_*^{-1/2})\zeta,\qquad y'=q_n+R_n(L_n^{-1})^T\xi +r_*R_n(L_nA_*^{-1/2})^{-T}\eta,
\]
where we used $p+R_nL_n\xi=x$ and $(L_nA_*^{-1/2})^{-T}=(L_n^{-1})^TA_*^{1/2}$. Thus the initial radius, source matrix, and target center for the interior iteration are $\widetilde R_0:=r_*R_n$, $\widetilde L_0:=L_nA_*^{-1/2}$, and $\widetilde q_0:=q_n+R_n(L_n^{-1})^T\xi$. The source center is $x$, and the total mass multiplier relative to $\pi_\eps$ is $(r_*R_n)^{-d}$.

The marginal supports of $\widetilde\pi$ are $\widetilde X=r_*^{-1}A_*^{1/2}(X_n-\xi)$ and $\widetilde Y=r_*^{-1}A_*^{-1/2}(Y_n-\xi)$. By \eqref{geom:eq:interiorstartballs},
\[
\xi+r_*A_*^{-1/2}B_3\subset B(\xi,4r_*)\subset X_n, \qquad \xi+r_*A_*^{1/2}B_3\subset B(\xi,4r_*)\subset Y_n.
\]
Hence $B_3\subset\widetilde X\cap\widetilde Y$. By \cref{geom:lem:scaling}, their densities are
\[
\widetilde\rho_0(\zeta) =(\det A_*)^{-1/2} \rho_{0,n}(\xi+r_*A_*^{-1/2}\zeta),\qquad
\widetilde\rho_1(\eta) =(\det A_*)^{1/2} \rho_{1,n}(\xi+r_*A_*^{1/2}\eta).
\]
The identity $\det A_*=\rho_{0,n}(\xi)/\rho_{1,n}(\xi)$ gives $\widetilde\rho_0(0)=\widetilde\rho_1(0)$. Let $\widetilde\kappa_0$ denote the density quantity in \eqref{geom:eq:smallclassicaldata}, with the graph terms omitted.

To estimate $\widetilde E_0$, note that \eqref{m:display:046} gives both identities
\[
\begin{aligned}
\eta-\zeta &=r_*^{-1}A_*^{-1/2}(y_n-x_n) +r_*^{-1}(A_*^{-1/2}-A_*^{1/2})(x_n-\xi),\\
\eta-\zeta &=r_*^{-1}A_*^{1/2}(y_n-x_n) +r_*^{-1}(A_*^{-1/2}-A_*^{1/2})(y_n-\xi).
\end{aligned}
\]
Use the first when $\zeta\in B_1$, so that $|x_n-\xi|\le(4/3)r_*$, and the second when $\eta\in B_1$, so that $|y_n-\xi|\le(4/3)r_*$. In either case, $|\eta-\zeta|^2\le Cr_*^{-2}|y_n-x_n|^2+C\kappa_n^2$.

The inverse image of $\#B_1$ under this change of coordinates is
\[
\mathcal D:= \bigl((\xi+r_*A_*^{-1/2}B_1)\times\R^d\bigr) \cup \bigl(\R^d\times(\xi+r_*A_*^{1/2}B_1)\bigr).
\]
Since $|\xi|+(4/3)r_*<1$, we have $\mathcal D\subset\#B_1$. The marginal density bounds give
\[
\begin{aligned}
\pi^{(n)}(\mathcal D) &\le \int_{\xi+r_*A_*^{-1/2}B_1}\rho_{0,n}(x_n)\dd x_n +\int_{\xi+r_*A_*^{1/2}B_1}\rho_{1,n}(y_n)\dd y_n\\
&\le C\bigl( |r_*A_*^{-1/2}B_1|+|r_*A_*^{1/2}B_1| \bigr) \le Cr_*^d.
\end{aligned}
\]
Therefore
\[
\begin{aligned}
\widetilde E_0 &=r_*^{-d}\int_{\mathcal D} |\eta-\zeta|^2\dd\pi^{(n)}\\
&\le Cr_*^{-d-2} \int_{\mathcal D}|y_n-x_n|^2\dd\pi^{(n)} +Cr_*^{-d}\kappa_n^2\pi^{(n)}(\mathcal D)\\
&\le Cr_*^{-d-2}E_n+C\kappa_n^2.
\end{aligned}
\]
For $i=0,1$, the density formulas and the chain rule give
\[
\|D\widetilde\rho_i\|_{L^\infty(B_1)} \le CKr_*R_n,\qquad [D\widetilde\rho_i]_{\alpha;B_1} \le CK(r_*R_n)^{1+\alpha}.
\]
Moreover, $\|\widetilde\rho_i-\widetilde\rho_i(0)\|_{L^\infty(B_1)} \le\|D\widetilde\rho_i\|_{L^\infty(B_1)}$. Since $r_*R_n\le1$, we conclude that
\begin{equation}\label{geom:eq:restartestimates}
\widetilde E_0 \le Cr_*^{-d-2}E_n+C\kappa_n^2,\qquad \widetilde\kappa_0\le CKr_*R_n.
\end{equation}
The new parameter is $\eps/(r_*R_n)^{d+2}$, and its regularization length is $\ell/(r_*R_n)$.

Let $\delta_{\rm long}$ be the threshold in \cref{geom:lem:long}. Choose $\Lambda_{\rm int}$ sufficiently large, and then choose $\Lambda$ large enough for the boundary iteration and such that
\begin{equation}\label{geom:eq:stopcompatibility}
r_*M_0c_b\Lambda>4\Lambda_{\rm int}.
\end{equation}
Since $R_n>M_0c_b\Lambda\ell$, equations \eqref{geom:eq:kappan}, \eqref{geom:eq:iterationbound}, and \eqref{geom:eq:restartestimates} imply
\[
\widetilde E_0 \le Cr_*^{-d-2} (E_0+K^2R_0^2+\Lambda^{-2}),\qquad \frac{\ell}{r_*R_n} <\frac1{r_*M_0c_b\Lambda}.
\]
By enlarging $\Lambda$ and then decreasing $R_0$ and $\eps_0$, using \eqref{geom:eq:initial}, we may therefore require
\[
C\bigl(\sqrt{\widetilde E_0} +Kr_*R_n+\Lambda_{\rm int}^{-1}\bigr) <\frac14,\qquad
C\bigl(\widetilde E_0 +K^2(r_*R_n)^2+\Lambda_{\rm int}^{-2}\bigr) +\Lambda_{\rm int}^{-1} <\min\{\delta_\theta,\delta_{\rm long}\}.
\]
Here $C$ is fixed sufficiently large for all the iteration estimates below. These requirements hold uniformly in $x,p,n$ and $0<\eps<\eps_0$.

Keep $n$ fixed and index the interior iteration by $j$, with radii $\widetilde R_j:=\theta^jr_*R_n$. Let $J$ be the largest index with $\widetilde R_J\ge\Lambda_{\rm int}\ell$. It exists because \eqref{geom:eq:stopcompatibility} gives $r_*R_n>4\Lambda_{\rm int}\ell$, and satisfies $\Lambda_{\rm int}\ell\le\widetilde R_J <\Lambda_{\rm int}\ell/\theta$.

Starting from $\widetilde\pi^{(0)}:=\widetilde\pi$, we use coordinates $x'=x+\widetilde R_j\widetilde L_j\zeta_j$ and $y'=\widetilde q_j+ \widetilde R_j\widetilde L_j^{-T}\eta_j$. The plan $\widetilde\pi^{(j)}$ is the corresponding pushforward of $\pi_\eps$, multiplied by $\widetilde R_j^{-d}$, and $\widetilde E_j:=\cE(\widetilde\pi^{(j)},1)$.

Proceed inductively under the bounds $\|\widetilde L_0^{-1}\widetilde L_j\|, \|\widetilde L_j^{-1}\widetilde L_0\|\le2$ and the inclusion of $B_3$ in both current supports. Both conditions hold at $j=0$. The density rescaling gives $\widetilde\kappa_j\le CK\widetilde R_j$. Density matching gives $|\det\widetilde L_j|^2 =\rho_1(\widetilde q_j)/\rho_0(x)$, so the marginal density bounds remain uniform.

Whenever $\widetilde E_j+\widetilde\kappa_j^2 +(\ell/\widetilde R_j)^2\le\delta_\theta$, apply the interior version of \cref{geom:prop:improvement}. It supplies $H_j=H_j^T\succ0$ and $b_j$, with $O_j=\Id$, and $|b_j|+\|H_j-\Id\| \le C(\widetilde E_j^{1/2}+\widetilde\kappa_j)$. Define $\widetilde L_{j+1}:=\widetilde L_jH_j^{-1/2}$ and $\widetilde q_{j+1}:=\widetilde q_j+ \widetilde R_j\widetilde L_j^{-T}b_j$. The consecutive coordinates satisfy $\zeta_j=\theta H_j^{-1/2}\zeta_{j+1}$ and $\eta_j=b_j+\theta H_j^{1/2}\eta_{j+1}$. In particular, the source center remains $x$.

As in \eqref{geom:eq:recurrence}--\eqref{geom:eq:summability}, every constructed level satisfies
\[
\begin{aligned}
\widetilde E_{j+1} &\le a\widetilde E_j +CK^2\widetilde R_j^2 +C(\ell/\widetilde R_j)^2,\\
\widetilde E_j &\le C\theta^{2\gamma j} \bigl(\widetilde E_0+K^2(r_*R_n)^2\bigr) +C(\ell/\widetilde R_j)^2,\\
\sum_{k=0}^j (\sqrt{\widetilde E_k}+K\widetilde R_k) &\le C\bigl(\sqrt{\widetilde E_0} +Kr_*R_n+\Lambda_{\rm int}^{-1}\bigr).
\end{aligned}
\]
The constants are independent of $n,j,J,\eps$. Our choices imply $\widetilde E_j+\widetilde\kappa_j^2 +(\ell/\widetilde R_j)^2<\delta_\theta$. The matrix-product estimate used in the boundary iteration gives
\[
\max\{ \|\widetilde L_0^{-1}\widetilde L_j\|, \|\widetilde L_j^{-1}\widetilde L_0\| \} \le \exp\left( C\sum_{k<j} (\sqrt{\widetilde E_k}+K\widetilde R_k) \right) <2.
\]
Moreover, $|b_j|+\|H_j-\Id\|<1/4$. Thus every eigenvalue of $H_j$ belongs to $(3/4,5/4)$, and $\|H_j^{1/2}\|,\|H_j^{-1/2}\|<2$. Consequently,
\begin{equation}\label{m:display:047}
\theta H_j^{-1/2}B_3\subset B_{6\theta}\subset B_3, \qquad b_j+\theta H_j^{1/2}B_3 \subset B_{1/4+6\theta}\subset B_3.
\end{equation}
Both transformed supports at level $j+1$ therefore contain $B_3$. This closes the induction and constructs every level $0\le j\le J$.

At $j=J$, the parameter is $\eps/\widetilde R_J^{d+2}$, and the preceding bounds give
\[
\widetilde E_J+\frac{\ell}{\widetilde R_J} \le C\bigl(\widetilde E_0+K^2(r_*R_n)^2 +\Lambda_{\rm int}^{-2}\bigr) +\Lambda_{\rm int}^{-1} <\delta_{\rm long}.
\]
Both marginal supports contain $B_3$, so the interior version of \cref{geom:lem:long} yields
\[
\sup_{(\zeta,\eta)\in \spt\widetilde\pi^{(J)}\cap\#B_{1/2}} |\zeta-\eta| \le C\left( \widetilde E_J^{1/(d+2)} +\frac{\ell}{\widetilde R_J} \right) \le C.
\]
For $y\in S_{\eps,x}$, \eqref{eqv:eq:support} and the coordinate change give $(0,\eta)\in\spt\widetilde\pi^{(J)}\cap\#B_{1/2}$, since $x$ has source coordinate $0$. Hence $|\eta|\le C$. For any $y,y'\in S_{\eps,x}$, with target coordinates $\eta,\eta'$, it follows that
\[
|y-y'| =\widetilde R_J |(\widetilde L_J^{-1})^T(\eta-\eta')| \le C\widetilde R_J \le \frac{C\Lambda_{\rm int}}{\theta}\ell.
\]
Thus $\diam S_{\eps,x}\le C\ell$, since $\Lambda_{\rm int}$ and $\theta$ are fixed.

It remains to consider $\dist(x,\partial X)\ge R_0/M_0$. Choose a fixed smooth domain $D\Subset\intr X$ containing $\{x\in X:\dist(x,\partial X)\ge R_0/M_0\}$. By \cref{ass:geometrydata,lem:classicalfinite}, the hypotheses of \cite[Theorem~3.1]{GGK} hold. Apply that theorem with entropy exponent $p=2$ and regularization parameter $\eps/2$; its objective then agrees with \eqref{eq:problem} up to an additive constant. Its diameter estimate passes to our closed sections by taking closures, giving $\diam S_{\eps,x}\le C\ell$ uniformly for $x\in D$, after decreasing $\eps_0$. Together with the preceding interior iteration and \eqref{geom:eq:boundarystrip}, this covers every $x\in X$. Interchanging the marginals gives
\begin{equation}\label{geom:eq:diameters}
\sup_{x\in X}\diam S_{\eps,x} +\sup_{y\in Y}\diam K_{\eps,y} \le C\ell,\qquad 0<\eps<\eps_0.
\end{equation}

As for the bookkeeping, the constants are chosen in the following order: the comparison neighborhoods, matrix and density bounds, the exponents and $\theta$, then $c_b,M_0,r_*$, $\Lambda_{\rm int}$, $\Lambda$, $R_0$, and finally $\eps_0$. %
\subsection{Section geometry and uniform curvature}
\label{geom:sec:geometry}

We deduce the section geometry and Hessian bounds from \eqref{geom:eq:diameters}.

\begin{lemma}[Section geometry from diameter bounds]
\label{geom:lem:diametertoballs}
Let $X,Y$ be convex bodies, with continuous marginal densities satisfying $0<m\le\rho_i\le M<\infty$ on their respective supports. If \eqref{geom:eq:diameters} holds, then \eqref{eq:balls} and \eqref{eq:scales} hold uniformly in the base points, for all sufficiently small $\eps$.
\end{lemma}

\begin{proof}
Under the present hypotheses, \cite[Lemma~2.2]{GdBN} gives
$u_\eps\in C^1(X)$ and $v_\eps\in C^1(Y)$, together with
the barycenter identities \eqref{eqv:eq:centers}.
Fix $x\in X$. By \eqref{eqv:eq:centers}, $z_{\eps,x}\in S_{\eps,x}$ and $\zeta_{\eps,y}\in K_{\eps,y}$. For $y\in S_{\eps,x}$, we also have $x\in K_{\eps,y}$. Hence \eqref{geom:eq:diameters} gives
\begin{equation}\label{m:display:048}
|\nabla_yq_\eps(x,y)| =|x-\zeta_{\eps,y}| \le\diam K_{\eps,y} \le C\ell, \qquad y\in S_{\eps,x}.
\end{equation}

Set $H_x:=\max_Yq_\eps(x,\cdot)$ and choose $y_+\in Y$ attaining this maximum. The marginal identity \eqref{eqv:eq:marginal} gives $H_x>0$. For $\eps$ sufficiently small, $\diam S_{\eps,x}<\diam Y$, so $S_{\eps,x}\ne Y$. Continuity along a segment from $y_+$ to a point of $Y\setminus S_{\eps,x}$ gives $y_0\in S_{\eps,x}$ with $q_\eps(x,y_0)=0$. The segment $[y_0,y_+]$ lies in $S_{\eps,x}$ by convexity. Thus \eqref{m:display:048} yields
\begin{equation}\label{m:display:049}
\begin{aligned}
H_x &=\int_0^1 \nabla_yq_\eps(x,y_0+s(y_+-y_0)) \cdot(y_+-y_0)\dd s\le C\ell\,\diam S_{\eps,x} \le C\ell^2.
\end{aligned}
\end{equation}
Moreover, \eqref{eqv:eq:marginal} and \eqref{geom:eq:diameters} imply
\begin{equation}\label{m:display:050}
\ell^{d+2}=\eps =\int_{S_{\eps,x}}q_\eps(x,y)\rho_1(y)\dd y \le MH_x|S_{\eps,x}|,\qquad |S_{\eps,x}|\le C\ell^d.
\end{equation}
Using the volume upper bound gives $H_x\ge c\ell^2$, while using \eqref{m:display:049} gives $|S_{\eps,x}|\ge c\ell^d$. Interchanging the marginals proves \eqref{eq:scales}.

By \eqref{eqv:eq:centers}, $z_{\eps,x}$ is the centroid of $S_{\eps,x}$ with respect to the density $\rho_1$. The bounds $|S_{\eps,x}|\ge c\ell^d$ and $\diam S_{\eps,x}\le C\ell$, together with \cref{eqv:lem:centroid}, give $B(z_{\eps,x},c\ell)\subset S_{\eps,x}$. Since $z_{\eps,x}\in S_{\eps,x}$, the diameter bound also gives $S_{\eps,x}\subset B(z_{\eps,x},C\ell)$, after enlarging $C$. The same argument applies to $K_{\eps,y}$, proving \eqref{eq:balls}.
\end{proof}

\begin{proof}[Completion of the proof of
\cref{thm:geometry}] The estimates \eqref{eq:balls} and \eqref{eq:scales} follow from \cref{geom:lem:diametertoballs}. It remains to prove the Hessian bounds \eqref{eq:hessians}. Indeed, let $x\in\intr X$ and $y\in\partial S_{\eps,x}\cap\intr Y$. Continuity gives $q_\eps(x,y)=0$, so $x\in K_{\eps,y}$. If $x\in\intr K_{\eps,y}$, the function $q_\eps(\cdot,y)$ has a local minimum at $x$, so $\nabla_xq_\eps(x,y)=0$. Concavity would then give $q_\eps(\cdot,y)\le0$ on $X$, contradicting the second identity in \eqref{eqv:eq:marginal}. Thus $x\in\partial K_{\eps,y}$, and \eqref{eq:balls} gives
\begin{equation}\label{eq:freedenominator}
c\ell\le|x-\zeta_{\eps,y}|\le C\ell, \qquad y\in\partial S_{\eps,x}\cap\intr Y.
\end{equation}
By \eqref{eq:scales} and the density bounds, $\nu(S_{\eps,x})\asymp\ell^d$. Apply \cref{eqv:lem:cone} to $S_{\eps,x}\subset Y$, using \eqref{eq:balls} with outer radius $C\ell$. For $\eps$ sufficiently small, this gives
\[
\int_{\partial S_{\eps,x}\cap\intr Y} ((y-z_{\eps,x})\cdot e)^2\dd\HH(y) \asymp\ell^{d+1}, \qquad |e|=1.
\]
Consequently, \eqref{eqv:eq:hessian}, \eqref{eq:freedenominator}, and the density bounds yield, for almost every $x\in X$,
\[
e^TD^2u_\eps(x)e \asymp \ell^{-d-1} \int_{\partial S_{\eps,x}\cap\intr Y} ((y-z_{\eps,x})\cdot e)^2\dd\HH(y) \asymp1, \qquad |e|=1.
\]
All constants are independent of $x,e,\eps$. Interchanging the marginals proves \eqref{eq:hessians}.
\end{proof}

\subsection{Conditional covariance and section centers}
\label{sec:conditional-regularity}

We apply \cref{thm:geometry} to the conditional laws \eqref{eq:canonical-conditional-laws}. The covariance estimate gives the transport-cost lower bound. We abbreviate $S_x=S_{\eps,x}$, $K_y=K_{\eps,y}$, $z_x=z_{\eps,x}$, and $\zeta_y=\zeta_{\eps,y}$. Define the conditional means and covariance matrices by
\begin{equation}\label{eq:conditional-covariance-definition}
\begin{gathered}
m_\eps(x):=\int_Y y\,\pi_\eps^x(\dd y),\qquad \widehat m_\eps(y):=\int_X x\,\widehat\pi_\eps^y(\dd x),\\
\Sigma_\eps(x):=\int_Y(y-m_\eps(x))\otimes(y-m_\eps(x))\,\pi_\eps^x(\dd y),\\
\widehat\Sigma_\eps(y):=\int_X(x-\widehat m_\eps(y))\otimes(x-\widehat m_\eps(y))\,\widehat\pi_\eps^y(\dd x).
\end{gathered}
\end{equation}
Moreover, define the transport-cost excess and the regularization penalty by
\begin{equation}\label{eq:costpenalty}
 C_\eps=\frac12\int|x-y|^2\dd\pi_\eps-\OT,\qquad
 P_\eps=\frac\eps2\int h_\eps^2\dd(\mu\otimes\nu).
\end{equation}

\begin{proposition}[Conditional covariance and excess transport cost]
\label{prop:conditional-covariance}
Under \cref{ass:geometrydata}, there are $c,C,\eps_0>0$ such that, for $0<\eps<\eps_0$,
\begin{equation}\label{eq:conditional-covariance-bounds}
\begin{aligned}
c\ell^2\Id\preceq\Sigma_\eps(x)\preceq C\ell^2\Id, &\qquad x\in X,\\
c\ell^2\Id\preceq\widehat\Sigma_\eps(y)\preceq C\ell^2\Id, &\qquad y\in Y.
\end{aligned}
\end{equation}
The transport-cost excess $C_\eps$ defined in \eqref{eq:costpenalty} satisfies
\begin{equation}\label{eq:cost-excess-sharp}
c\ell^2\le C_\eps\le C\ell^2.
\end{equation}
\end{proposition}

\begin{proof}
Fix $x\in X$. By \eqref{eqv:eq:centers}, concavity, and \eqref{eqv:eq:marginal},
\begin{equation}\label{eq:uniform-central-slack}
q_\eps(x,z_x)\ge\frac{1}{\nu(S_x)}\int_{S_x}q_\eps(x,y)\dd\nu(y)=\frac{\eps}{\nu(S_x)}\ge c\ell^2,
\end{equation}
where the last inequality uses $\nu(S_x)\le M|S_x|\le C\ell^d$ from \eqref{eq:scales} and $\eps=\ell^{d+2}$. By \eqref{m:display:048}, $|\nabla_yq_\eps(x,y)|\le C\ell$ on $S_x$. Choose $b>0$ smaller than the constant~$c$ in \eqref{eq:balls}. For $y\in B(z_x,b\ell)$, integration along $[z_x,y]\subset S_x$ gives $q_\eps(x,y)\ge q_\eps(x,z_x)-C\ell|y-z_x|$. Using \eqref{eq:uniform-central-slack} and decreasing $b$ if necessary, we obtain
\begin{equation}\label{eq:positive-conditional-core}
\begin{gathered}
B(z_x,b\ell)\subset S_x,\qquad q_\eps(x,y)\ge c\ell^2,\\
\frac{\dd\pi_\eps^x}{\dd y}(y)=\frac{(q_\eps(x,y))_+\rho_1(y)}{\eps}\ge c\ell^{-d}\qquad(y\in B(z_x,b\ell)).
\end{gathered}
\end{equation}
For every unit vector $e$ and $a\in\R^d$, symmetry of the ball gives
\begin{equation}\label{eq:covariance-ball-moment}
\begin{aligned}
\int_{B(z_x,b\ell)}((y-a)\cdot e)^2\dd y &=\frac{\omega_d}{d+2}(b\ell)^{d+2} +\omega_d(b\ell)^d((z_x-a)\cdot e)^2\ge\frac{\omega_d}{d+2}(b\ell)^{d+2},
\end{aligned}
\end{equation}
where $\omega_d=|B_1|$. Taking $a=m_\eps(x)$ and using \eqref{eq:positive-conditional-core} gives $e^T\Sigma_\eps(x)e\ge c\ell^2$. Since $m_\eps(x)\in S_x$, $e^T\Sigma_\eps(x)e \le\int_Y|y-m_\eps(x)|^2\pi_\eps^x(\dd y) \le(\diam S_x)^2\le C\ell^2$. Interchanging the marginals proves \eqref{eq:conditional-covariance-bounds}.

Set $G_0(x,y):=u_0(x)+v_0(y)-x\cdot y$. The classical contact identity gives $G_0(x,y)=v_0(y)-v_0(T(x))-x\cdot(y-T(x))$. Taylor's formula, $\nabla v_0(T(x))=x$, and \cref{lem:classicalfinite} then yield
\begin{equation}\label{eq:classical-gap-two-sided}
c|y-T(x)|^2\le G_0(x,y)\le C|y-T(x)|^2, \qquad C_\eps=\int G_0\dd\pi_\eps,
\end{equation}
where the last identity is \eqref{eq:classicalgap-exact}. Since $\int_Y(y-m_\eps(x))\pi_\eps^x(\dd y)=0$,
\begin{equation}\label{eq:conditional-variance-decomposition}
\int|y-T(x)|^2\dd\pi_\eps =\int_X\tr\Sigma_\eps(x)\dd\mu(x) +\int_X|m_\eps(x)-T(x)|^2\dd\mu(x).
\end{equation}
Thus $C_\eps\ge c\int_X\tr\Sigma_\eps(x)\dd\mu(x) \ge c\ell^2$. The upper bound follows from \eqref{geom:eq:globalenergy} and the nonnegativity of the regularization penalty.
\end{proof}

\paragraph{Sections centered at the conditional means.}
The ball inclusions also hold with the conditional means as centers:
\begin{align}\label{eq:conditional-mean-section-balls}
B(m_\eps(x),c\ell) &\subset S_x\subset B(m_\eps(x),C\ell), &&x\in X,
\end{align}
and analogously for $\widehat m_\eps(y)$. Indeed, fix $x\in X$ and a unit vector $e$, and set $a_e:=\min_{y\in S_x}e\cdot y$. For $y\in S_x$, $0\le e\cdot y-a_e \le\operatorname{width}_e(S_x)\le C\ell$. Consequently,
\begin{equation}\label{eq:conditional-mean-width}
c\ell^2\le e^T\Sigma_\eps(x)e \le\int_{S_x}(e\cdot y-a_e)^2\,\pi_\eps^x(\dd y)
\le\operatorname{width}_e(S_x) \bigl(e\cdot m_\eps(x)-a_e\bigr).
\end{equation}
The first inequality uses \eqref{eq:conditional-covariance-bounds}, and the second uses that the mean minimizes the second moment. Since $\operatorname{width}_e(S_x)\le C\ell$, we obtain $e\cdot m_\eps(x)-a_e\ge c\ell$ for every unit $e$. The supporting-halfspace representation of $S_x$ gives $B(m_\eps(x),c\ell)\subset S_x$. The outer inclusion follows from $m_\eps(x)\in S_x$ and $\diam S_x\le C\ell$. Interchanging the marginals proves \eqref{eq:conditional-mean-section-balls}.

Since $S_x\subset Y$ and $K_y\subset X$, these inclusions
give $\dist(m_\eps(x),\partial Y)\ge c\ell$ for $x\in X$
and $\dist(\widehat m_\eps(y),\partial X)\ge c\ell$
for $y\in Y$.

\section{Localization around the Brenier graph}\label{sec:localization}
The diameter estimates in \cref{thm:geometry} do not yet locate the sections relative to the Brenier graph. We now obtain this localization by controlling the conditional mean and the potential error. Throughout this section, we impose \cref{ass:geometrydata}, use the normalization \eqref{eq:gauge}, and take $\eps$ small enough for the validity of \cref{thm:geometry} and the energy bound \eqref{geom:eq:globalenergy}. We write $\ell=\eps^{1/(d+2)}$ and abbreviate the potential errors by
\begin{equation}\label{eq:potentialerrors}
r_\eps=u_\eps-u_0,\qquad s_\eps=v_\eps-v_0.
\end{equation}

\subsection{A potential for the conditional mean}
\label{bmo:sec:correction}

We construct a potential for the conditional mean $m_\eps$ defined in \eqref{eq:conditional-covariance-definition}. Set
\begin{equation}\label{bmo:eq:correctiondef}
 a_\eps(x):=\frac1{2\eps}
       \int_Y(q_\eps(x,y)_+)^2\dd\nu(y),
 \qquad
 \Theta_\eps:=u_\eps+a_\eps.
\end{equation}
The analogous target correction is defined by interchanging the marginals. We will use the barycenter identity \eqref{eqv:eq:centers} in the form
\begin{equation}\label{bmo:eq:restrictedmean}
 z_{\eps,x}
 =\frac{\int_{S_{\eps,x}}y\dd\nu(y)}
        {\nu(S_{\eps,x})},
 \qquad
 \int_{S_{\eps,x}}(y-z_{\eps,x})\dd\nu(y)=0.
\end{equation}

\begin{lemma}[Conditional-mean potential]
\label{bmo:lem:conditional}
Under \cref{ass:geometrydata}, for all sufficiently small~$\eps$,
\begin{equation}\label{bmo:eq:conditional}
 \nabla\Theta_\eps=m_\eps,\qquad
 c\ell^2\le a_\eps\le C\ell^2,\qquad
 \|\nabla a_\eps\|_\infty\le C\ell,\qquad
 c\Id\preceq D^2\Theta_\eps\preceq C\Id.
\end{equation}
The first three assertions hold everywhere on $X$ and the Hessian estimates hold a.e.\ on $X$. The constants are independent of $\eps$. The same assertions hold with the marginals interchanged.
\end{lemma}

\begin{proof}
Since $t\mapsto(t_+)^2$ is continuously differentiable, differentiation under the integral and \eqref{eqv:eq:marginal} give
\begin{equation}\label{m:display:052}
 \begin{aligned}
 \nabla a_\eps(x)
 &=\frac1\eps\int_Y
       q_\eps(x,y)_+(y-z_{\eps,x})\dd\nu(y)=m_\eps(x)-z_{\eps,x}.
 \end{aligned}
\end{equation}
Thus $\nabla\Theta_\eps=m_\eps$. The height estimate \eqref{m:display:049} gives
\[
 a_\eps(x)\le
 (2\eps)^{-1}\max_Yq_\eps(x,\cdot)
       \int_Yq_\eps(x,y)_+\dd\nu(y)
 \le C\ell^2.
\]
For the lower bound, Cauchy--Schwarz yields
\begin{equation}\label{m:display:053}
 a_\eps(x)
 =\frac\eps2\int_Yh_\eps(x,y)^2\dd\nu(y)
 \ge\frac{\eps}{2\nu(S_{\eps,x})}
       \left(\int_Yh_\eps(x,y)\dd\nu(y)\right)^2
 =\frac{\eps}{2\nu(S_{\eps,x})}
 \ge c\ell^2.
\end{equation}
Both $m_\eps(x)$ and $z_{\eps,x}$ belong to $S_{\eps,x}$, so \eqref{m:display:052} and \eqref{geom:eq:diameters} give $|\nabla a_\eps(x)|\le C\ell$.

For fixed $x$, concavity and the positive row integral imply $\{y\in Y:q_\eps(x,y)=0\}\subset\partial S_{\eps,x}$. This set has zero Lebesgue measure and hence zero $\nu$-measure. Since $\nabla_xq_\eps(x,y)=y-z_{\eps,x}$ is bounded, dominated convergence justifies differentiating $m_\eps(x)=\eps^{-1}\int_Yyq_\eps(x,y)_+\dd\nu(y)$ at every interior base point. Using \eqref{bmo:eq:restrictedmean}, we obtain
\begin{equation}\label{bmo:eq:covariance}
 \begin{aligned}
 Dm_\eps(x)
 &=\frac1\eps\int_{S_{\eps,x}}
       y\otimes(y-z_{\eps,x})\dd\nu(y)=\frac1\eps\int_{S_{\eps,x}}
       (y-z_{\eps,x})\otimes(y-z_{\eps,x})\dd\nu(y).
 \end{aligned}
\end{equation}
The ball inclusions \eqref{eq:balls} and the density bounds imply, for every unit vector $e$,
\[
 c\ell^{d+2}
 \le\int_{S_{\eps,x}}
       ((y-z_{\eps,x})\cdot e)^2\dd\nu(y)
 \le C\ell^{d+2}.
\]
Since $\eps=\ell^{d+2}$ and $D^2\Theta_\eps=Dm_\eps$, this proves the Hessian estimates. The gradient identities extend to the closed supports by continuity. In particular, $m_\eps$ is uniformly Lipschitz and strongly monotone on $X$.
\end{proof}

Recall the transport-cost excess and penalty:
\begin{equation}\label{m:display:054}
 C_\eps=\frac12\int|x-y|^2\dd\pi_\eps-\OT,
 \qquad
 P_\eps=\frac\eps2\int h_\eps^2\dd(\mu\otimes\nu).
\end{equation}
By \eqref{bmo:eq:correctiondef}, $P_\eps=\int_Xa_\eps\dd\mu$, so \cref{bmo:lem:conditional} and \eqref{geom:eq:globalenergy} give $P_\eps\asymp\ell^2$ and $\QOT_\eps-\OT=C_\eps+P_\eps\asymp\ell^2$.

The identity $h_\eps=(q_\eps)_+/\eps$ gives $\int q_\eps\dd\pi_\eps =\eps\int h_\eps^2\dd(\mu\otimes\nu)=2P_\eps$. Together with \eqref{eq:classical-gap-two-sided},
\[
 \int_Xr_\eps\dd\mu+\int_Ys_\eps\dd\nu
 =-\int(q_\eps+G_0)\dd\pi_\eps
 =-2P_\eps-C_\eps.
\]
Under the normalization \eqref{eq:gauge}, the two integrals agree, so
\begin{equation}\label{bmo:eq:means}
 \int_Xr_\eps\dd\mu
 =\int_Ys_\eps\dd\nu
 =-P_\eps-\tfrac12C_\eps.
\end{equation}
Recall \eqref{bmo:eq:correctiondef} and define
\begin{equation}\label{bmo:eq:wdef}
 \begin{aligned}
 w&:=\Theta_\eps-u_0=r_\eps+a_\eps,
 &\bar w&:=\int_Xw\dd\mu=-\tfrac12C_\eps,\\
 W&:=w-\bar w,
 &b&:=\nabla w=m_\eps-T.
 \end{aligned}
\end{equation}
We suppress the dependence of $w,\bar w,W,b$ on $\eps$. In particular, $\int_XW\dd\mu=0$ and $\nabla W=b$.

For a convex body $D$, set $D_{x,r}:=D\cap B(x,r)$. For $f\in L^1(D)$, define
\begin{equation}\label{m:display:051}
 \begin{aligned}
 \langle f\rangle_{D_{x,r}}
 &:=\frac1{|D_{x,r}|}\int_{D_{x,r}}f(z)\dd z,\\
 [f]_{\BMO(D)}
 &:=\sup_{\substack{x\in D\\0<r\le2\diam D}}
       \frac1{|D_{x,r}|}
       \int_{D_{x,r}}
       |f(z)-\langle f\rangle_{D_{x,r}}|\dd z.
 \end{aligned}
\end{equation}
This seminorm is unchanged by adding constants. Averages against a density bounded above and away from zero define an equivalent seminorm.

Conditional Jensen's inequality gives $|b(x)|^2 \le\int_Y|y-T(x)|^2\pi_\eps^x(\dd y)$. Integrating and using \eqref{geom:eq:globalLtwo}, we obtain
\begin{equation}\label{bmo:eq:averaged}
 \|b\|_{L^2(\mu)}^2
 \le\int|y-T(x)|^2\dd\pi_\eps
 \le C\ell^2.
\end{equation}
By \cref{bmo:lem:conditional,lem:classicalfinite}, $b$ is uniformly Lipschitz. Convexity and the density lower bound give $\mu(X\cap B(x,r))\ge cr^d$ uniformly in $x\in X$ and sufficiently small $r>0$. Averaging over this set and applying Cauchy--Schwarz yields
\[
 \begin{aligned}
 |b(x)|
 &\le Cr+
 \frac1{\mu(X\cap B(x,r))}
       \int_{X\cap B(x,r)}|b(z)|\dd\mu(z)\\
 &\le Cr+Cr^{-d/2}\|b\|_{L^2(\mu)}
 \le C(r+\ell r^{-d/2}).
 \end{aligned}
\]
Choose $r=\ell^{2/(d+2)}$. Since $\int_XW\dd\mu=0$ and $\nabla W=b$, $\|W\|_\infty\le\diam X\,\|b\|_\infty$. Also, $[W]_{\BMO(X)}\le2\|W\|_\infty$. Consequently,
\begin{equation}\label{bmo:eq:rough}
 \|b\|_\infty\le C\ell^{2/(d+2)},\qquad
 \|W\|_\infty\le C\ell^{2/(d+2)},\qquad
 [W]_{\BMO(X)}\to0,
\end{equation}
where all norms are over $X$ and the convergence is as $\eps\downarrow0$.

\subsection{Coarse affine approximation from the stopped iteration}
\label{bmo:sec:coarseinput}

We use the stopped iteration to approximate $\nabla W=m_\eps-T$ by constants on relative balls.

Recall the iteration of Section~\ref{geom:sec:iteration}. At $p\in\partial X$, initialize with $L_0=DT(p)^{-1/2}O$ and $q_0=T(p)$. At radius $R_n=\theta^nR_0$, the coordinates are
\begin{equation}\label{bmo:eq:itercoords}
 x=p+R_nL_n\xi,\qquad
 y=q_n+R_n(L_n^{-1})^T\zeta.
\end{equation}
The plan $\pi^{(n)}$ is the pushforward of $\pi_\eps$ under the inverse coordinate map, multiplied by $R_n^{-d}$. Its excess is
\begin{equation}\label{m:display:055}
 E_n=\cE(\pi^{(n)},1)
 =\int_{\#B_1}|\xi-\zeta|^2
       \dd\pi^{(n)}(\xi,\zeta).
\end{equation}
For the largest index $N$ with $R_N\ge\Lambda\ell$,
\begin{align}\label{bmo:eq:iterest}
 E_0
 &\le CR_0^2+C\ell^2R_0^{-d-2},
 \nonumber\\
 E_{n+1}
 &\le\theta^{2\gamma}E_n+CK^2R_n^2
       +C(\ell/R_n)^2
       &&(n<N),
 \nonumber\\
 E_n
 &\le C\theta^{2\gamma n}(E_0+K^2R_0^2)
       +C(\ell/R_n)^2
       &&(n\le N),
 \nonumber\\
 \sum_{n=0}^N(\sqrt{E_n}+KR_n)
 &\le C\bigl(\sqrt{E_0}+KR_0+\Lambda^{-1}\bigr).
\end{align}
The first estimate uses the Lipschitz bound for $DT$ in \cref{lem:classicalfinite}. The matrices satisfy
\begin{equation}\label{bmo:eq:matrixupdate}
 L_{n+1}=L_nH_n^{-1/2}O_n,\qquad
 \|H_n-\Id\|\le C(\sqrt{E_n}+KR_n),
\end{equation}
where $O_n$ is orthogonal.

For the interior restart, \eqref{geom:eq:restartestimates} gives
\begin{equation}\label{bmo:eq:restart}
 \widetilde E_0
 \le Cr_*^{-d-2}E_n+C(KR_n)^2,\qquad
 \widetilde R_0=r_*R_n,
\end{equation}
and $\|A_*-\Id\|\le CKR_n$. The restarted iteration satisfies the same recurrence and summability estimates, with stopping factor $\Lambda_{\rm int}$. The constants in these estimates and in the uniform matrix bounds
are independent of the stopping factors and the initial radius,
provided the stated smallness and compatibility conditions hold.
In the following proof we use new instances of the stopped
iterations, choosing the stopping factors, initial radius and
regularization threshold for each prescribed accuracy $\eta$. Recall that $X_{x,r}=X\cap B(x,r)$.

\begin{proposition}[Coarse-scale approximation of the conditional-mean error]
\label{bmo:prop:coarsefromiteration}
\label{bmo:ass:coarse}
Under \cref{ass:geometrydata}, the stopped iteration and \cref{lem:classicalfinite} have the following consequence. For every $\eta>0$, there exist $L_\eta\ge1$ and $\eps_\eta>0$ such that, for $0<\eps<\eps_\eta$ and $R=L_\eta\ell$,
\begin{equation}\label{bmo:eq:coarse}
 \sup_{x\in X}\inf_{a\in\R^d}
 \left(
 \frac1{|X_{x,8R}|}
 \int_{X_{x,8R}}|m_\eps(z)-T(z)-a|^2\dd z
 \right)^{1/2}
 \le\eta R.
\end{equation}
The analogous estimate holds for $\widehat m_\eps-S$, with the same choices of $L_\eta$ and $\eps_\eta$.
\end{proposition}

\begin{proof}
Fix $\eta>0$. Use the uniform matrix bounds established after \eqref{geom:eq:summability} and in the interior iteration of Section~\ref{geom:sec:iteration}. Choose $c_0\in(0,1]$ so that
$B(x_c,c_0\rho)\subset x_c+\rho\mathsf L B_1$
for every source matrix $\mathsf L$ in the boundary and
restarted interior iterations, and so that
\[
 2c_0\sup_{x\in X}\|DT(x)^{1/2}\|\le1.
\]
The latter condition also gives the required inclusion for
the interior iterations initialized below by
$\widetilde L_0=DT(x)^{-1/2}$, under the inductive bound
$\|\widetilde L_j^{-1}\widetilde L_0\|\le2$. Fix $h>8/c_0$, choose $K_0$ with
\begin{equation}\label{m:display:060}
 r_*M_0K_0>2h/\theta,
\end{equation}
and fix the scalar $H>(K_0+8)/c_0$. These constants depend only on the data and the fixed iteration parameters.

Let $\delta>0$ be sufficiently small, to be fixed at the end. Choose $\Lambda_{\rm int},\Lambda$ sufficiently large, subject to \eqref{geom:eq:stopcompatibility}, and then $R_0$ sufficiently small that $R_0+KR_0+\Lambda_{\rm int}^{-1}+\Lambda^{-1} \le c\delta$. Choose $L_\eta\ge1$ with $hL_\eta>\Lambda_{\rm int}$, $HL_\eta>\Lambda$, and $K_0L_\eta>c_b\Lambda$. Finally, decrease $\eps_\eta$ so that, for $\eps<\eps_\eta$ and $R=L_\eta\ell$, $\ell R_0^{-(d+2)/2}\le c\delta$, $HR/\theta<R_0$, and $K_0R<R_0/M_0$. By \eqref{bmo:eq:iterest},
\[
 \sup_{n\le N}\sqrt{E_n}\le C\delta,\qquad
 \sum_{n=0}^N(\sqrt{E_n}+KR_n)\le C\delta,
\]
uniformly in the boundary point $p$. Define
\begin{equation}\label{m:display:058}
 \mathcal A_n(z)=q_n+B_n(z-p),\qquad
 B_n=(L_n^{-1})^{T}L_n^{-1}.
\end{equation}
Orthogonality of $O_n$ and \eqref{bmo:eq:matrixupdate} give
\begin{equation}\label{m:display:056}
 B_{n+1}-B_n
 =(L_n^{-1})^T(H_n-\Id)L_n^{-1}.
\end{equation}
Since $L_0=DT(p)^{-1/2}O$, we have $B_0=DT(p)$. The uniform matrix bounds therefore imply
\begin{equation}
\label{bmo:eq:fittedderivative}
 \|B_n-DT(p)\|
 \le C\sum_{j<n}(\sqrt{E_j}+KR_j)
 \le C\delta.
\end{equation}

We record an estimate valid at any boundary or interior level. If $(z,y)$ has coordinates $(\xi,\zeta)$ in \eqref{bmo:eq:itercoords}, then $y-\mathcal A_n(z)=R_n(L_n^{-1})^T(\zeta-\xi)$. Conditional Jensen's inequality and \eqref{m:display:055} yield
\begin{equation}\label{m:display:059}
 \int_{X\cap(p+R_nL_nB_1)}
       |m_\eps-\mathcal A_n|^2\dd\mu
 \le
 \int_{(X\cap(p+R_nL_nB_1))\times Y} |y-\mathcal A_n(z)|^2\dd\pi_\eps(z,y)
 \le CR_n^{d+2}E_n.
\end{equation}
For $z\in X\cap(p+R_nL_nB_1)$, $|z-p|\le CR_n$. Convexity of $X$ and the Lipschitz bound for $DT$ in \cref{lem:classicalfinite} give
\begin{equation}\label{bound:An}
\begin{aligned}
 &|\mathcal A_n(z)-T(z)-(q_n-T(p))|\\
 &\quad\le
 \|B_n-DT(p)\|\,|z-p|
 +\tfrac12\Lip(DT)|z-p|^2
\le C\bigl(\|B_n-DT(p)\|+R_n\bigr)R_n.
\end{aligned}
\end{equation}

Let $x_c\in X$ and suppose $X_{x_c,r}\subset p+R_nL_nB_1$. For sufficiently small $r$, the fixed convex body $X$ satisfies $|X_{x_c,r}|\ge cr^d$, uniformly in $x_c$. Together with $\rho_0\ge m$ and \eqref{m:display:059}, this gives
\begin{equation}\label{m:display:061}
\begin{aligned}
 \left(
 \frac1{|X_{x_c,r}|}
 \int_{X_{x_c,r}}|m_\eps-\mathcal A_n|^2\dd z
 \right)^{1/2}
 &\le Cr^{-d/2}
 \left(
 \int_{X\cap(p+R_nL_nB_1)}
       |m_\eps-\mathcal A_n|^2\dd\mu
 \right)^{1/2} \\
 &\le C(R_n/r)^{d/2}R_n\sqrt{E_n}.
\end{aligned}
\end{equation}
Choosing $a=q_n-T(p)$, the triangle inequality and \eqref{bound:An} imply
\begin{equation}\label{bmo:eq:meanflat}
\begin{aligned}
 &\inf_{a\in\R^d}
 \left(
 \frac1{|X_{x_c,r}|}
 \int_{X_{x_c,r}}|m_\eps-T-a|^2\dd z
 \right)^{1/2}\\
 &\qquad\le
 C\left(
 (R_n/r)^{d/2}\sqrt{E_n}
 +\|B_n-DT(p)\|+R_n
 \right)R_n.
\end{aligned}
\end{equation}
This estimate uses only the coordinate representation, the excess, and the classical regularity. It applies to an interior level with its corresponding source center, radius, and matrices. In particular, its right-hand side is at most $C\delta R_n$ whenever $r\asymp R_n$ and $\sqrt{E_n}+\|B_n-DT(p)\|+R_n\le C\delta$.

Fix $x\in X$ and set $t=\dist(x,\partial X)$. Suppose first that $t\le K_0R$. Choose $p\in\partial X$ with $|x-p|=t$ and let $n=\max\{k\ge0:R_k\ge HR\}$ in the boundary iteration based at $p$. This index is well-defined because $HR<R_0$ and $R_k\to0$. Maximality gives $HR\le R_n<HR/\theta$. Also $HR>\Lambda\ell$, so $n\le N$. For $z\in B(x,8R)$, $|z-p|<(K_0+8)R<c_0R_n$, whence
\[
 B(x,8R)\subset B(p,c_0R_n)
 \subset p+R_nL_nB_1.
\]
Apply \eqref{bmo:eq:meanflat} with $x_c=x$ and $r=8R$. Since $R_n/R\in[H,H/\theta)$ and \eqref{bmo:eq:fittedderivative} holds, the infimum in \eqref{bmo:eq:coarse} is at most $C\delta R_n\le C\delta R$.

For $t>K_0R$, we initialize an interior iteration centered at $x$. We specify its initial radius~$\widetilde R_0$, source matrix $\widetilde L_0$, and target center $\widetilde q_0$ in the following two cases.

If $K_0R<t<R_0/M_0$, then $t>K_0L_\eta\ell>c_b\Lambda\ell$. Choose a nearest boundary point $p$ and the level $n$ supplied by \eqref{geom:eq:coverlevel}. Thus $M_0t\le R_n<M_0t/\theta$. By \eqref{geom:eq:coverconstants} and \eqref{geom:eq:stopradius}, $R_n\ge M_0t>2\Lambda\ell/\theta>R_N$, so $n<N$. Use the interior restart constructed in Section~\ref{geom:sec:iteration}, with
\[
 \widetilde R_0=r_*R_n,\qquad
 \widetilde L_0=L_nA_*^{-1/2},\qquad
 \widetilde q_0=\mathcal A_n(x).
\]
The last identity follows from $q_n+R_n(L_n^{-1})^T\xi =q_n+(L_n^{-1})^TL_n^{-1}(x-p)$, where $\xi$ is given by \eqref{m:display:043}. The construction following \eqref{geom:eq:interiorstartballs} gives $B_3$ in both normalized supports and equality of their density values at zero.

By \eqref{bmo:eq:restart}, $\widetilde E_0^{1/2}\le C\delta$. Moreover,
\[
\begin{aligned}
 \|(\widetilde L_0^{-1})^T\widetilde L_0^{-1}-DT(x)\|
&\le
 \|(L_n^{-1})^T(A_*-\Id)L_n^{-1}\|
 +\|B_n-DT(p)\|+\|DT(p)-DT(x)\|\\
 &\le C(KR_n+\delta+t)\le C\delta.
\end{aligned}
\]
Equation~\eqref{m:display:060} also gives $\widetilde R_0>r_*M_0K_0R>2hR/\theta$.

If $t\ge R_0/M_0$, boundary correspondence gives $S(y)\in\partial X$ for $y\in\partial Y$. Hence $t\le|x-S(y)|\le\Lip(S)|T(x)-y|$, and
\[
 \dist(T(x),\partial Y)
 \ge\frac{t}{\Lip(S)}
 \ge\frac{R_0}{M_0\Lip(S)}.
\]
Choose a fixed $\sigma\in(0,1)$ such that $3\sigma\|DT(x)^{-1/2}\|<M_0^{-1}$ and $3\sigma\Lip(S)\|DT(x)^{1/2}\|<M_0^{-1}$ uniformly in $x$. Take $\widetilde R_0=\sigma R_0$, $\widetilde L_0=DT(x)^{-1/2}$, and $\widetilde q_0=T(x)$. Then
\[
 x+\widetilde R_0\widetilde L_0B_3\subset\intr X,
 \qquad
 T(x)+\widetilde R_0(\widetilde L_0^{-1})^TB_3
 \subset\intr Y.
\]
By \cref{geom:lem:scaling} and $\det DT(x)=\rho_0(x)/\rho_1(T(x))$, the two transformed density values at zero are equal:
\[
 |\det\widetilde L_0|\rho_0(x)
 =|\det\widetilde L_0|^{-1}\rho_1(T(x))
 =\sqrt{\rho_0(x)\rho_1(T(x))}.
\]
Moreover, $(\widetilde L_0^{-1})^T\widetilde L_0^{-1}=DT(x)$.

The initialization argument proving \eqref{geom:eq:initial} applies at $x$, with $O=\Id$ and radius $\widetilde R_0$. It uses the Taylor estimates for the normalized classical map and its inverse, together with \eqref{geom:eq:globalLtwo}. By \cref{lem:classicalfinite}, both derivatives are Lipschitz, so that argument gives
\[
 \widetilde E_0
 \le C\widetilde R_0^2
      +C\ell^2\widetilde R_0^{-d-2}
 =C(\sigma R_0)^2
      +C\ell^2(\sigma R_0)^{-d-2}.
\]
Decrease $\eps_\eta$ so that $\ell(\sigma R_0)^{-(d+2)/2}\le c\delta$ and $hR/\theta<\sigma R_0$. Then $\widetilde E_0^{1/2}\le C\delta$ and $\widetilde R_0>hR/\theta$. The density calculation in \eqref{geom:eq:kappan}, with the graph terms omitted, gives $\widetilde\kappa_0\le CK\widetilde R_0$.

Both initializations therefore satisfy
\[
 \widetilde E_0^{1/2}
 +K\widetilde R_0
 +\|(\widetilde L_0^{-1})^T\widetilde L_0^{-1}-DT(x)\|
 \le C\delta,
 \qquad
 hR/\theta<\widetilde R_0\le R_0.
\]
Their normalized supports contain $B_3$, and their density values agree at zero. Apply the interior construction of Section~\ref{geom:sec:iteration}. The recurrence and summability estimates \eqref{geom:eq:recurrence}--\eqref{geom:eq:summability}, together with the preservation of $B_3$ in \eqref{m:display:047}, construct every level up to the largest $J$ with $\widetilde R_J\ge\Lambda_{\rm int}\ell$, provided $\delta$ is sufficiently small. More precisely, the bounds $C\delta^2<\delta_\theta$ and $C\delta<1/4$ ensure, respectively, the applicability of \cref{geom:prop:improvement} and the matrix and ball inclusions used in that induction.

Retain the interior notation $\widetilde R_j=\theta^j\widetilde R_0$, $\widetilde L_j$, $\widetilde q_j$, and $\widetilde E_j$ from that construction. Since $\Lambda_{\rm int}^{-1}\le c\delta$,
\[
 \sup_{j\le J}\widetilde E_j^{1/2}\le C\delta,
 \qquad
 \sum_{j=0}^J
 \bigl(\widetilde E_j^{1/2}+K\widetilde R_j\bigr)
 \le C\delta.
\]
Set $\widetilde B_j=(\widetilde L_j^{-1})^T \widetilde L_j^{-1}$. Applying the update identity \eqref{m:display:056} to the interior matrices gives
\[
 \|\widetilde B_j-DT(x)\|
 \le \|\widetilde B_0-DT(x)\|
      +C\sum_{i<j}
       \bigl(\sqrt{\widetilde E_i}
             +K\widetilde R_i\bigr)
 \le C\delta.
\]

Choose $j=\max\{i\ge0:\widetilde R_i\ge hR\}$. Then $hR\le\widetilde R_j<hR/\theta$. Since $hR>\Lambda_{\rm int}\ell$, we have $j\le J$. By $c_0h>8$,
\[
 B(x,8R)\subset B(x,c_0\widetilde R_j)
 \subset x+\widetilde R_j\widetilde L_jB_1.
\]
Apply \eqref{bmo:eq:meanflat} to this interior level, with source center $x$, radius $\widetilde R_j$, and relative ball $X_{x,8R}$. The bounds just proved and $\widetilde R_j/R\in[h,h/\theta)$ show that the infimum in \eqref{bmo:eq:coarse} is at most $C\delta\widetilde R_j\le C\delta R$.

Choose $\delta$ so that $C\delta\le\eta$ and take the supremum over $x\in X$. All constants are uniform in $x$ and $\eps$. The same choices can be made for the interchanged marginals, proving the estimate on $Y$ as well.
\end{proof}

\subsection{A fixed elliptic operator for the potential error}
\label{bmo:sec:duality}

Recall $A(x)=\rho_0(x)(DT(x))^{-1}$ and the notation in \eqref{bmo:eq:wdef}. By \cref{lem:classicalfinite}, $A\in C^{1,\beta}(X)$ is symmetric and satisfies $c\Id\preceq A\preceq C\Id$, independently of $\eps$. For a smooth test function $f$ on $X$, set $\bar f=\int_Xf\dd\mu$ and let $\psi_f$ solve
\begin{equation}\label{bmo:eq:neumann}
\begin{gathered}
 -\operatorname{div}(A\nabla\psi_f)
 =\rho_0(f-\bar f)
 \quad\text{in }\intr X,\\
 A\nabla\psi_f\cdot n_X=0
 \quad\text{on }\partial X,
 \qquad
 \int_X\psi_f\dd\mu=0,
\end{gathered}
\end{equation}
where $n_X$ is the outward unit normal. The compatibility condition $\int_X\rho_0(f-\bar f)\dd x=0$ holds due to $\mu(X)=1$. Neumann solvability and the Schauder estimates give a unique $\psi_f\in C^{2,\beta}(X)$ with this normalization; see \cite[Theorem~B.5]{GMS} and \cite[Section~6.7]{GT}. For $d=1$, these assertions and the $W^{2,2}$ estimate used below follow by integrating $A(x)\psi_f'(x)=-\int_{x_-}^x\rho_0(t)(f(t)-\bar f)\dd t$, where $X=[x_-,x_+]$. The identities $T_\#\mu=\nu$ and $(\mathrm{pr}_y)_\#\pi_\eps=\nu$ imply
\begin{equation}\label{m:display:064}
 \int_{X\times Y}
 \bigl((\psi_f\circ S)(y)-\psi_f(x)\bigr)
 \dd\pi_\eps(x,y)=0.
\end{equation}
For $t\in[0,1]$, set $Y_t(x,y)=(1-t)T(x)+ty\in Y$. For every Borel set $B\subset Y$, define
\begin{equation}\label{bmo:eq:tensordef}
 \cR_\eps(B)
 :=
 \int_{X\times Y}\int_0^1
 (1-t)\mathbf1_B(Y_t(x,y))
 (y-T(x))(y-T(x))^T
 \dd t\dd\pi_\eps(x,y).
\end{equation}
This defines a finite positive semidefinite matrix-valued measure, with entries denoted by $\cR_{\eps,ij}$.

Since $S=T^{-1}$ and $DT$ is symmetric, $\nabla(\psi_f\circ S)(T(x))=(DT(x))^{-1}\nabla\psi_f(x)$. By \eqref{bmo:eq:wdef}, the conditional mean of $y-T(x)$ is $b(x)=\nabla W(x)$. Disintegration therefore gives
\begin{equation}\label{m:display:065}
\begin{aligned}
 &\int_{X\times Y}
 \nabla(\psi_f\circ S)(T(x))
 \cdot(y-T(x))\dd\pi_\eps(x,y)\\
 &\quad=
 \int_X
 \bigl((DT)^{-1}\nabla\psi_f\bigr)
 \cdot\nabla W\dd\mu\\
 &\quad=
 \int_X A\nabla\psi_f\cdot\nabla W\dd x
 =\int_XW(f-\bar f)\dd\mu
 =\int_XWf\dd\mu.
\end{aligned}
\end{equation}
The penultimate equality is the weak formulation of \eqref{bmo:eq:neumann}, tested with $W\in H^1(X)$. The last equality follows from $\int_XW\dd\mu=0$. Taylor's formula along $[T(x),y]\subset Y$, together with \eqref{m:display:064} and \eqref{m:display:065}, therefore yields
\begin{equation}\label{bmo:eq:dualityexact}
 \int_XWf\dd\mu
 =-\sum_{i,j=1}^d
 \int_Y\partial_{ij}(\psi_f\circ S)(y)
       \cR_{\eps,ij}(\dd y).
\end{equation}

\begin{lemma}[Uniform bounds for interpolation densities]
\label{bmo:lem:density}
For every $t\in[0,1]$, $(Y_t)_\#\pi_\eps=g_t\dd y$ with $0\le g_t\le C$. Moreover, $\cR_\eps$ has a positive semidefinite Lebesgue density, denoted by the same symbol, and
\begin{equation}\label{bmo:eq:tensorbound}
 \|\cR_\eps\|_{L^\infty(Y)}
 \le C\bigl(\ell^2+\|b\|_{L^\infty(X)}^2\bigr).
\end{equation}
The constants are uniform in $t$ and sufficiently small $\eps$.
\end{lemma}

\begin{proof}
Set $F_t=(1-t)T+tm_\eps$. By \cref{lem:classicalfinite,bmo:lem:conditional}, $T$ and $m_\eps$ are uniformly strongly monotone. Their convex combination therefore satisfies
\begin{equation}\label{m:display:066}
\begin{gathered}
 (F_t(x)-F_t(x'))\cdot(x-x')
 \ge c|x-x'|^2,\\
 |F_t(x)-F_t(x')|\ge c|x-x'|
 \qquad(x,x'\in X).
\end{gathered}
\end{equation}
Since $m_\eps(x)\in S_{\eps,x}$, \eqref{geom:eq:diameters} implies $|y-m_\eps(x)|\le C\ell$ on $\spt\pi_\eps$. Consequently, $|Y_t(x,y)-F_t(x)|\le Ct\ell$ there.

Fix $t>0$. For each $x\in X$, the change of variables
$\eta=(1-t)T(x)+ty$ has inverse
$y=t^{-1}(\eta-(1-t)T(x))$, so $\dd y=t^{-d}\dd\eta$. Since $\dd\pi_\eps(x,y) =h_\eps(x,y)\rho_0(x)\rho_1(y)\dd x\dd y$, the measure $(Y_t)_\#\pi_\eps$ has density
\begin{equation}\label{m:display:067}
 g_t(\eta)
 =
 t^{-d}\int_X
 h_\eps\!\left(
 x,\frac{\eta-(1-t)T(x)}t
 \right)\rho_0(x)
 \rho_1\!\left(
 \frac{\eta-(1-t)T(x)}t
 \right)\dd x,
\end{equation}
where the integrand is zero when its target argument lies outside $Y$. Fix $\eta\in Y$. If the integrand in \eqref{m:display:067} is nonzero, then $y=t^{-1}(\eta-(1-t)T(x))\in S_{\eps,x}$ and the preceding bound gives $|\eta-F_t(x)|\le Ct\ell$. Thus only source points in $E_{t,\eta}:=\{x\in X:|F_t(x)-\eta|\le Ct\ell\}$ contribute to \eqref{m:display:067}. For $x,x'\in E_{t,\eta}$, \eqref{m:display:066} and the triangle inequality give
\[
 c|x-x'|
 \le |F_t(x)-F_t(x')|
 \le |F_t(x)-\eta|+|F_t(x')-\eta|
 \le Ct\ell.
\]
If $E_{t,\eta}$ is nonempty, choosing any $x'\in E_{t,\eta}$ therefore gives $E_{t,\eta}\subset B(x',Ct\ell)$. Consequently, $|E_{t,\eta}|\le C(t\ell)^d$; the same bound holds if $E_{t,\eta}$ is empty.

By \eqref{eq:scales} and $\eps=\ell^{d+2}$, $h_\eps=(q_\eps)_+/\eps\le C\ell^{-d}$. Using also $\rho_0,\rho_1\le M$ in \eqref{m:display:067}, we conclude
\[
 0\le g_t(\eta)
 \le Ct^{-d}\ell^{-d}|E_{t,\eta}|
 \le Ct^{-d}\ell^{-d}(t\ell)^d
 \le C.
\]
The constant is independent of $t\in(0,1]$, $\eta$, and sufficiently small $\eps$. For $t=0$, $(Y_0)_\#\pi_\eps=T_\#\mu=\nu$, so $g_0=\rho_1\le M$. On $\spt\pi_\eps$,
\begin{equation}\label{m:display:068}
 |y-T(x)|^2
 \le 2|m_\eps(x)-T(x)|^2
      +2|y-m_\eps(x)|^2
 \le 2\|b\|_\infty^2+C\ell^2.
\end{equation}
Hence, for every Borel set $B\subset Y$,
\[
\begin{aligned}
 \sum_{i=1}^d\cR_{\eps,ii}(B)
 &\le
 C(\ell^2+\|b\|_\infty^2)
 \int_0^1(1-t)\int_Bg_t(y)\dd y\dd t\le C(\ell^2+\|b\|_\infty^2)|B|.
\end{aligned}
\]
Positivity of $\cR_\eps$ now gives absolute continuity and \eqref{bmo:eq:tensorbound}.
\end{proof}

\begin{lemma}[Elliptic tensor-to-BMO estimate]
\label{bmo:lem:ellipticbmo}
For every $R\in L^\infty(Y;\R^{d\times d})$, there is a unique $\mathcal U R\in L^2(\mu)$ satisfying
\begin{equation}\label{m:display:069}
 \int_X(\mathcal U R)f\dd\mu
 =-\sum_{i,j=1}^d
 \int_Y\partial_{ij}(\psi_f\circ S)(y)
       R_{ij}(y)\dd y
\end{equation}
for every smooth $f$, where $\psi_f$ solves \eqref{bmo:eq:neumann}. Moreover, $\int_X\mathcal U R\dd\mu=0$ and
\begin{equation}\label{bmo:eq:ellipticbmo}
 [\mathcal U R]_{\BMO(X)}
 \le C\|R\|_{L^\infty(Y)},
\end{equation}
where the constant depends only on the data.
\end{lemma}

\begin{proof}
For $d\ge2$, apply \cite[Lemma~A.1, equation~(78)]{IF} with $p=2$ and zero conormal data; for $d=1$, use the integrated formula above. The normalization $\int_X\psi_f\dd\mu=0$, the density bounds, and the chain rule give
\[
 \|\psi_f\|_{W^{2,2}(X)}
 +\|D^2(\psi_f\circ S)\|_{L^2(Y)}
 \le C\|f\|_{L^2(\mu)},
\]
where we used \cref{lem:classicalfinite} for the composition with $S$. By the Riesz representation theorem, \eqref{m:display:069} defines a unique $\mathcal U R\in L^2(\mu)$ with $\|\mathcal U R\|_{L^2(\mu)} \le C\|R\|_{L^2(Y)}$. Taking $f=1$ gives $\psi_f=0$ and $\int_X\mathcal U R\dd\mu=0$.

Let $G(z,x)$ be the symmetric Neumann Green function, with the sign convention and normalization
\begin{equation}\label{m:display:070}
\begin{gathered}
 -\operatorname{div}_z(A(z)\nabla_zG(z,x))
 =\delta_x-|X|^{-1},
 \qquad \int_XG(z,x)\dd z=0,\\
 A(z)\nabla_zG(z,x)\cdot n_X(z)=0
 \quad\text{on }\partial X.
\end{gathered}
\end{equation}
For $d\ge2$, existence and symmetry follow from \cite[Propositions~B.6 and~B.8]{GMS}. For $d=1$, write $X=[x_-,x_+]$ and define $G$ by
\[
 A(z)\partial_zG(z,x)=\frac{z-x_-}{x_+-x_-}-\1_{\{z>x\}},\qquad \int_{x_-}^{x_+}G(z,x)\dd z=0,\qquad x_-<x<x_+.
\]
This kernel satisfies \eqref{m:display:070}; integration by parts gives symmetry. Off the diagonal, $\partial_zG,\partial_z^2G$ are bounded and $\partial_x\partial_z^jG=0$ for $j=1,2$, so the Green derivative bounds used below hold also in this case.

Set $g_\mu(z)=\int_XG(z,x)\dd\mu(x)$. Integrating \eqref{m:display:070} against $\mu$ gives
\[
\begin{gathered}
 -\operatorname{div}(A\nabla g_\mu)
 =\rho_0-|X|^{-1}
 \quad\text{in }\intr X,\\
 A\nabla g_\mu\cdot n_X=0
 \quad\text{on }\partial X,
 \qquad \int_Xg_\mu\dd x=0.
\end{gathered}
\]
The identities $\rho_0=\rho_1(T)\det DT$ and $\operatorname{div}\bigl(\det DT\,(DT)^{-1}\bigr)=0$ give $\operatorname{div}A=\rho_0(\nabla\log\rho_1)\circ T$. Thus its equation is equivalently
\[
 \sum_{i,j=1}^d
 ((DT)^{-1})_{ij}\partial_{ij}g_\mu
 +(\nabla\log\rho_1)(T)\cdot\nabla g_\mu
 =\frac{1}{|X|\rho_0}-1.
\]
By boundary correspondence, $(DT)^{-1}n_X$ is a positive multiple of $n_Y\circ T$, where $n_Y$ is the outward unit normal to $Y$. The boundary condition therefore becomes $(n_Y\circ T)\cdot\nabla g_\mu=0$. The compatibility condition is $\int_X(|X|^{-1}-\rho_0)\dd x=0$. For $d\ge2$, \cite[Theorem~5.6]{GS}, with exponent $\beta$, and uniqueness modulo constants give $g_\mu\in C^{2,\beta}(X)$; for $d=1$, this follows by direct integration.

The Green representation yields
\begin{equation}\label{m:display:071}
 \psi_f(z)=\int_XG(z,x)f(x)\dd\mu(x)-\bar f\,g_\mu(z)+c_f,
\end{equation}
where $c_f$ enforces $\int_X\psi_f\dd\mu=0$. Take $f$ supported away from $S(\supp R)$. Then
$x\in\supp f$ and $S(y)$ with $y\in\supp R$ remain a positive
distance apart, so the off-diagonal Green derivative bounds
justify Fubini's theorem and differentiation under the integral.
Testing \eqref{m:display:069} therefore gives, for
$x\notin S(\supp R)$,
\begin{equation}\label{m:display:072}
 (\mathcal U R)(x)
 =-\sum_{i,j=1}^d\int_Y
 \partial_{y_i y_j}
 \bigl[G(S(y),x)-g_\mu(S(y))\bigr]
 R_{ij}(y)\dd y,
\end{equation}
where $\supp R$ denotes the essential support. This formula defines a continuous representative off $S(\supp R)$, and the contribution of $g_\mu$ is constant in $x$.

With $y=T(z)$, the nonconstant term becomes $\sum_{i,j=1}^d\int_XK_{ij}(x,z)R_{ij}(T(z))\dd z$, with kernel
\[
\begin{aligned}
 K(x,z)=-\det DT(z)\Bigl(
 &(DS(T(z)))^T D_z^2G(z,x)DS(T(z))+\sum_{i=1}^d
 \partial_{z_i}G(z,x)D^2S_i(T(z))
 \Bigr).
\end{aligned}
\]
We claim that
\begin{equation}\label{m:display:073}
 |K(x,z)|\le C|x-z|^{-d},
 \qquad
 |\nabla_xK(x,z)|\le C|x-z|^{-d-1}.
\end{equation}
The first bound follows from \cref{lem:classicalfinite} and $|D_z^jG(z,x)|\le C|x-z|^{2-d-j}$, $j=1,2$; see \cite[Appendix~C, equation~(67)]{GMS}.

For the second, fix distinct interior points $x,z$ and put $r=|x-z|/8$. By symmetry and \eqref{m:display:070}, sufficiently small difference quotients of order $j=1,2$ in $z$, viewed as functions of the second variable, solve the homogeneous equation in $X\cap B(x,2r)$ with homogeneous conormal data on $\partial X\cap B(x,2r)$. Indeed, the shifted poles remain outside this neighborhood and the constant right-hand side cancels. The preceding derivative bounds control their suprema by $Cr^{2-d-j}$. For $r$ below a fixed boundary-chart radius, dilation by $r$
leaves the coefficient and boundary-chart norms uniformly bounded.
The local gradient estimate therefore contributes a factor
$r^{-1}$, with a constant independent of $r$ and of the
difference-quotient increment. The local interior and boundary gradient estimates \cite[Sections~6.1 and~6.7]{GT}, followed by passage to the difference-quotient limit, yield
\begin{equation}\label{eq:greenmixed}
 |\nabla_xD_z^jG(z,x)|
 \le Cr^{1-d-j},
 \qquad j=1,2.
\end{equation}
The estimates extend to boundary points by limits; pairs separated by a fixed positive distance are controlled by local regularity. All coefficients in the formula for $K$ depend only on $z$, so \eqref{eq:greenmixed} proves \eqref{m:display:073}. We conclude by the standard near--far decomposition, as in \cite[Section~C.2.1, proof of Theorem~B.2]{GMS}. For a sufficiently small relative ball $X_{x_0,r}$, set $R_1(y)=R(y)\mathbf1_{B(x_0,4r)}(S(y))$ and $R_2=R-R_1$. The $L^2$ estimate, the Jacobian bounds, and $|X_{x_0,r}|\ge cr^d$ give
\[
 \frac1{|X_{x_0,r}|}\int_{X_{x_0,r}}|\mathcal U R_1|\dd x\le Cr^{-d/2}\|R_1\|_{L^2(Y)}\le C\|R\|_\infty.
\]
For $x\in X_{x_0,r}$, convexity and \eqref{m:display:073} give $|K(x,z)-K(x_0,z)|\le Cr|z-x_0|^{-d-1}$ when $|z-x_0|\ge4r$. The constant contribution of $g_\mu$ cancels, so
\[
 |(\mathcal U R_2)(x)-(\mathcal U R_2)(x_0)|\le Cr\|R\|_\infty\int_{X\setminus B(x_0,4r)}|z-x_0|^{-d-1}\dd z\le C\|R\|_\infty.
\]
These estimates bound the mean oscillation on $X_{x_0,r}$. For radii bounded below, the global $L^2$ estimate suffices. Taking the supremum proves \eqref{bmo:eq:ellipticbmo}.
\end{proof}

Equations \eqref{bmo:eq:dualityexact} and \eqref{m:display:069} identify $W=\mathcal U\cR_\eps$ in $L^2(\mu)$. By \cref{bmo:lem:density,bmo:lem:ellipticbmo} and $b=\nabla W$,
\begin{equation}\label{bmo:eq:nonlinear}
 [W]_{\BMO(X)}
 \le C\bigl(\ell^2+\|\nabla W\|_\infty^2\bigr),
\end{equation}
which will be used in the next subsection.

\subsection{BMO absorption and sharp graph localization}
\label{bmo:sec:bmogoal}
\label{bmo:sec:absorption}

Recall $\ell=\eps^{1/(d+2)}$, $r_\eps=u_\eps-u_0$, $s_\eps=v_\eps-v_0$, and $W=r_\eps+a_\eps-\bar w$ from \eqref{bmo:eq:wdef}. The goal of this subsection is the following theorem, based on \eqref{bmo:eq:nonlinear} and the coarse approximation in \cref{bmo:prop:coarsefromiteration}.

\begin{theorem}[Sharp BMO, gradient and support estimates]\label{bmo:thm:main}
Under \cref{ass:geometrydata}, for all sufficiently small $\eps$,
\begin{equation}\label{m:display:074}
\begin{aligned}
 [r_\eps]_{\BMO(X)}&\asymp\ell^2,
 & [s_\eps]_{\BMO(Y)}&\asymp\ell^2,\\
 \|\nabla u_\eps-T\|_{L^\infty(X)}&\asymp\ell,
 & \|\nabla v_\eps-S\|_{L^\infty(Y)}&\asymp\ell.
\end{aligned}
\end{equation}
The section estimates \eqref{eq:sharpfiber} hold for every $x\in X$ and $y\in Y$, and \eqref{eq:hausdorff} holds.
\end{theorem}

The following is a bounded-domain version of the BMO Gagliardo--Nirenberg inequality; see \cite[Theorem~1.3]{Dao2023} for a whole-space version. We include a proof valid up to the boundary of $D$.
\begin{lemma}[BMO interpolation]
\label{bmo:lem:bmo-interpolation}
Let $D\subset\R^d$ be a convex body and let $f\in C^{1,1}(D)$. Then
\begin{equation}\label{bmo:eq:bmo-interpolation}
 \|\nabla f\|_\infty^2
 \le C[f]_{\BMO(D)}
 \bigl(\|D^2f\|_\infty+[f]_{\BMO(D)}\bigr),
\end{equation}
where all norms are taken on $D$ and $C$ depends only on $D$.
\end{lemma}

\begin{proof}
Fix $x\in D$ and $0<r\le\diam D$. Apply the map $z\mapsto x+\frac{r}{2\diam D}(z-x)$ to a fixed interior ball of $D$. By convexity, its image is contained in $D$, and every image point is at distance at most $r/2$ from $x$. Thus $D_{x,r}$ contains a ball $B(z_0,cr)$, where $c>0$ depends only on $D$.

We first establish two estimates for affine polynomials, the second of which will also be used in \cref{bmo:lem:smoothing}:
\begin{equation}\label{bmo:eq:affine-estimates}
\begin{aligned}
 r|\nabla P|
 &\le C\inf_{a\in\R}
 \frac1{|D_{x,r}|}\int_{D_{x,r}}|P-a|\dd z,\\
 |P(x)|+r|\nabla P|
 &\le Cr^{-d/2}\|P\|_{L^2(D_{x,r})}.
\end{aligned}
\end{equation}
For the first, symmetry of the contained ball and $|u+v|+|u-v|\ge2|v|$ give, for every $a\in\R$,
\[
\begin{aligned}
 \int_{D_{x,r}}|P-a|\dd z
 &\ge
 \frac12\int_{B(0,cr)}
 \bigl(
 |P(z_0)-a+\nabla P\cdot z|
 +|P(z_0)-a-\nabla P\cdot z|
 \bigr)\dd z\\
 &\ge \int_{B(0,cr)}|\nabla P\cdot z|\dd z
 \ge cr^{d+1}|\nabla P|.
\end{aligned}
\]
Divide by $|D_{x,r}|\le Cr^d$ and take the infimum over $a$. For the second, symmetry gives
\[
\begin{aligned}
 \int_{D_{x,r}}|P|^2\dd z
 &\ge
 |B(0,cr)|\,|P(z_0)|^2
 +\int_{B(0,cr)}|\nabla P\cdot z|^2\dd z\\
 &\ge cr^d|P(z_0)|^2
      +cr^{d+2}|\nabla P|^2.
\end{aligned}
\]
Since $|x-z_0|\le r$, $|P(x)|\le|P(z_0)|+r|\nabla P|$, which proves \eqref{bmo:eq:affine-estimates}.

Now take $P(z)=f(x)+\nabla f(x)\cdot(z-x)$. Taylor's formula along $[x,z]\subset D$ gives
\[
 |f(z)-P(z)|
 \le \tfrac12\|D^2f\|_\infty|z-x|^2
 \le \tfrac12\|D^2f\|_\infty r^2
 \qquad(z\in D_{x,r}).
\]
Apply the first estimate in \eqref{bmo:eq:affine-estimates} with $a=\langle f\rangle_{D_{x,r}}$. Using the definition of the BMO seminorm,
\begin{equation}\label{m:display:075}
\begin{aligned}
 r|\nabla f(x)|
 &\le
 \frac{C}{|D_{x,r}|}
 \int_{D_{x,r}}
 |P-\langle f\rangle_{D_{x,r}}|\dd z\le C\bigl(
 [f]_{\BMO(D)}+\|D^2f\|_\infty r^2
 \bigr).
\end{aligned}
\end{equation}

Assume first that $[f]_{\BMO(D)}>0$ and $\|D^2f\|_\infty>0$, and choose
\[
 r=\min\left\{
 \frac{\diam D}{2},
 \sqrt{\frac{[f]_{\BMO(D)}}{\|D^2f\|_\infty}}
 \right\}.
\]
If the second term is the minimum, \eqref{m:display:075} gives $|\nabla f(x)|^2 \le C[f]_{\BMO(D)}\|D^2f\|_\infty$. Otherwise, $\|D^2f\|_\infty \le4[f]_{\BMO(D)}/(\diam D)^2$, and \eqref{m:display:075} gives $|\nabla f(x)|^2\le C[f]_{\BMO(D)}^2$. Taking the supremum over $x\in D$ proves \eqref{bmo:eq:bmo-interpolation}.

If $\|D^2f\|_\infty=0$, take $r=\diam D/2$ in \eqref{m:display:075}. If $[f]_{\BMO(D)}=0$, divide \eqref{m:display:075} by $r$ and let $r\downarrow0$ to obtain $\nabla f=0$. The argument applies to every $x\in D$, including boundary points.
\end{proof}

\begin{lemma}[Smooth approximation with small Hessian]
\label{bmo:lem:smoothing}
Let $W=W_\eps\in C^{1,1}(X)$ satisfy $\|D^2W\|_\infty\le K$ uniformly in $\eps$. Suppose that, for every $\eta>0$, \eqref{bmo:eq:coarse} holds with $m_\eps-T$ replaced by $\nabla W$. For every $\eta_0>0$ and all sufficiently small $\eps$, there exist $V$, smooth on a neighborhood of $X$, and $E\in C^{1,1}(X)$ such that
\begin{equation}\label{bmo:eq:decomposition}
 W=V+E,\qquad
 \|D^2V\|_\infty\le\eta_0,\qquad
 \|E\|_\infty\le C_{\eta_0}\ell^2,\qquad
 \|\nabla E\|_\infty\le C_{\eta_0}\ell,
\end{equation}
where all norms are taken on $X$. The constants and the threshold for $\eps$ depend only on $\eta_0$, $K$, the domain, and the constants and thresholds in \eqref{bmo:eq:coarse}.
\end{lemma}

\begin{proof}
Fix $0<\eta\le1$, to be chosen below, and set $R=L_\eta\ell$ as in \eqref{bmo:eq:coarse}. Take $\eps<\eps_\eta$ sufficiently small that \eqref{bmo:eq:affine-estimates} applies at radii up to $8R$.

Choose a maximal $R$-separated family $\{x_j\}\subset X$. Then the balls $B(x_j,R)$ cover $X$, and the balls $B(x_j,8R)$ have uniformly bounded overlap. Let $\{\chi_j\}$ be a nonnegative smooth partition of unity on a neighborhood of $X$, with $\supp\chi_j\subset B(x_j,2R)$ and $\|D^k\chi_j\|_\infty\le CR^{-k}$ for $k=0,1,2$. Set $D_j=X_{x_j,4R}$.

Choose $P_j$ affine with $\nabla P_j =|X_{x_j,8R}|^{-1} \int_{X_{x_j,8R}}\nabla W(z)\dd z$ and $\int_{D_j}(W-P_j)\dd z=0$. As the average minimizes the squared error, \eqref{bmo:eq:coarse} and Poincar\'e's inequality on the convex set $D_j$ give
\begin{equation}\label{m:display:076}
\begin{aligned}
 \|\nabla W-\nabla P_j\|_{L^2(D_j)}
 &\le C\eta R^{d/2+1},\\
 \|W-P_j\|_{L^2(D_j)}
 &\le CR\|\nabla W-\nabla P_j\|_{L^2(D_j)}
 \le C\eta R^{d/2+2}.
\end{aligned}
\end{equation}

If $x\in X\cap\supp\chi_j\cap\supp\chi_k$, then $X_{x,R}\subset D_j\cap D_k$. Thus \eqref{m:display:076} implies $\|P_j-P_k\|_{L^2(X_{x,R})} \le C\eta R^{d/2+2}$. By \eqref{bmo:eq:affine-estimates},
\begin{equation}\label{m:display:077}
 |P_j(x)-P_k(x)|
 +R|\nabla P_j-\nabla P_k|
 \le C\eta R^2.
\end{equation}

Define $V=\sum_j\chi_jP_j$. For $x\in X$, choose $k$ with $\chi_k(x)>0$. Since $\sum_jD\chi_j=\sum_jD^2\chi_j=0$, \eqref{m:display:077} and bounded overlap yield
\[
\begin{aligned}
 |D^2V(x)|
 &\le C\sum_{j:\,x\in\supp\chi_j}
 \left(
 R^{-2}|P_j(x)-P_k(x)|
 +R^{-1}|\nabla P_j-\nabla P_k|
 \right)
 \le C\eta.
\end{aligned}
\]

To estimate $E=W-V$, fix $x\in X\cap\supp\chi_j$. Apply \eqref{bmo:eq:affine-estimates} on $X_{x,R}\subset D_j$ to $z\mapsto P_j(z)-W(x)-\nabla W(x)\cdot(z-x)$. Taylor's formula and \eqref{m:display:076} give
\[
 |W(x)-P_j(x)|
 +R|\nabla W(x)-\nabla P_j|
 \le CR^{-d/2}\|P_j-W\|_{L^2(X_{x,R})}
       +CKR^2\le C(K+\eta)R^2.
\]
On $X\cap\supp\chi_j$, the preceding estimates give $|W-P_j|\le C(K+1)R^2$ and $|\nabla W-\nabla P_j|\le C(K+1)R$. Since $\chi_j\ge0$ and $\sum_j\chi_j=1$, $E=\sum_j\chi_j(W-P_j)$ satisfies $\|E\|_\infty\le C(K+1)R^2$. Moreover,
\[
 \nabla E
 =\sum_j\nabla\chi_j(W-P_j)
  +\sum_j\chi_j(\nabla W-\nabla P_j).
\]
The supports have bounded overlap and $|\nabla\chi_j|\le CR^{-1}$. Thus $\|\nabla E\|_\infty \le CR^{-1}(K+1)R^2+C(K+1)R \le C(K+1)R$.

Choose $\eta$ so that $C\eta\le\eta_0$. Then $L_\eta$ is fixed, and substituting $R=L_\eta\ell$ proves \eqref{bmo:eq:decomposition}.
\end{proof}

\begin{proof}[Proof of \cref{bmo:thm:main}]
By \cref{bmo:lem:conditional,lem:classicalfinite}, $\|D^2W\|_\infty\le C$ uniformly in $\eps$. Moreover, \eqref{bmo:eq:rough} gives $[W]_{\BMO(X)}\to0$, and \cref{bmo:prop:coarsefromiteration} gives \eqref{bmo:eq:coarse}.

Fix $0<\eta_0\le1$, to be chosen below, and apply \cref{bmo:lem:smoothing}. Since $W=V+E$, we have
\begin{equation}\label{m:display:078}
 [V]_{\BMO(X)}
 \le [W]_{\BMO(X)}+2\|E\|_\infty
 \le [W]_{\BMO(X)}+C_{\eta_0}\ell^2.
\end{equation}
By \cref{bmo:lem:bmo-interpolation}, for $\ell\le1$ and $\eps$ small enough that $[W]_{\BMO(X)}\le1$,
\begin{equation}\label{m:display:079}
 \|\nabla V\|_\infty^2
 \le C[V]_{\BMO(X)}
          \bigl(\eta_0+[V]_{\BMO(X)}\bigr)
 \le C\eta_0[W]_{\BMO(X)}
       +C[W]_{\BMO(X)}^2+C_{\eta_0}\ell^2,
\end{equation}
where $C$ is independent of $\eta_0$. Using $\|\nabla W\|_\infty^2 \le2\|\nabla V\|_\infty^2+2\|\nabla E\|_\infty^2$ in \eqref{bmo:eq:nonlinear}, together with \eqref{bmo:eq:decomposition}, gives
\begin{equation}\label{bmo:eq:absorb}
 [W]_{\BMO(X)}
 \le C\eta_0[W]_{\BMO(X)}
      +C[W]_{\BMO(X)}^2+C_{\eta_0}\ell^2.
\end{equation}
Choose $\eta_0$ so that $C\eta_0\le1/4$. Then take $\eps$ sufficiently small that \cref{bmo:lem:smoothing} applies and $C[W]_{\BMO(X)}\le1/4$. Thus \eqref{bmo:eq:absorb} yields $\frac12[W]_{\BMO(X)}\le C_{\eta_0}\ell^2$. With $\eta_0$ fixed, this gives $[W]_{\BMO(X)}\le C\ell^2$.

Since $r_\eps=W+\bar w-a_\eps$, invariance under addition of constants and \cref{bmo:lem:conditional} imply
\begin{equation}\label{m:display:080}
 [r_\eps]_{\BMO(X)}
 \le [W]_{\BMO(X)}+2\|a_\eps\|_\infty
 \le C\ell^2.
\end{equation}
Interchanging the marginals gives $[s_\eps]_{\BMO(Y)}\le C\ell^2$.

By \eqref{eq:hessians} and \cref{lem:classicalfinite}, $\|D^2r_\eps\|_\infty+\|D^2s_\eps\|_\infty\le C$. Hence \cref{bmo:lem:bmo-interpolation} gives
\begin{equation}\label{m:display:081}
 \|\nabla r_\eps\|_{L^\infty(X)}
 +\|\nabla s_\eps\|_{L^\infty(Y)}
 \le C\ell.
\end{equation}
For $x\in\partial X$, the inner inclusion in \eqref{eq:balls} gives $\dist(\nabla u_\eps(x),\partial Y)\ge c\ell$. Since $T(x)\in\partial Y$, it follows that $|\nabla r_\eps(x)|\ge c\ell$. The same argument gives $|\nabla s_\eps(y)|\ge c\ell$ for $y\in\partial Y$.

Applying \cref{bmo:lem:bmo-interpolation} once more, and using \eqref{m:display:080}, we obtain $c\ell^2 \le C[r_\eps]_{\BMO(X)} (1+[r_\eps]_{\BMO(X)}) \le C[r_\eps]_{\BMO(X)}$ for sufficiently small $\eps$. The corresponding estimate for $s_\eps$ follows by interchanging the marginals. This proves \eqref{m:display:074}.

For $x\in X$ and $y\in S_{\eps,x}$, \eqref{eq:balls} and \eqref{m:display:081} yield
\begin{equation}\label{m:display:083}
 |y-T(x)|
 \le |y-z_{\eps,x}|+|\nabla r_\eps(x)|
 \le C\ell.
\end{equation}
The inner inclusion in \eqref{eq:balls} also gives $\diam S_{\eps,x}\ge2c\ell$. Consequently, $\sup_{y\in S_{\eps,x}}|y-T(x)| \ge\frac12\diam S_{\eps,x}\ge c\ell$. This proves the source estimate in \eqref{eq:sharpfiber}. Interchanging the marginals proves the target estimate.

By \eqref{eqv:eq:support} and \eqref{m:display:083}, every $(x,y)\in\spt\pi_\eps$ is within $C\ell$ of $(x,T(x))$. Conversely, $(x,z_{\eps,x})\in\spt\pi_\eps$ and $|z_{\eps,x}-T(x)|\le C\ell$. These two estimates give the Hausdorff upper bound.

For any $z\in X$, $|y-T(x)|\le|y-T(z)|+\Lip(T)|x-z|$. Cauchy--Schwarz therefore gives
\begin{equation}\label{m:display:084}
\begin{aligned}
 \dist((x,y),\operatorname{graph}T)
 &=\inf_{z\in X}
   \bigl(|x-z|^2+|y-T(z)|^2\bigr)^{1/2}\ge \frac{|y-T(x)|}{\sqrt{1+\Lip(T)^2}}.
\end{aligned}
\end{equation}
Fix $x\in X$ and choose $y\in S_{\eps,x}$ with $|y-T(x)|\ge c\ell$, as provided by \eqref{eq:sharpfiber}. Then $(x,y)\in\spt\pi_\eps$, and \eqref{m:display:084} proves the Hausdorff lower bound in \eqref{eq:hausdorff}.
\end{proof}

\section{Convergence of the potentials}
\label{end:sec:criterion}

We now derive the sharp potential estimates from support localization. The argument uses the classical potentials and the marginal equations, but does not use uniform Hessian bounds for the regularized potentials. We first formulate this implication under weaker hypotheses, then derive a reversible equation for the potential error and prove the nonlocal estimate needed to control it.

\subsection{A support-localization criterion for sharp potential convergence}
\label{end:sec:setting}

Let $X,Y\subset\R^d$ be convex bodies, and let
$\mu=\rho_0\1_X\dd x$ and $\nu=\rho_1\1_Y\dd y$
be probability measures, where the measurable densities satisfy
$0<m\le\rho_i\le M<\infty$ a.e.\ on their respective
supports. Let $u_\eps\in C^1(X)$ and $v_\eps\in C^1(Y)$ be the convex QOT potentials, satisfying \eqref{eq:slack} and \eqref{eqv:eq:marginal}. We recall $\ell=\eps^{1/(d+2)}$ and use the symmetric normalization \eqref{eq:gauge} unless stated otherwise.

\begin{assumption}[Classical map and support localization]
\label{end:ass:inputs}
Let $u_0\in C^{2,1}(X)$, and suppose that $T=\nabla u_0$ is a bijection from $X$ onto $Y$ with $T_\#\mu=\nu$. Let $v_0$ be the convex conjugate of $u_0$, restricted to $Y$. For some $0<\lambda\le\Lambda<\infty$ and $L_0<\infty$, assume
\begin{equation}\label{end:eq:classical}
\lambda\Id\preceq D^2u_0\preceq\Lambda\Id,\qquad \Lip(D^2u_0)\le L_0.
\end{equation}
Moreover, there are $R_0,\eps_0>0$ such that, for $0<\eps<\eps_0$,
\begin{equation}\label{end:eq:tube}
|y-T(x)|\le R_0\ell\qquad((x,y)\in X\times Y,\ q_\eps(x,y)\ge0).
\end{equation}
\end{assumption}

By \eqref{end:eq:classical}, $T$ and $S=T^{-1}=\nabla v_0$ are bi-Lipschitz. The support identity \eqref{eqv:eq:support} also holds under the present hypotheses. Indeed, the row marginal equation gives a point of positive slack for every $x\in X$, and concavity in $y$ makes every zero-slack point a limit of positive-slack points in the same row. Continuity of $q_\eps$ and positivity of the densities then give $\{q_\eps\ge0\}=\overline{\{q_\eps>0\}}=\spt\pi_\eps$, including at boundary points.

\begin{theorem}[Sharp potential convergence]
\label{end:thm:main}
Under the hypotheses stated at the beginning of this subsection, \cref{end:ass:inputs}, and normalization~\eqref{eq:gauge}, there are $c,C,\eps_1>0$ such that, for all $0<\eps<\eps_1$ and $1\le p\le\infty$,
\begin{align}\label{end:eq:main}
c\ell^2\le\|u_\eps-u_0\|_{L^p(\mu)}\le C\ell^2,\qquad
c\ell^2\le\|v_\eps-v_0\|_{L^p(\nu)}\le C\ell^2.
\end{align}
The constants depend only on the dimension, the domains, the density bounds, and the constants in \cref{end:ass:inputs}. Under the point normalization $u_\eps(x_*)=u_0(x_*)$ at any fixed $x_*\in X$, the upper bounds remain valid, after enlarging $C$ if necessary.
\end{theorem}

Under \cref{ass:geometrydata}, \cref{lem:classicalfinite,bmo:thm:main} verify \cref{end:ass:inputs}. Thus \cref{end:thm:main} yields the potential estimates in \cref{thm:sharp}. The proof of \cref{end:thm:main} is completed in Section~\ref{end:sec:completion}.

\begin{corollary}[Equivalence of sharp support and potential bounds]
\label{end:cor:characterization}
Assume  the hypotheses stated at the beginning of this subsection and the first part of \cref{end:ass:inputs}, up to \eqref{end:eq:classical}. Under either \eqref{eq:gauge} or the point normalization $u_\eps(x_*)=u_0(x_*)$ at a fixed $x_*\in X$, the following estimates are equivalent as $\eps\downarrow0$:
\begin{align}\label{eq:potential-tube-equivalence}
\|u_\eps-u_0\|_{L^\infty(X)}+\|v_\eps-v_0\|_{L^\infty(Y)}&=O(\ell^2),\nonumber\\
\sup_{(x,y)\in\spt\pi_\eps}|y-T(x)|&=O(\ell).
\end{align}
\end{corollary}
\begin{proof}
The support estimate implies the potential estimate by \cref{end:thm:main}. Conversely, \eqref{eq:slack} gives
\begin{equation}\label{eq:uniform-error-detachment}
u_0(x)+v_0(y)-x\cdot y=-q_\eps(x,y)-(u_\eps-u_0)(x)-(v_\eps-v_0)(y).
\end{equation}
Since $q_\eps\ge0$ on $\spt\pi_\eps$, the potential estimate bounds the left-hand side by $C\ell^2$ there. By conjugacy and \eqref{end:eq:classical}, $D^2v_0\succeq\Lambda^{-1}\Id$ and $\nabla v_0(T(x))=x$. Taylor's formula along $[T(x),y]\subset Y$ therefore yields
\[
\frac{|y-T(x)|^2}{2\Lambda}\le v_0(y)-v_0(T(x))-x\cdot(y-T(x))=u_0(x)+v_0(y)-x\cdot y\le C\ell^2,
\]
which proves the support estimate.
\end{proof}

\subsection{Support geometry and reversible kernels}
\label{end:sec:geometry}

We work under the hypotheses of \cref{end:thm:main}. Write $q=q_\eps$, $h=h_\eps$, $r=u_\eps-u_0$, and $s=v_\eps-v_0$, and abbreviate the sections by $S_x=S_{\eps,x}$ and $K_y=K_{\eps,y}$. All constants are independent of sufficiently small $\eps$; $C_R$ may also depend on a fixed $R>0$.

\subsubsection{Relative volume and barycenter identities}

For each $D=X,Y$, there are $c_D,r_D>0$ such that
\begin{equation}\label{end:eq:volume}
|D\cap B(x,t)|\ge c_Dt^d\qquad(x\in D,\ 0<t<r_D).
\end{equation}
To see this, fix $B(a,r_*)\subset D$. For $0<t<\diam D$, set $\tau=t/(2\diam D)$. By convexity, $(1-\tau)x+\tau B(a,r_*)\subset D\cap B(x,t)$, and this ball has volume $\omega_d(r_*t/(2\diam D))^d$. The density bounds give the corresponding estimates for $\mu$ and $\nu$.

The barycenter identities \eqref{eqv:eq:centers} hold under the present hypotheses: each section has positive mass, and concavity places its zero-slack set in its boundary, which has zero marginal measure. Differentiation under the integral is therefore justified at interior base points. The identities extend to the closed supports by continuity. By \eqref{end:eq:tube}, $S_x\subset\overline B(T(x),R_0\ell)$; likewise, $x\in K_y$ implies $|x-S(y)|\le\Lip(S)|T(x)-y|\le C\ell$. Averaging over the sections gives
\begin{equation}\label{end:eq:gradient}
|\nabla u_\eps(x)-T(x)|\le C\ell,\qquad |\nabla v_\eps(y)-S(y)|\le C\ell\qquad(x\in X,\ y\in Y).
\end{equation}
Consequently, for every fixed $R>0$,
\begin{equation}\label{end:eq:slopes}
|\nabla_xq(x,y)|+|\nabla_yq(x,y)|\le C_R\ell\qquad\text{if }|y-T(x)|\le R\ell.
\end{equation}
Indeed, the two terms are bounded by $|y-T(x)|+C\ell$ and $\Lip(S)|T(x)-y|+C\ell$, respectively.

\subsubsection{Height and overlap of nearby rows}

Fix $x\in X$ and choose $y_x^*\in Y$ with $q(x,y_x^*)=H_x:=\max_Yq(x,\cdot)$. The row marginal equation gives $H_x>0$. Since $\nu(S_x)\le C\ell^d$, it also gives $\eps\le H_x\nu(S_x)\le CH_x\ell^d$, and hence $H_x\ge c\ell^2$.

For $y\in Y\cap B(y_x^*,\ell)$, the segment $[y_x^*,y]$ lies in $Y\cap B(T(x),(R_0+1)\ell)$. Integrating \eqref{end:eq:slopes} along this segment yields $q(x,y)\ge H_x-C\ell^2$. Thus \eqref{end:eq:volume} and the row marginal equation imply
\begin{equation}\label{m:display:086}
\eps\ge\nu(Y\cap B(y_x^*,\ell))(H_x-C\ell^2)_+\ge c\ell^d(H_x-C\ell^2)_+.
\end{equation}
Using $\eps=\ell^{d+2}$, we obtain
\begin{equation}\label{end:eq:height}
c\ell^2\le H_x\le C\ell^2,\qquad 0\le h(x,y)\le C\ell^{-d}.
\end{equation}

The segment $[y_x^*,T(x)]$ lies in $Y$ and has length at most $R_0\ell$. Equations~\eqref{end:eq:slopes} and \eqref{end:eq:height} therefore give
\begin{equation}\label{end:eq:diagonal}
|q(x,T(x))|\le C\ell^2.
\end{equation}
Similarly, integrating \eqref{end:eq:slopes} along $[T(x),T(z)]\subset Y$ gives, for every fixed $R>0$,
\begin{equation}\label{end:eq:nearvalues}
|q(x,T(z))|\le C_R\ell^2\qquad(x,z\in X,\ |x-z|\le R\ell).
\end{equation}

Define
\begin{equation}\label{end:eq:k}
k(x,z)=h(x,T(z)),\qquad \sigma_\eps=(\id,S)_\#\pi_\eps=k(x,z)\mu(\dd x)\mu(\dd z).
\end{equation}
Since $T_\#\mu=\nu$, both marginals of $\sigma_\eps$ equal $\mu$. More precisely, \eqref{eqv:eq:marginal} gives $\int_Xk(x,z)\dd\mu(z)=\int_Xk(z,x)\dd\mu(z)=1$ for every $x\in X$.

\begin{lemma}[Overlap of nearby rows]
\label{end:lem:overlap}
There are $a,c>0$ such that, for sufficiently small $\eps$,
\begin{equation}\label{end:eq:overlap}
\int_Xk(x,z)k(x',z)\dd\mu(z)\ge c\ell^{-d}\qquad(x,x'\in X,\ |x-x'|\le a\ell).
\end{equation}
\end{lemma}
\begin{proof}
By \eqref{end:eq:height} and \eqref{end:eq:slopes}, there are fixed $b\in(0,1]$ and $c_1>0$ such that
\begin{equation}\label{m:display:087}
q(x,y)\ge c_1\ell^2\qquad(y\in Y\cap B(y_x^*,b\ell)).
\end{equation}
Indeed, $q(x,y)\ge H_x-Cb\ell^2$ on this relative ball, so it suffices to choose $b$ small enough. Set $E_x=S(Y\cap B(y_x^*,b\ell))$. By $T_\#\mu=\nu$ and \eqref{end:eq:volume}, $\mu(E_x)\ge c\ell^d$, while \eqref{m:display:087} gives $k(x,z)\ge c_1\ell^{-d}$ on $E_x$. 
Suppose $|x-x'|\le a\ell$, where $a\le1$. For $z\in E_x$ and $\xi\in[x,x']\subset X$,
\[
|T(z)-T(\xi)|\le|T(z)-y_x^*|+|y_x^*-T(x)|+|T(x)-T(\xi)|\le(b+R_0+\Lip(T))\ell.
\]
Integrating \eqref{end:eq:slopes} along $[x,x']$ therefore gives
\begin{equation}\label{m:display:088}
q(x',T(z))\ge q(x,T(z))-Ca\ell^2\ge(c_1-Ca)\ell^2.
\end{equation}
Choose $a$ so that $Ca\le c_1/2$. Both $k(x,z)$ and $k(x',z)$ are then bounded below by $c\ell^{-d}$ on $E_x$. Hence $\int_Xk(x,z)k(x',z)\dd\mu(z)\ge c\ell^{-2d}\mu(E_x)\ge c\ell^{-d}$.
\end{proof}

\subsubsection{A reversible equation with antisymmetric forcing}
\label{end:sec:transposition}

Decompose the potential errors as
\begin{equation}\label{end:eq:errors}
e(x)=r(x)+s(T(x)),\qquad \eta(x)=\tfrac12\bigl(r(x)-s(T(x))\bigr).
\end{equation}
Classical dual equality, \eqref{end:eq:diagonal}, and \eqref{eq:gauge} give
\begin{equation}\label{end:eq:emean}
e(x)=-q(x,T(x)),\qquad \|e\|_\infty\le C\ell^2,\qquad \int_X\eta\dd\mu=0.
\end{equation}
For the last identity, use $\int_Xs\circ T\dd\mu=\int_Ys\dd\nu=\int_Xr\dd\mu$. Since $r=e/2+\eta$ and $s\circ T=e/2-\eta$, it remains to bound $\eta$.

Set
\begin{equation}\label{end:eq:bregman}
\begin{aligned}
D(x,z)&=u_0(x)-u_0(z)-T(z)\cdot(x-z),\\
A(x,z)&=\tfrac12\bigl(D(x,z)-D(z,x)\bigr).
\end{aligned}
\end{equation}
Then $A(z,x)=-A(x,z)$ and
\begin{equation}\label{end:eq:cubic}
|A(x,z)|\le C|x-z|^3.
\end{equation}
Indeed, writing $\xi=x-z$, Taylor's formula along $[z,x]\subset X$ gives
\begin{equation}\label{m:display:089}
D(x,z)-D(z,x)=\int_0^1(1-2t)\xi^TD^2u_0(z+t\xi)\xi\dd t.
\end{equation}
Since $\int_0^1(1-2t)\dd t=0$, we may subtract $D^2u_0(z)$ inside the integral. The Lipschitz bound in \eqref{end:eq:classical} then yields \eqref{end:eq:cubic}.

Write $Q(x,z)=q(x,T(z))$. Using $v_0(T(z))=z\cdot T(z)-u_0(z)$ and \eqref{end:eq:errors}, we obtain
\begin{align}\label{end:eq:transpose}
Q(x,z)&=\eta(z)-\eta(x)-\tfrac12(e(x)+e(z))-D(x,z),\nonumber\\
Q(x,z)-Q(z,x)&=2\bigl(\eta(z)-\eta(x)-A(x,z)\bigr).
\end{align}
Define $J_\eps(x,z)=\ell^{-d}\omega_\eps(x,z)$, where
\begin{equation}\label{end:eq:divdiff}
\begin{gathered}
\omega_\eps(x,z)=\begin{cases}
\displaystyle\frac{Q(x,z)_+-Q(z,x)_+}{Q(x,z)-Q(z,x)},&Q(x,z)\ne Q(z,x),\\[5pt]
1,&Q(x,z)=Q(z,x)>0,\\
0,&Q(x,z)=Q(z,x)\le0.
\end{cases}
\end{gathered}
\end{equation}
These kernels are measurable and symmetric, with $0\le\omega_\eps\le1$. Since $k=Q_+/\eps$ and $\eps=\ell^{d+2}$, \eqref{end:eq:transpose} implies
\begin{equation}\label{end:eq:densitydifference}
k(x,z)-k(z,x)=2\ell^{-2}J_\eps(x,z)\bigl(\eta(z)-\eta(x)-A(x,z)\bigr).
\end{equation}
The row and column integrals of $k$ coincide, so integration in $z$ gives
\begin{equation}\label{end:eq:jequation}
\int_XJ_\eps(x,z)\bigl(\eta(z)-\eta(x)-A(x,z)\bigr)\dd\mu(z)=0\qquad(x\in X).
\end{equation}

\begin{lemma}[Conductance bounds]
\label{end:lem:conductance}
There are $c,C,R>0$ such that, for sufficiently small $\eps$ and all $x,z\in X$,
\begin{equation}\label{end:eq:jbound}
c\bigl(k(x,z)+k(z,x)\bigr)\le J_\eps(x,z)\le C\ell^{-d}\1_{\{|x-z|\le R\ell\}}.
\end{equation}
Consequently, $d_\eps(x):=\int_XJ_\eps(x,z)\dd\mu(z)$ satisfies $c\le d_\eps(x)\le C$ on $X$.
\end{lemma}
\begin{proof}
If $J_\eps(x,z)>0$, at least one of $Q(x,z),Q(z,x)$ is positive. By \eqref{end:eq:tube}, $|T(x)-T(z)|\le R_0\ell$, and hence $|x-z|\le\Lip(S)R_0\ell$. Together with $\omega_\eps\le1$, this proves the upper bound in \eqref{end:eq:jbound}.

For such pairs, \eqref{end:eq:nearvalues} places both slack values in $[-C\ell^2,C\ell^2]$. Their divided difference satisfies
\begin{equation}\label{m:display:090}
\omega_\eps(x,z)\ge\frac{Q(x,z)_++Q(z,x)_+}{2C\ell^2}.
\end{equation}
If both values are positive, the divided difference is one; if exactly one is positive, the absolute value of its denominator is at most $2C\ell^2$. If neither is positive, both sides vanish. Multiplying by $\ell^{-d}$ and using $Q_+=\ell^{d+2}k$ gives the lower bound in \eqref{end:eq:jbound}. Finally, integration yields $d_\eps(x)\ge c\int_X(k(x,z)+k(z,x))\dd\mu(z)=2c$ and $d_\eps(x)\le C\ell^{-d}\mu(B(x,R\ell))\le C$.
\end{proof}

Set $Z_\eps=\int_Xd_\eps\dd\mu$ and define
\begin{equation}\label{end:eq:markov}
\vartheta_\eps(\dd x)=Z_\eps^{-1}d_\eps(x)\mu(\dd x),\qquad p_\eps(x,z)=\frac{Z_\eps J_\eps(x,z)}{d_\eps(x)d_\eps(z)}.
\end{equation}
The measure $\vartheta_\eps$ is a probability with Lebesgue density bounded above and away from zero uniformly in $\eps$. Moreover, $p_\eps(x,z)=p_\eps(z,x)$ and $\int_Xp_\eps(x,z)\dd\vartheta_\eps(z)=1$. Thus $p_\eps$ is a reversible Markov transition density with invariant measure $\vartheta_\eps$.

For bounded measurable $v$, let $(\mathsf P_\eps v)(x)=\int_Xp_\eps(x,z)v(z)\dd\vartheta_\eps(z)$. Since $p_\eps(x,z)\vartheta_\eps(\dd z)=d_\eps(x)^{-1}J_\eps(x,z)\mu(\dd z)$, equation~\eqref{end:eq:jequation} becomes
\begin{equation}\label{end:eq:pequation}
(\mathsf P_\eps-I)\eta=f_\eps,\qquad f_\eps(x)=\int_Xp_\eps(x,z)A(x,z)\dd\vartheta_\eps(z).
\end{equation}
By \cref{end:lem:conductance}, $p_\eps\ge c(k+k^T)$ and $p_\eps\le C\ell^{-d}\1_{\{|x-z|\le R\ell\}}$, where $k^T(x,z)=k(z,x)$.

\subsubsection{Two-step ellipticity and cubic forcing}
\label{end:sec:twostep}

The transition density of $\mathsf P_\eps^2$ is
\begin{equation}\label{end:eq:twokernel}
K_\eps(x,w)=\int_Xp_\eps(x,z)p_\eps(z,w)\dd\vartheta_\eps(z).
\end{equation}
It is again symmetric and Markov: $K_\eps(x,w)=K_\eps(w,x)$ and $\int_XK_\eps(x,w)\dd\vartheta_\eps(w)=1$.

\begin{lemma}[Two-step minorization]
\label{end:lem:two}
There are $0<a\le b$ and $c,C>0$ such that, for sufficiently small $\eps$ and all $x,w\in X$,
\begin{equation}\label{end:eq:elliptickernel}
c\ell^{-d}\1_{\{|x-w|\le a\ell\}}\le K_\eps(x,w)\le C\ell^{-d}\1_{\{|x-w|\le b\ell\}}.
\end{equation}
Moreover, there is a measurable antisymmetric function $B_\eps$ with $|B_\eps|\le C\ell^3$ such that
\begin{equation}\label{end:eq:kequation}
\int_XK_\eps(x,w)\bigl(\eta(w)-\eta(x)-B_\eps(x,w)\bigr)\dd\vartheta_\eps(w)=0\qquad(x\in X).
\end{equation}
\end{lemma}
\begin{proof}
By symmetry and the preceding bounds, $p_\eps(z,w)=p_\eps(w,z)\ge ck(w,z)$, $p_\eps(x,z)\ge ck(x,z)$, and $\vartheta_\eps\ge c\mu$. Hence \cref{end:lem:overlap} gives
\begin{equation}\label{m:display:091}
K_\eps(x,w)\ge c\int_Xk(x,z)k(w,z)\dd\mu(z)\ge c\ell^{-d}\qquad(|x-w|\le a\ell).
\end{equation}
For the upper bound, $K_\eps(x,w)\le C\ell^{-d}\int_Xp_\eps(x,z)\dd\vartheta_\eps(z)=C\ell^{-d}$. A nonzero integrand requires $|x-z|\le R\ell$ and $|z-w|\le R\ell$, so $K_\eps(x,w)=0$ when $|x-w|>2R\ell$. Taking $b=2R$ and decreasing $a$ if necessary proves \eqref{end:eq:elliptickernel}.

When $K_\eps(x,w)>0$, set
\begin{equation}\label{end:eq:pathforce}
B_\eps(x,w)
=\frac{1}{K_\eps(x,w)}
\int_X p_\eps(x,z)p_\eps(z,w)
\bigl(A(x,z)+A(z,w)\bigr)\dd\vartheta_\eps(z),
\end{equation}
and set $B_\eps(x,w)=0$ otherwise. Symmetry of $p_\eps,K_\eps$ and antisymmetry of $A$ imply $B_\eps(w,x)=-B_\eps(x,w)$. On the support of the integrand, \eqref{end:eq:cubic} gives $|A(x,z)+A(z,w)|\le C\ell^3$. Since the weights in \eqref{end:eq:pathforce} have total mass $K_\eps(x,w)$, we obtain $|B_\eps(x,w)|\le C\ell^3$.

Fubini's theorem and the Markov property give
\[
\int_XK_\eps(x,w)B_\eps(x,w)\dd\vartheta_\eps(w)=\int_Xp_\eps(x,z)\bigl(A(x,z)+f_\eps(z)\bigr)\dd\vartheta_\eps(z)
=f_\eps(x)+(\mathsf P_\eps f_\eps)(x).
\]
On the other hand, applying $I+\mathsf P_\eps$ to \eqref{end:eq:pequation} yields $(\mathsf P_\eps^2-I)\eta=f_\eps+\mathsf P_\eps f_\eps$. These identities prove \eqref{end:eq:kequation}.
\end{proof}

\subsection{A uniform nonlocal Neumann estimate}
\label{end:sec:nonlocal}

We prove the analytic estimate needed for \eqref{end:eq:kequation}, uniformly in the interaction length. The reference measure may vary with that length and need not have a continuous density. The term ``Neumann'' refers to the confinement of transitions to the domain and conservation of constants; no exterior boundary data are prescribed.

\begin{theorem}[Uniform nonlocal $L^\infty$ estimate]\label{end:thm:nonlocal}
Let $D\subset\R^d$ be a convex body, and let $\vartheta$ be a probability measure on $D$ with Lebesgue density between two fixed positive constants. Fix $0<a\le b$ and $c,C>0$. There are $\delta_0,C_D>0$, depending only on $D$, the density bounds, and $a,b,c,C$, with the following property. For $0<\delta<\delta_0$, let $K_\delta:D\times D\to[0,\infty)$ be a measurable symmetric Markov density relative to $\vartheta$, that is, $K_\delta(x,y)=K_\delta(y,x)\ge0$ and $\int_DK_\delta(x,y)\dd\vartheta(y)=1$. Assume
\begin{equation}\label{end:eq:kernelhyp}
c\delta^{-d}\1_{\{|x-y|\le a\delta\}}\le K_\delta(x,y)\le C\delta^{-d}\1_{\{|x-y|\le b\delta\}} \qquad \mbox{a.e.}
\end{equation}
Let $F:D\times D\to\R$ be measurable and antisymmetric, with $|F(x,y)|\le A\delta$ a.e.\ on $\{K_\delta>0\}$ for some $A\ge0$. If $g\in L^2(\vartheta)$ solves
\begin{equation}\label{end:eq:generalequation}
\int_DK_\delta(x,y)\bigl(g(y)-g(x)-F(x,y)\bigr)\dd\vartheta(y)=0\qquad\text{for }\vartheta\text{-a.e. }x,
\end{equation}
then
\begin{equation}\label{end:eq:nonlocalbound}
\left\|g-\int_Dg\dd\vartheta\right\|_{L^\infty(\vartheta)}\le C_D A.
\end{equation}
The constants are independent of $\delta,g,F,A$ and of the choice of $\vartheta$ within the stated density bounds; in particular, $\vartheta$ may depend on $\delta$.
\end{theorem}

\subsubsection{Smoothing and an averaged Sobolev inequality}

Reduce $\delta_0$ so that $\delta_0\le1$ and
$2b\delta_0<r_D$, where $r_D$ is the radius in
\eqref{end:eq:volume} applied to $D$. Fix $0<\delta<\delta_0$ and write $K=K_\delta$, also using $K$ for the operator $(Kv)(x)=\int_DK(x,y)v(y)\dd\vartheta(y)$. Throughout this subsection, $\|\cdot\|_p$ denotes the $L^p(\vartheta)$ norm. Define
\begin{equation}\label{end:eq:energy}
\cE_\delta(v)=\frac1{2\delta^2}\iint_{D\times D}K(x,y)(v(y)-v(x))^2\dd\vartheta(x)\dd\vartheta(y).
\end{equation}
Symmetry and the Markov property give $\cE_\delta(v)\le2\delta^{-2}\|v\|_2^2$, so the energy is finite for every $v\in L^2(\vartheta)$ at fixed $\delta$. Choose a smooth cutoff $0\le\chi\le1$ equal to one on $\overline B_b$ and supported in $B_{2b}$, and set
\begin{equation}\label{end:eq:smoothing}
\mathcal S_\delta v(x)=\frac{\int_D\chi((x-y)/\delta)v(y)\dd\vartheta(y)}{\int_D\chi((x-y)/\delta)\dd\vartheta(y)}.
\end{equation}
By the density bounds and the relative-volume estimate \eqref{end:eq:volume}, applied to $D$, the denominator is comparable with $\delta^d$, uniformly for $x\in D$.

\begin{lemma}[Smoothing, Poincar\'e and averaged Sobolev estimates]\label{end:lem:smooth}
For every $v\in L^2(\vartheta)$,
\begin{align}\label{end:eq:smoothingbounds}
\|\mathcal S_\delta v\|_2&\le C\|v\|_2,\nonumber\\
\|v-\mathcal S_\delta v\|_2&\le C\delta\cE_\delta(v)^{1/2},\qquad \|\nabla\mathcal S_\delta v\|_2\le C\cE_\delta(v)^{1/2}.
\end{align}
Moreover,
\begin{equation}\label{end:eq:poincare}
\left\|v-\int_Dv\dd\vartheta\right\|_2\le C\cE_\delta(v)^{1/2}.
\end{equation}
With $\kappa=1+1/d>1$, every nonnegative $v\in L^2(\vartheta)$ satisfies
\begin{equation}\label{end:eq:averagedsobolev}
\|Kv\|_{2\kappa}\le C\bigl(\|v\|_2+\cE_\delta(v)^{1/2}\bigr).
\end{equation}
All constants are uniform in $\delta$, and in $\vartheta$ within the density bounds fixed in \cref{end:thm:nonlocal}.
\end{lemma}
\begin{proof}
We first compare differences over distances $2b\delta$ and $a\delta$. Choose an integer $n\ge2b/a$ and set $z_j=(1-j/n)x+(j/n)y$ for $0\le j\le n$. Convexity keeps these points in $D$, and telescoping gives
\begin{equation}\label{m:display:092}
|v(y)-v(x)|^2\le n\sum_{j=0}^{n-1}|v(z_{j+1})-v(z_j)|^2.
\end{equation}
For each $j$, the linear map $(x,y)\mapsto(z_j,z_{j+1})$ has determinant $n^{-d}$ in $\R^{2d}$ and sends pairs at distance at most $2b\delta$ to pairs at distance at most $a\delta$. Changing variables and enlarging the image of each map therefore gives
\begin{equation}\label{end:eq:energycompare}
\iint_{\substack{x,y\in D\\|x-y|\le2b\delta}}|v(y)-v(x)|^2\dd x\dd y\le n^{d+2}\iint_{\substack{x,y\in D\\|x-y|\le a\delta}}|v(y)-v(x)|^2\dd x\dd y.
\end{equation}
The density bounds and the lower kernel bound imply that the left-hand side is at most $C\delta^{d+2}\cE_\delta(v)$.

Jensen's inequality and Fubini's theorem give the $L^2$ bound for $\mathcal S_\delta$. Jensen's inequality also bounds $|v(x)-\mathcal S_\delta v(x)|^2$ by $C\delta^{-d}\int_{D\cap B(x,2b\delta)}|v(y)-v(x)|^2\dd\vartheta(y)$. Integration and \eqref{end:eq:energycompare} yield the approximation estimate. The derivative in $x$ of the normalized kernel in \eqref{end:eq:smoothing} has zero $\vartheta$-integral and is bounded by $C\delta^{-d-1}\1_{\{|x-y|\le2b\delta\}}$. Subtracting $v(x)$ inside the differentiated integral and applying Cauchy--Schwarz gives
\begin{equation}\label{m:display:093}
|\nabla\mathcal S_\delta v(x)|^2\le C\delta^{-d-2}\int_{D\cap B(x,2b\delta)}|v(y)-v(x)|^2\dd\vartheta(y).
\end{equation}
Integration proves the gradient estimate. In particular, no derivative of the density of $\vartheta$ is needed.

Apply Poincar\'e's inequality on $D$ to $\mathcal S_\delta v$. Its $\vartheta$-mean differs from that of $v$ by at most $\|\mathcal S_\delta v-v\|_2$, so \eqref{end:eq:smoothingbounds} yields \eqref{end:eq:poincare}. Finally, the upper kernel bound and the choice of $\chi$ give $Kv\le C\mathcal S_\delta v$ for $v\ge0$. The Sobolev embedding $H^1(D)\hookrightarrow L^{2\kappa}(D)$ and \eqref{end:eq:smoothingbounds} then imply \eqref{end:eq:averagedsobolev}. The fixed convex domain admits both inequalities, and passage between Lebesgue and $\vartheta$ norms changes only their constants.
\end{proof}

The gain of integrability in \eqref{end:eq:averagedsobolev} applies to $Kv$, not to $v$ itself. The equation supplies the comparison needed to use this estimate in the iteration below.

\subsubsection{Energy and finite-step Moser iteration}

\begin{proof}[Proof of \cref{end:thm:nonlocal}]
Symmetry of $K$ and antisymmetry of $F$ give the weak identity
\begin{equation}\label{end:eq:weak}
\iint K\,\Delta g\,\Delta\phi\dd(\vartheta\otimes\vartheta)=\iint KF\,\Delta\phi\dd(\vartheta\otimes\vartheta),\qquad \Delta v(x,y)=v(y)-v(x),
\end{equation}
for every $\phi\in L^2(\vartheta)$, with double integrals over $D\times D$. If $A=0$, testing with $g$ and using \eqref{end:eq:poincare} shows that $g$ is constant. Otherwise, divide $g,F$ by $A$ and subtract the $\vartheta$-mean of $g$. We may thus assume $\int_Dg\dd\vartheta=0$ and $|F|\le\delta$ on $\{K>0\}$.

Testing \eqref{end:eq:weak} with $g$ and applying Cauchy--Schwarz, together with $\iint K\dd(\vartheta\otimes\vartheta)=1$, gives $\cE_\delta(g)\le C$. By \eqref{end:eq:poincare}, $\|g\|_2\le C$. Rewrite the equation as
\begin{equation}\label{eq:nonlocalforcing}
g=Kg-f_F,\qquad f_F(x)=\int_DK(x,y)F(x,y)\dd\vartheta(y),\qquad |f_F|\le\delta.
\end{equation}
The upper kernel bound and the Markov property imply
$\|Kg\|_\infty\le C\delta^{-d/2}\|g\|_2.$
Thus $g$ is bounded for each fixed $\delta$. This preliminary bound is not uniform, but makes the following power tests admissible.

Set $U=1+g_+\ge1$. For $p\ge2$, take $\phi=U^{p-1}-1$ and write $V=U^{p/2}$. For each pair $(x,y)$,
\begin{equation}\label{m:display:095}
\Delta g\,\Delta\phi\ge\Delta U\,\Delta(U^{p-1})\ge\frac{4(p-1)}{p^2}|\Delta V|^2,
\end{equation}
and
\begin{equation}\label{m:display:096}
0\le\frac{\Delta\phi}{\Delta g}\le(p-1)\bigl(U(x)^{p-2}+U(y)^{p-2}\bigr),
\end{equation}
where the ratio is set to zero when $\Delta g=0$. Young's inequality yields
\begin{equation}\label{m:display:097}
|F\,\Delta\phi|\le\tfrac12\Delta g\,\Delta\phi+\tfrac12F^2\frac{\Delta\phi}{\Delta g}.
\end{equation}
Insert this estimate into \eqref{end:eq:weak} and absorb the first term. Using \eqref{m:display:095}--\eqref{m:display:096}, $|F|\le\delta$, and the Markov property, we obtain
\begin{equation}\label{end:eq:powerenergy}
\cE_\delta(V)\le Cp^2\int_DU^{p-2}\dd\vartheta\le Cp^2\|U\|_p^p.
\end{equation}

By \eqref{eq:nonlocalforcing} and Jensen's inequality for the positive part,
\begin{equation}\label{m:display:098}
U\le KU+\delta\le(1+\delta)KU,
\end{equation}
where the last inequality uses $KU\ge1$. Raising to the power $p/2$ and applying Jensen's inequality again gives
\begin{equation}\label{m:display:099}
V\le(1+\delta)^{p/2}KV.
\end{equation}
Apply \eqref{end:eq:averagedsobolev} to $V$ and use \eqref{end:eq:powerenergy}. Since $\|V\|_2=\|U\|_p^{p/2}$, taking the power $2/p$ yields
\begin{equation}\label{end:eq:moserstep}
\|U\|_{p\kappa}\le(1+\delta)(Cp)^{2/p}\|U\|_p.
\end{equation}

The factor $1+\delta$ prevents an infinite iteration with a uniform bound. Instead, reduce $\delta_0$ so that $\delta_0<e^{-2}$, set $p_j=2\kappa^j$, and stop at the first index $J$ for which $p_J\ge\log(1/\delta)$. Then
\begin{equation}\label{m:display:100}
\prod_{j=0}^{J-1}(Cp_j)^{2/p_j}\le C,\qquad (1+\delta)^J\le C.
\end{equation}
The first bound follows from $\sum_jp_j^{-1}\log(Cp_j)<\infty$, and the second from $J=O(1+\log\log(1/\delta))$. Since $\|U\|_2\le C$, the finite iteration gives $\|U\|_{p_J}\le C$. For the final step, use the upper kernel bound directly:
\begin{equation}\label{end:eq:laststep}
KU(x)\le\left(\int_DK(x,y)U(y)^{p_J}\dd\vartheta(y)\right)^{1/p_J}\le(C\delta^{-d})^{1/p_J}\|U\|_{p_J}\le C,
\end{equation}
where the last inequality uses
\begin{equation}\label{m:display:101}
\delta^{-d/p_J}=\exp\!\left(\frac{d\log(1/\delta)}{p_J}\right)\le e^d.
\end{equation}
Thus \eqref{m:display:098} gives $U\le C$. Repeating the argument for $-g$ with forcing $-F$ bounds the negative part. Undoing the normalization proves \eqref{end:eq:nonlocalbound}.
\end{proof}

\subsection{Proof of the potential convergence theorem}
\label{end:sec:completion}

\begin{proof}[Proof of \cref{end:thm:main}]
\textit{Uniform upper bounds.} Set $\bar\eta_\eps=\int_X\eta\dd\vartheta_\eps$. Apply \cref{end:thm:nonlocal} to \eqref{end:eq:kequation}, with $g=\eta$, $\delta=\ell$, reference measure $\vartheta_\eps$, and forcing $F=B_\eps$. The density bounds for $\vartheta_\eps$ and \cref{end:lem:two} verify the kernel hypotheses, while $\eta\in C(X)\subset L^2(\vartheta_\eps)$. Since $|B_\eps|\le C\ell^3$, we may take $A=C\ell^2$ and obtain
\begin{equation}\label{end:eq:etabound}
\|\eta-\bar\eta_\eps\|_{L^\infty(\vartheta_\eps)}\le C\ell^2.
\end{equation}
The positive density of $\vartheta_\eps$ and continuity of $\eta$ extend this bound to all of $X$. Since $\int_X\eta\dd\mu=0$ by \eqref{end:eq:emean},
\begin{equation}\label{m:display:102}
|\bar\eta_\eps|=\left|\int_X(\bar\eta_\eps-\eta)\dd\mu\right|\le C\ell^2.
\end{equation}
Hence $\|\eta\|_{L^\infty(X)}\le C\ell^2$. 
By \eqref{end:eq:errors}, \eqref{end:eq:emean}, and $T(X)=Y$,
\[
\|u_\eps-u_0\|_{L^\infty(X)}+\|v_\eps-v_0\|_{L^\infty(Y)}\le \|e\|_{L^\infty(X)}+2\|\eta\|_{L^\infty(X)}\le C\ell^2.
\]

\textit{Sharpness in all $L^p$ norms.} Recall the transport-cost excess and regularization penalty
\begin{equation}\label{m:display:103}
C_\eps=\frac12\int|x-y|^2\dd\pi_\eps-\OT\ge0,\qquad P_\eps=\frac\eps2\int h_\eps^2\dd(\mu\otimes\nu).
\end{equation}
Cauchy--Schwarz in each row, $\int_Yh_\eps(x,y)\dd\nu(y)=1$, and $\nu(S_x)\le C\ell^d$ yield
\begin{equation}\label{end:eq:penaltylower}
P_\eps\ge\frac\eps2\int_X\frac1{\nu(S_x)}\dd\mu(x)\ge c\ell^2.
\end{equation}
Since $q_\eps=\eps h_\eps$ wherever $h_\eps>0$, we have $\int q_\eps\dd\pi_\eps=2P_\eps$. Moreover, dual equality for the unregularized OT problem and the marginal constraints give $\int(u_0(x)+v_0(y)-x\cdot y)\dd\pi_\eps=C_\eps$. Integrating \eqref{eq:uniform-error-detachment} therefore yields
\begin{equation}\label{end:eq:meanidentity}
\int_Xr\dd\mu+\int_Ys\dd\nu=-C_\eps-2P_\eps.
\end{equation}
Under the symmetric normalization \eqref{eq:gauge}, the two means are equal, so
\begin{equation}\label{end:eq:individualmean}
\int_Xr\dd\mu=\int_Ys\dd\nu=-\tfrac12C_\eps-P_\eps\le-c\ell^2.
\end{equation}
For a probability measure, every $L^p$ norm dominates the absolute mean and is bounded by the supremum norm. Together with the uniform upper bounds, this proves \eqref{end:eq:main}, with constants independent of $p$.

\textit{Point normalization.} To impose $u_\eps(x_*)=u_0(x_*)$, replace $r$ by $r-r(x_*)$ and $s$ by $s+r(x_*)$. The shift is $O(\ell^2)$ by \eqref{end:eq:main}, so both upper bounds persist. %
\end{proof}

\section{Sharpness and further consequences}
\label{sec:consequences}

Throughout this section, we assume \cref{ass:geometrydata} and the symmetric normalization \eqref{eq:gauge}, and retain $\ell=\eps^{1/(d+2)}$, $r_\eps=u_\eps-u_0$, and $s_\eps=v_\eps-v_0$. The results of \cref{thm:geometry,bmo:thm:main,end:thm:main} establish \cref{thm:sharp} and the BMO bounds \eqref{eq:bmosharp-main}. We now complete the proofs of \cref{cor:consequences,cor:conditional-concentration}.

\paragraph{Conditional moments and convergence of the plans.}
The support estimate \eqref{eq:sharpfiber} gives the upper bounds in \eqref{eq:rowwise-moment-rates}. For the lower bounds, \eqref{end:eq:height} gives $h_\eps\le C\ell^{-d}$, so the density upper bound implies
\begin{equation}\label{m:display:104}
\pi_\eps^x(B(T(x),a\ell))=\int_{Y\cap B(T(x),a\ell)}h_\eps(x,y)\dd\nu(y)\le Ca^d\qquad(x\in X).
\end{equation}
Choose $a>0$ so that $Ca^d\le1/2$. At least half the conditional mass then lies at distance at least $a\ell$ from $T(x)$. Thus, for $1\le p<\infty$,
\begin{equation}\label{eq:row-moment-anticoncentration}
\left(\int_Y|y-T(x)|^p\pi_\eps^x(\dd y)\right)^{1/p}\ge2^{-1/p}a\ell\ge\tfrac a2\ell.
\end{equation}
The same mass bound yields the estimate for $p=\infty$. Interchanging the marginals proves the reverse estimates, completing \eqref{eq:rowwise-moment-rates}. For finite $p$, raising these estimates to the $p$th power and integrating over the conditioning point gives \eqref{eq:displacement}.

Set $\pi_0=(\id,T)_\#\mu$. The pushforward of $\pi_\eps$ under $(x,y)\mapsto((x,y),(x,T(x)))$ couples $\pi_\eps$ to $\pi_0$ with displacement $|y-T(x)|\le C\ell$. Hence $W_p(\pi_\eps,\pi_0)\le C\ell$ for every $1\le p\le\infty$. Conversely, any coupling to $\pi_0$ moves $(x,y)$ by at least its distance from $\operatorname{graph}T$. By \eqref{m:display:084} and the conditional lower bounds \eqref{eq:row-moment-anticoncentration},
\begin{equation}\label{eq:plan-wasserstein-lower}
W_p(\pi_\eps,\pi_0)\ge\frac{\bigl\|(x,y)\mapsto|y-T(x)|\bigr\|_{L^p(\pi_\eps)}}{\sqrt{1+\Lip(T)^2}}\ge c\ell,\qquad 1\le p\le\infty.
\end{equation}
Here the norms are interpreted as essential suprema when $p=\infty$. This proves \cref{cor:conditional-concentration}.

\paragraph{H\"older upper bounds.}
For distinct $x,x'\in X$, write $\Delta=|x-x'|$. The uniform potential and gradient estimates, together with the Hessian bounds \eqref{eq:hessians} and \cref{lem:classicalfinite}, give
\begin{equation}\label{m:display:105}
\begin{aligned}
|r_\eps(x)-r_\eps(x')|\le C\min\{\ell^2,\ell\Delta\},\qquad
|\nabla r_\eps(x)-\nabla r_\eps(x')|\le C\min\{\ell,\Delta\}.
\end{aligned}
\end{equation}
For $0<\beta\le1$ and $0<\gamma<1$, dividing by $\Delta^\beta$ and $\Delta^\gamma$, respectively, and distinguishing $\Delta\le\ell$ from $\Delta\ge\ell$ yields the required seminorm bounds. Combining these with the uniform potential and gradient bounds proves the upper estimates in \eqref{eq:holder}; the cases $\beta=0$ and $\gamma=0$ follow directly from the latter bounds. The argument for $s_\eps$ is identical.

\paragraph{H\"older lower bounds and failure of $C^{1,1}$ convergence.}
Fix $x\in\partial X$. By \eqref{eq:centerboundary} and $T(\partial X)=\partial Y$, we have $|\nabla r_\eps(x)|\ge c_0\ell$. The uniform interior cone condition supplies a unit vector $e$ with $|\nabla r_\eps(x)\cdot e|\ge c_1\ell$ and $x+te\in X$ for $0\le t\le t_0$, where $c_1,t_0>0$ are uniform. Since $\Lip(\nabla r_\eps)\le C$, choosing $t=b\ell$ with $b>0$ fixed sufficiently small gives
\begin{equation}\label{m:display:106}
|r_\eps(x+te)-r_\eps(x)|\ge t|\nabla r_\eps(x)\cdot e|-Ct^2\ge c\ell^2.
\end{equation}
For $0<\beta\le1$, division by $t^\beta$ yields $[r_\eps]_{C^{0,\beta}(X)}\ge c_\beta\ell^{2-\beta}$. The case $\beta=0$ follows from \eqref{eq:allp}.

For $0<\gamma\le1$, set $J_\gamma=[\nabla r_\eps]_{C^{0,\gamma}(X)}$. Taylor's formula on $X\cap B(x,R)$ has remainder bounded by $J_\gamma R^{1+\gamma}$. Comparing with the affine part by \eqref{bmo:eq:affine-estimates} gives
\begin{equation}\label{m:display:107}
[r_\eps]_{\BMO(X)}\ge cR|\nabla r_\eps(x)|-CJ_\gamma R^{1+\gamma},\qquad 0<R\le\diam X.
\end{equation}
By \eqref{m:display:080}, $[r_\eps]_{\BMO(X)}\le C_0\ell^2$. Choose a fixed $A>0$ with $cc_0A>2C_0$ and take $R=A\ell$, which is admissible for small $\eps$. Rearranging \eqref{m:display:107} yields $J_\gamma\ge c_\gamma\ell^{1-\gamma}$. For $\gamma=0$, use the uniform gradient lower bound. Interchanging the marginals proves the corresponding estimates for $s_\eps$ and completes \eqref{eq:holder}. Taking $\gamma=1$ shows that the Lipschitz seminorms of both gradient errors remain bounded away from zero, so neither potential converges in $C^{1,1}$.

\paragraph{Potential error on the Brenier graph.}
Equations~\eqref{end:eq:emean} and \eqref{bmo:eq:means}, together with $T_\#\mu=\nu$, give
$\|r_\eps+s_\eps\circ T\|_{L^\infty(X)}\le C\ell^2$ and $
\int_X(r_\eps+s_\eps\circ T)\dd\mu=-C_\eps-2P_\eps.
$ 
Since $C_\eps\ge0$ and $P_\eps\ge c\ell^2$, the absolute mean gives the matching lower bound. This proves \eqref{eq:graphsum} and completes \cref{cor:consequences}.

\paragraph{Declarations.} AGS was partially supported by the RSME--Fundaci\'on BBVA Jos\'e Luis Rubio de Francia Start-up Grant, awarded in the framework of the Jos\'e Luis Rubio de Francia Prize. MN was partially supported by NSF Grants DMS-2407074 and DMS-2606731. The authors declare no competing interests. No data were created or analyzed. AI was used in the preparation of this manuscript. Any remaining errors are our own.
\printbibliography[heading=bibintoc]
\end{document}